\documentclass[a4paper,12pt]{amsart}

\pdfoutput=1

\usepackage[T1]{fontenc}
\usepackage[utf8]{inputenc}
\usepackage[english]{babel}

\usepackage[a4paper,top=3cm,bottom=2cm,left=2.5cm,right=2.5cm,marginparwidth=1.75cm]{geometry}

\usepackage{mathtools} 
\usepackage{newpxtext} 
\usepackage{eulerpx} 
\usepackage{eucal} 

\usepackage{bm}

\DeclareSymbolFont{bighat}{U}{zeus}{m}{n}
\DeclareMathAccent{\bighat}{\mathalpha}{bighat}{222}

\usepackage{framed}

\usepackage[normalem]{ulem}

\usepackage{xcolor}

\newcommand\soutc{\bgroup\markoverwith{\textcolor{cyan}{\rule[0.4ex]{1pt}{1.7pt}}}\ULon}

\usepackage{tikz}
\usetikzlibrary{cd}

\usepackage{pifont}
\makeatletter
\renewcommand{\@fnsymbol}[1]{%
\ifcase#1\hbox{} \or \ding{118} \or \ding{168} \or \ding{74} \or \ding{116} \or \ding{110} \or \ding{94} \or \ding{115} \or \ding{108} \or \ding{68} \or \ding{117}
\else\@ctrerr\fi\relax}
\makeatother

\DeclareSymbolFont{sfoperators}{T1}{cmss}{m}{n}
\DeclareSymbolFontAlphabet{\mathsf}{sfoperators}
\makeatletter
\renewcommand{\operator@font}{\mathgroup\symsfoperators}
\makeatother

\DeclareMathOperator{\lpow}{LPOW}
\DeclareMathOperator{\mpow}{MPOW}
\DeclareMathOperator{\ppow}{PPOW}
\DeclareMathOperator{\pow}{POW}
\DeclareMathOperator{\Char}{char}
\DeclareMathOperator{\Def}{def}
\DeclareMathOperator{\Sp}{Spec}

\newcommand{\N}{\mathbb{N}}
\newcommand{\Z}{\mathbb{Z}}
\newcommand{\Q}{\mathbb{Q}}
\newcommand{\F}{\mathbb{F}}
\newcommand{\Zm}{\mathbb{Z}^{+}}
\newcommand{\Spf}{\Sp^{\circ}}
\newcommand{\pZ}{p^{\mathbb{Z}}}

\newcommand{\CA}{\mathcal{A}}
\newcommand{\CB}{\mathcal{B}}
\newcommand{\CH}{\mathcal{H}}
\newcommand{\CP}{\mathcal{P}}

\newcommand{\Fa}{\mathfrak{a}}
\newcommand{\Fb}{\mathfrak{b}}
\newcommand{\Fc}{\mathfrak{c}}
\newcommand{\Fg}{\mathfrak{g}}
\newcommand{\Fh}{\mathfrak{h}}
\newcommand{\Fm}{\mathfrak{m}}
\newcommand{\Fn}{\mathfrak{n}}
\newcommand{\Fp}{\mathfrak{p}}
\newcommand{\Fq}{\mathfrak{q}}

\bmdefine{\VV}{\mathcal{V}}
\bmdefine{\XX}{\mathcal{X}}
\bmdefine{\ZZ}{\mathcal{Z}}

\DeclareMathOperator{\TS}{T\kern-2.27bpS}

\numberwithin{equation}{section}

\usepackage[pdfencoding=unicode,colorlinks=true,allcolors=blue,bookmarks=true,pdfpagelayout=OneColumn,pdfstartview={FitH},backref=page,linktocpage=false]{hyperref}

\usepackage{bookmark}

\usepackage[alphabetic,backrefs,nobysame,msc-links]{amsrefs}

\makeatletter
\def\print@backrefs#1{\space\SentenceSpace[cited on page(s) \csname br@#1\endcsname]}
\makeatother

\usepackage{enumitem}
\setlist[enumerate,1]{label=\texttt{\textup{\alph*.}}}

\usepackage{aliascnt}

\usepackage{cleveref}

\makeatletter
\crefdefaultlabelformat{\leavevmode\sw@slant#2{\upshape#1}#3}
\makeatother

\creflabelformat{equation}{#2(\textup{#1})#3}

\newaliascnt{formulac}{equation}
\aliascntresetthe{formulac}
\crefname{formulac}{formula}{formulas}
\creflabelformat{formulac}{#2\textup{(#1)}#3}
\makeatletter

\def\endformula{\eqno \hbox{\@eqnnum}$$\@ignoretrue}
\makeatother

\crefname{subsection}{subsection}{subsections}

\theoremstyle{plain}
\newtheorem{dummy}{Dummy}[section]

\newtheorem{theoremzz}[dummy]{Theorem}
\newenvironment{theorem}[1][]{\begin{theoremzz}[#1]\begin{leftbar}}{\end{leftbar}\end{theoremzz}}
\newtheorem{corollaryzz}[dummy]{Corollary}
\newenvironment{corollary}[1][]{\begin{corollaryzz}[#1]\begin{leftbar}}{\end{leftbar}\end{corollaryzz}}
\newtheorem{propositionzz}[dummy]{Proposition}
\newenvironment{proposition}[1][]{\begin{propositionzz}[#1]\begin{leftbar}}{\end{leftbar}\end{propositionzz}}
\newtheorem{lemmazz}[dummy]{Lemma}
\newenvironment{lemma}[1][]{\begin{lemmazz}[#1]\begin{leftbar}}{\end{leftbar}\end{lemmazz}}

\theoremstyle{definition}
\newtheorem{definitionzz}[dummy]{Definition}
\newenvironment{definition}[1][]{\begin{definitionzz}[#1]\begin{leftbar}}{\end{leftbar}\end{definitionzz}}
\newtheorem{examplezz}[dummy]{Example}
\newenvironment{example}[1][]{\begin{examplezz}[#1]\begin{leftbar}}{\end{leftbar}\end{examplezz}}
\newtheorem{remarkzz}[dummy]{Remark}
\newenvironment{remark}[1][]{\begin{remarkzz}[#1]\begin{leftbar}}{\end{leftbar}\end{remarkzz}}

\crefname{theoremzz}{theorem}{theorems}
\crefname{corollaryzz}{corollary}{corollaries}
\crefname{propositionzz}{proposition}{propositions}
\crefname{lemmazz}{lemma}{lemmas}
\crefname{definitionzz}{definition}{definitions}
\crefname{examplezz}{example}{examples}
\crefname{remarkzz}{remark}{remarks}

\newlist{enumtheorem}{enumerate}{1}
\setlist[enumtheorem]{label=\texttt{\textup{\alph*.}},ref=\thetheoremzz\texttt{\textup{\alph*}}}
\crefalias{enumtheoremi}{theorem}

\newlist{enumproposition}{enumerate}{1}
\setlist[enumproposition]{label=\texttt{\textup{\alph*.}},ref=\thepropositionzz\texttt{\textup{\alph*}}}
\crefalias{enumpropositioni}{proposition}

\newlist{enumlemma}{enumerate}{1}
\setlist[enumlemma]{label=\texttt{\textup{\alph*.}},ref=\thelemmazz\texttt{\textup{\alph*}}}
\crefalias{enumlemmai}{lemma}

\makeatletter
\newcounter{subcreftmpcnt}%
\newcommand\alphsubformat[1]{\texttt{\textup{\alph{#1}}}}
\newcommand\subcref[2][\alphsubformat]{%
\ifcsname r@#2@cref\endcsname
\cref@getcounter {#2}{\mylabel}%
\setcounter{subcreftmpcnt}{\mylabel}%
\alphsubformat{subcreftmpcnt}%
\else ?? \fi}
\makeatother

\BibSpec{book}{%
+{}  {\PrintPrimary}                {transition}
+{,} { \textit}                     {title}
+{.} { }                            {part}
+{:} { \textit}                     {subtitle}
+{,} { \PrintEdition}               {edition}
+{}  { \PrintEditorsB}              {editor}
+{,} { \PrintTranslatorsC}          {translator}
+{,} { \PrintContributions}         {contribution}
+{,} { }                            {series}
+{,} { \voltext}                    {volume}
+{,} { }                            {publisher}
+{,} { }                            {organization}
+{,} { }                            {address}
+{,} { \PrintDateB}                 {date}
+{,} { }                            {status}
+{.} { \PrintDOI}                   {doi}
+{}  { \parenthesize}               {language}
+{}  { \PrintTranslation}           {translation}
+{;} { \PrintReprint}               {reprint}
+{.} { }                            {note}
+{.} {}                             {transition}
+{.}  {\SentenceSpace \PrintReviews} {review}
}

\usepackage[doipre={}]{uri}

\begin{document}

\title[Definability of integers and interpretability in a class of polynomial rings]{Interpretability and uniform definability of integers, and undecidability of reduced indecomposable polynomial rings}

\author{Marco Barone}
\address{Departamento de Matemática\\Universidade Federal de Pernambuco\\Avenida Jornalista Aníbal Fernandes, S/N - Cidade Universitária\\Recife/PE - Brasil - 50740-560}
\email[M.~Barone]{marco@dmat.ufpe.br}

\author{Nicolás Caro}
\address{Departamento de Matemática\\Universidade Federal de Pernambuco\\Avenida Jornalista Aníbal Fernandes, S/N - Cidade Universitária\\Recife/PE - Brasil - 50740-560}
\email[N.~Caro]{jorge.caro@dmat.ufpe.br}

\author{Eudes Naziazeno}
\address{Departamento de Matemática\\Universidade Federal de Pernambuco\\Avenida Jornalista Aníbal Fernandes, S/N - Cidade Universitária\\Recife/PE - Brasil - 50740-560}
\email[E.~Naziazeno]{eudes@dmat.ufpe.br}

\subjclass[2010]{03B10,13B25,13F99,13L05,16U99.}

\begin{abstract}

We prove first-order definability of the prime subring inside polynomial rings, whose coefficient rings are (commutative unital) reduced and indecomposable. This is achieved by means of a uniform formula in the language of rings with signature $(0,1,+,\cdot)$. In the characteristic zero case, the claim implies that the full theory is undecidable, for rings of the referred type; in this direction, we also provide a separate proof of the undecidability of these rings that works uniformly in any characteristic. These definability and undecidability assertions extend a series of results by Raphael Robinson (1951), holding for certain polynomial integral domains, to a more general class. Finally, we show that the rational integers are interpretable in these rings, even in positive characteristic.

\end{abstract}

\maketitle

\tableofcontents

\section{Introduction}

\noindent Over more than 60 years, the problem of defining rational integers inside a ring has been object of extensive investigation (for an overview, we refer the reader to the surveys \cites{Koenigsmann2014,PheidasZ2008,Poonen2008,Shlapentokh2011}). Much attention has been drawn onto Diophantine definability, for this would yield a counterpart result about other versions of Hilbert's tenth problem (see \cite{Matijasevic1970}). More specifically, Diophantine definability implies the undecidability of polynomial equations over $\Z$.

In a similar vein, first-order (not necessarily Diophantine) definability of integers in a characteristic zero ring is known to imply that the full first-order theory of such a ring is undecidable. For instance, Julia Robinson showed that $\Z$ is first-order definable in $\Q$ (\cite{RobinsonJ1949}). Concerning negative results, it was recently proved (\cite{AschenbrennerKNS2018}*{Lemma 4.7}) that the direct product of two infinite finitely generated rings is not bi-interpretable with $\Z$; the proof of this result can be mimicked to obtain, for example, that $\Z$ is not definable in $\Z\times\Z$.

The same questions arise within the class of polynomial rings over integral domains. Raphael Robinson (\cite{RobinsonR1951}*{\S 4d}) proved the undecidability of polynomial integral domains. Jan Denef in \cites{Denef1978,Denef1979} proves that, given an integral domain $R$ of characteristic zero (resp.~characteristic $p$), the problem of solvability in $R[T]$ of polynomials with coefficients in $\Z[T]$ (resp.~$(\Z/p\Z)[T]$) is undecidable. Furthermore, Thanases Pheidas and Karim Zahidi in \cite{PheidasZ1999} work with the language of the rings augmented by a symbol for the nonconstant polynomials, proving undecidability of the positive existential theory of polynomial rings over integral domains. Recently, Javier U\-tre\-ras proved interpretability of integers in polynomial rings over GCD domains, in a modified language (\cite{Utreras2019}).

However, except for the case of finitely generated $\Z$-algebras (\cite{AschenbrennerKNS2018}*{Corollary 2.19 and Subsection 6.3}) we have no knowledge of any attempt to extend definability and undecidability results outside the class of integral domains, partly due to the consistent use of field extensions of the quotient field of these rings throughout the results mentioned. In this paper, we work with polynomial rings $S=R[x]$ and formulate a criterion for the definability of the prime subring of $S$, that is, the smallest subring of $S$ (denoted here by $\ZZ_{S}$). In the characteristic zero case, $\ZZ_{S}$ is exactly $\Z$, and in positive characteristic it coincides with some quotient $\Z/n\Z$. Besides definability of $\ZZ_{S}$, we also prove undecidability of the full theory of such rings.

We put aside the assumption that $R$ be an integral domain, and explore a wider range of coefficient rings, which is in fact a natural class to which to extend the results, namely, the class of reduced indecomposable (commutative unital) rings (\Cref{lpowx=powximplies}). In general, any Noetherian reduced ring can be written out as a finite product of such rings (\cite{Cohn2003}*{Proposition 4.5.4}), so we may consider these rings as the basic bricks for building up an important class of objects in commutative algebra, corresponding to the notion of connected components of reduced schemes in algebraic geometry.

This work is divided as follows:

In \Cref{definitions}, we establish standard definitions and notation from Logic and Algebra that are going to be used throughout the paper, and we discuss some basic properties.

In \Cref{sectionlpow} we explore first-order definability of sets of powers, by introducing the concept of \emph{logical powers}, that is, a first-order property that coincides with the property of being a positive power of a given element of a ring, under some special conditions on both the element and the ring, mainly focusing on the case of polynomial rings in one variable.

In \Cref{sectionredindec}, we investigate such special conditions, and study the class of reduced indecomposable rings, proving several of its algebraic properties; we also provide examples of such rings that are not integral domains, both Noetherian and non-Noetherian.

In \Cref{sectionlpowR[x]} we use the theory developed in \Cref{sectionlpow} and \Cref{sectionredindec} to construct four special definable sets of polynomials with coefficients in a reduced indecomposable ring, which are crucially used in \Cref{sectionZdefin} in the proofs of the main results (by using explicit definitions for sets of powers of a fixed element).

In \Cref{sectionZdefin} we first prove the undecidability of the full theory of $R[x]$, whenever $R$ is a reduced indecomposable ring, by extending the scope of a technique firstly presented by Raphael Robinson in \cite{RobinsonR1951}*{\S\S 4b,4c}. Afterwards, we present a general criterion to define sets of exponents of powers of suitable elements. We specialize this criterion to reduced indecomposable polynomial rings, in two different versions, corresponding to two different subclasses of such polynomial rings.

The first version provides a uniform formula that ensures the definability of the prime subring, upon the condition that the nonzero integers are invertible. This condition is satisfied by all polynomial rings over a field or over reduced indecomposable rings of positive characteristic; the second one no longer relies on this condition, and it also provides a uniform formula, which works for polynomial rings over reduced indecomposable nonfields of characteristic zero. Afterwards, we gather the two formulas previously obtained into a single uniform formula defining the prime subring of $S=R[x]$, for any reduced indecomposable (commutative unital) ring $R$.

We end \Cref{sectionZdefin} by showing how the technique defined to extract exponents from sets of powers can be exploited to construct two-dimensional interpretations of the rings of the class considered, in two different ways covering, respectively, the case when the coefficient ring is a nonfield and that in which it is a field of characteristic zero. Finally, when the coefficient ring is a field of positive characteristic, we still provide a two-dimensional interpretation by means of a separate technique, involving properties of the set of linear polynomials.

This paper ends with \Cref{sectionAppendix}, a complementary collection of several properties of algebraic and logical nature involving the concepts defined throughout the work, as well as examples (be they revisited or novel) illustrating the variety of the objects attained by our results and counterexamples testing the limits of our hypotheses.

All our results and proofs are developed in the framework of Zermelo--Fraenkel (\textsf{ZF}) set theory; in particular, they do not depend on \textsf{AC} or any choice principle\footnote{However, some interesting issues concerning choice principles arise in \Cref{SpecBPI}, \Cref{C(X|B),B^I} and \Cref{nonstronglocal}.}\!.

\subsection*{Acknowledgements}

\noindent We would like to express our sincere thanks and appreciation to Thomas W.~Scanlon, Alexandra Shlapentokh, and Carlos Videla, for their kindness and inspiring advice. We are also indebted to Remy van Dobben de Bruyn and William F.~Sawin (from Math Overflow) and Robin Denis Arthan (Math Stack Exchange) for their help with some questions we raised on the websites mentioned. The second author is supported by FACEPE Grant APQ-0892-1.01/14.

\section{Preliminary definitions and notation}\label{definitions}

\noindent In this section we recall some basic notions from ring theory which will be used throughout this work (see \cite{Hungerford1980} for a background). We also discuss some logical issues concerning the axioms for reduced and/or indecomposable rings, and concerning the notion of ``integers'' in a given ring, as well as its definability. More specifically, we distinguish between zero and positive characteristic.

Except for \Cref{Feixistico}, all rings considered are commutative, unital and nonzero. Except in a few cases where emphasis is required, we denote the additive unit of a ring $S$ by $0$ instead of $0_{S}$, and similarly we denote by $1$ the multiplicative unit of $S$ (instead of $1_{S}$). Note that a ring is nonzero precisely when $1\neq 0$.

We will work in the first-order theory in the language of rings, with signature $(+,\cdot,0,1)$. Unless the dependency on parameters is explicitly mentioned, by ``definable'' we mean ``definable without parameters''.

For the sake of brevity and notational convenience, whenever a subset $A$ of a ring $S$ (or, more generally, a property $\CP$) is definable by a formula, say $\psi(\cdot)$, we will write ``$t\in A$'' (or, more generally, that ``$\CP$ holds'') instead of ``$\psi(t)$'' in subsequent formulas; likewise, for two-variable formulas expressing binary relations $\psi(\cdot,\cdot)$ which correspond to algebraic properties, we abbreviate by using classical notation (e.g.~``$s\mid t$'' for divisibility).

Let $S$ be a ring. An element $a\in S$ is said to be \textbf{nilpotent} if $a^{n}=0$ for some $n\geq 1$, and \textbf{idempotent} if $a^{2}=a$; in the latter case, the element $1-a$ is idempotent as well. The ring $S$ is said to be \textbf{reduced} if its only nilpotent element is zero, and \textbf{indecomposable}\footnote{Also referred to, in the literature, as \textbf{directly irreducible}. Indecomposable rings are equivalently (and more customarily) defined as those not isomorphic to the direct product of two nonzero rings.}\! if its only idempotent elements are $0$ and $1$. An element $a\in S$ is said to be \textbf{regular} if, whenever $ab=ac$, with $b,c\in S$, it follows that $b=c$; otherwise, it is said to be a \textbf{zerodivisor}. Notice that invertible elements are always regular. The multiplicative group of invertible elements of $S$ (also called \textbf{units} of $S$) is denoted by $S^{*}$. An \textbf{irreducible} element of $S$ is a nonzero, noninvertible element that cannot be written as a product of two nonunits. Finally, an element $p$ of $S$ is \textbf{prime} if it is nonzero and noninvertible, and whenever $p$ divides a product, it divides some of the factors.

For a ring $R$, we denote the polynomial ring in one indeterminate $x$ with coefficients in $R$ by $R[x]$, and we refer to the elements of $R$, that is, polynomials of degree zero, as the \textbf{constant polynomials}, or the \textbf{constants} of this larger ring (such ``constants'' should not be confused with the symbols of constants of the language of rings). Given $f\in R[x]$, we denote its coefficient of degree $i$ by $f_{i}\in R$. Finally, we will always make clear when we need to distinguish between the element $f\in R[x]$ and its associated polynomial function $f\colon R\to R$.

\subsection{Remarks on local rings}\label{ourlocal}

\noindent We say that a ring $S$ is \textbf{local} if $a+1$ is a unit for every nonunit $a$ of $S$; for example, any polynomial ring $R[x]$ is nonlocal (take $a=x$).

If $S$ is a local ring, then the set of nonunits of $S$ is closed under sums, and consequently it forms an ideal in $S$: indeed, if $b,c\notin S^{*}$, then for any $z\in S$ the element $a=bz-1$ satisfies $a+1=bz\notin S^{*}$, so necessarily $a\in S^{*}$. Since $-cz\notin S^{*}$, it follows that $a\neq -cz$, which amounts to saying that $(b+c)z\neq 1$. As $z$ is arbitrary, this proves that $b+c$ is a nonunit.

Conversely, if the set $\Fm$ of nonunits of a ring $S$ forms an ideal, then $S$ is local, because for any $a\in\Fm$ we have $(a+1)-a=1\notin\Fm$, so necessarily $a+1\notin\Fm$, that is $a+1\in S^{*}$.

Notice that our definition of ``local ring'' (as well as the equivalent characterization just proven: ``nonunits form an ideal'') differs from the standard definition used in commutative algebra and algebraic geometry, namely: a ring is local if it has a unique maximal ideal. We would like to stress that only our definition of ``local ring'' is used throughout the paper to prove the main results: we refrain from using the standard definition of local ring, because such a notion involves maximal ideals, and it is well-known that the existence of such ideals, in any nonzero commutative unital ring, is equivalent to the axiom of choice (\cite{Hodges1979}). In particular, our main results hold unconditionally on \textsf{ZF} and does not require assuming \textsf{AC}\footnote{See \Cref{nonstronglocal} for a thorough discussion concerning the relationship between these notions of ``locality'', which involves an equivalence to \textsf{AC}.}\!.

\subsection{On the theory of reduced/indecomposable rings}\label{Th-red-ind}

\noindent The existence of an idempotent element other than $0$ and $1$ in a ring is clearly a first-order predicate, so that the theory of indecomposable rings is finitely axiomatizable.

As a matter of fact, the same happens with reducedness, even though nilpotency cannot be expressed as a one-variable first-order formula\footnote{See \cite{Hodges1993}*{Exercise 8.5.1} for an example of a ring whose nilradical is not definable.}\!. Indeed, observe that if $a$ is a nonzero nilpotent element of a ring and $n\geq 2$ is its nilpotency index (i.e., the least positive integer such that $a^{n}=0$), then $a^{n-1}$ is a nonzero nilpotent element with nilpotency index $2$. Therefore a ring is reduced if and only if it contains no nonzero element whose square is zero, and this is obviously a first-order predicate.

Clearly, all the remaining ring-theoretic properties described at the beginning of the section, as well as \emph{our notion} of local ring, are first-order definable in the language of rings.

\subsection{The prime subring and its definability}\label{primesubring}

\noindent Let $S$ be a ring. The \textbf{prime subring} of $S$, denoted by $\ZZ_{S}$, is defined to be the smallest subring of $S$. It is not hard to show that $\ZZ_{S}\subseteq S$ is additively generated by $1_{S}$, and it is also the image of the (unique) ring homomorphism $j\colon\Z\to S$. The \textbf{characteristic} of $S$, denoted by $\Char(S)$, is defined to be the unique natural number $n$ such that $\ker j=n\Z$. Thus, $\ZZ_{S}$ is isomorphic to $\Z$ if the characteristic of $S$ is zero, and it is isomorphic to the ring $\Z/n\Z$ of integers modulo $n$ if $\Char(S)=n>0$.

This notion of ``prime subring'' clearly has nothing to do, and should not be confused, with the notion of ``prime element'' mentioned at the beginning of the section. When no ambiguity arises, we denote the prime subring $\ZZ_{S}$ of a ring $S$ simply by $\ZZ$, and we may sometimes refer informally to the elements of $\ZZ$ as the ``integers''. For $m\in\Z$ we will denote $m\cdot1_{S}\in S$, with a slight abuse of notation, simply by $m$, writing ``$m\in S$''. Likewise, we will informally refer to the elements of $\ZZ^{+}=j(\Zm)$ as the ``positive integers''. Notice that $\ZZ^{+}=\ZZ$ when $\Char(S)>0$.

The main goal of this work is to prove the definability of $\ZZ\subseteq S$ for $S$ belonging to a wide class of rings, namely, that of reduced indecomposable polynomial rings. As mentioned in the Introduction, the case of characteristic zero (when $\ZZ=\Z$) implies undecidability of the full theory of the corresponding ring. Regarding positive characteristic, if $\Char(S)=n>0$, then $\ZZ=\Z/n\Z$ is trivially definable, via the formula
\[\gamma_{n}(t)\colon\ \ \bigvee_{i=1}^{n}\ (\,t=\underbrace{1+\cdots+1}_{i\textnormal{ times}}\,)\,,\]
which depends on $n$ in a cumbersome way. Since we are able to construct a \emph{uniform} formula that covers all reduced indecomposable polynomial rings, regardless of the characteristic, we have in particular that, for $\Char(S)=n>0$, our formula does not depend on $n$. Obviously, in positive characteristic, our definability result does not imply undecidability of the full theory, so we resort to an alternative method (see \Cref{undecidable}) to prove undecidability in this case.

\section{A first-order approach to the definability of sets of powers}\label{sectionlpow}

\noindent Let $S$ be a ring. For an element $p\in S$, let $\pow(p)$ denote the set of positive powers of $p$. As will be clearer in \Cref{sectionZdefin}, the first clue for definability of $\ZZ$ comes from the idea of ``logically'' identifying positive integers with the exponents of a fixed element, reducing the task to defining sets of powers of a fixed element of the ring. This has led to the search for a first-order definable notion that approximates that of ``power''.

\subsection{Logical powers: definition and basic properties}

\noindent In this subsection we introduce an intuitive notion of positive power of an element $p\in S$ as a multiple of $p$ whose only divisors, up to units, are also multiples of $p$, together with an additional property which, in the case of polynomial rings and under special conditions, also guarantees monicity (as a monomial in $p$); this condition is encapsulated by \Cref{lpowformula} below. An analogous approach is considered in \cite{RobinsonR1951}*{p.~145}, where it is shown that the same property is satisfied precisely by the nonnegative powers of $p$, whenever $p$ is a prime element and $S$ is an integral domain (see item \subcref{lpowgeneral-d} of \Cref{lpowgeneral} for a slight generalization). We will explore our notion in a more general context where, for suitable conditions on $p$ (\Cref{TandUPV} and \Cref{automorphU}), the set $\pow(p)$ is first-order definable using $p$ as a parameter.

\begin{definition}\label{deflpow}

Let $S$ be a ring. Given $p\in S$, we define the set $\lpow(p)$ of \textbf{logical powers} of $p$ as the set of elements $f\in S$ satisfying:

\begin{itemize}

\item $p$ divides $f$;

\item $p-1$ divides $f-1$;

\item every divisor of $f$ is a unit or a multiple of $p$.

\end{itemize}

\end{definition}

\noindent Observe that $\lpow(p)$ is defined by the one-variable formula $\psi(\cdot,p)$, where $\psi$ is given by
\begin{formula}\label{lpowformula}
\psi(f,s)\colon\ \ s\mid f\ \wedge\ s-1\mid f-1\ \wedge\ \forall g\,[\,g\mid f\rightarrow(\,g\mid 1\ \vee\ s\mid g\,)\,]\,.
\end{formula}
In what follows, we explore the similarities between $\lpow(p)$ (a first-order definable set) and $\pow(p)$ (a set that we want to be first-order definable), in order to justify the expression ``logical powers''. Unfortunately, in the general case the definition of $\lpow(p)$ fails badly in conveying the concept of ``genuine powers'':

\begin{example}\label{ghnotinlpowgh}

If $g,h\in S$ are noninvertible and $h$ is regular, then $gh\notin\lpow(gh)$. In fact, we have that $h$ divides $gh$, but $h$ is neither a unit nor a multiple of $gh$ (if $h=qgh$, then canceling $h$ would imply that $g$ is a unit).

\end{example}

\noindent Another instance in which the two definitions clash is the following: on the one hand, $0\in\pow(p)$ if and only if $p$ is nilpotent; on the other hand, the following result characterizes whether the zero element is a logical power in nonlocal rings, a wide class of rings that includes all polynomial rings (see \Cref{ourlocal}):

\begin{proposition}\label{0inlpow}

Let $S$ be a nonlocal ring. For any $p\in S$, the following are equivalent:

\begin{enumproposition}

\item $0\in\lpow(p)$;

\item Both $p$ and $p-1$ are units;

\item $\lpow(p)=S$.

\end{enumproposition}

\end{proposition}

\begin{proof}\leavevmode

\begin{enumerate}[labelindent=23pt,labelwidth=\widthof{\texttt{a}~$\Rightarrow$~\texttt{b}},itemindent=0em,leftmargin=!]

\item[\texttt{(a}~$\Rightarrow$~\texttt{b):}] We have that $p-1$ divides $0-1=-1$, so $p-1$ is a unit. As $S$ is not local, there exists $s\in S$ such that $s$ and $s+1$ are nonunits. Since $s$ and $s+1$ trivially divide $0$ and $0\in\lpow(p)$, they must be multiples of $p$. Therefore $p$ divides $(s+1)-s=1$.

\item[\texttt{(b}~$\Rightarrow$~\texttt{c):}] If both $p$ and $p-1$ are units, then any element $t\in S$ obviously belongs to $\lpow(p)$, for $t$, as all its divisors, is a multiple of $p$, whilst $p-1$ divides $t-1$.

\item[\texttt{(c}~$\Rightarrow$~\texttt{a):}] Obvious.\qedhere

\end{enumerate}

\end{proof}

\noindent Notice that the hypothesis in \Cref{0inlpow} is only used in the proof of \texttt{a}~$\Rightarrow$~\texttt{b} to prove that $p$ is a unit, whereas \texttt{b}~$\Rightarrow$~\texttt{c}~$\Rightarrow$~\texttt{a}~$\Rightarrow$~``$p-1$ is a unit'' holds for any ring.

\subsection{Consequences of \texorpdfstring{\except{toc}{$\bm{{\operatorname{\textbf{\textsf{LPOW}}}(x)=\operatorname{\textbf{\textsf{POW}}}(x)}}$ in $\bm{R[x]}$}\for{toc}{$\lpow(x)=\pow(x)$ in $R[x]$}}{LPOW(x)=POW(x) in R[x]}}

\noindent The findings from the previous subsection suggest that our attempt at identifying the sets $\pow(p)$ by $\lpow(p)$ could be more successful if we avoid nilpotent and reducible elements. As a matter of fact, under certain hypotheses the two sets coincide, producing a first-order definition of the powers of some types of elements. Before proceeding in this direction, we list some general properties concerning logical powers that will be used in the sequel. At this point, one notation is worth introducing: given two elements $f,p$ of a ring, we say that $f$ is \textbf{infinitely divisible by} $p$ if $f$ is a multiple of arbitrarily large powers of $p$ (equivalently, a multiple of all positive powers of $p$).

\begin{proposition}\label{lpowgeneral}

Let $S$ be a ring, and let $p\in S$.

\begin{enumproposition}

\item\label{lpowgeneral-a} Any element $f$ of $\lpow(p)$ is either infinitely divisible by $p$, or an element of the form $up^{n}$, for some $n\geq 1$ and some unit $u$ satisfying $p-1\mid u-1$. In particular, if $u=1$, then $f\in\pow(p)$.

\item\label{lpowgeneral-b} If $f\in\lpow(p)$ and $u$ is a unit such that $p-1$ divides $u-1$, then $uf\in\lpow(p)$.

\item\label{lpowgeneral-c} If $p$ is either invertible or irreducible, then $p\in\lpow(p)$.

\item\label{lpowgeneral-d} If $p$ is regular and prime, then $\pow(p)\subseteq\lpow(p)$.

\end{enumproposition}

\end{proposition}

\begin{proof}\leavevmode

\begin{enumerate}

\item If $f$ is not infinitely divisible by $p$, let $n\geq 1$ be the greatest exponent such that $p^{n}\mid f$, so that $f=up^{n}$ for some $u$ not divisible by $p$. Since $u$ divides $f$ and $f\in\lpow(p)$, $u$ must be a unit. Finally, we have $f-1=up^{n}-1=u\cdot(p^{n}-1)+u-1$, and since both $f-1$ and $p^{n}-1$ are multiples of $p-1$, so is $u-1$.

\item Obviously $p$ divides $uf$. Since $p-1$ divides both $f-1$ and $u-1$, it follows that $p-1$ divides $u\cdot(f-1)+u-1=uf-1$. Finally, if $g$ divides $uf$, then $g$ divides $u^{-1}\cdot(uf)=f$. Since $f\in\lpow(p)$, we conclude that $g$ is a unit or a multiple of $p$.

\item It suffices to observe that every divisor of $p$ would be either invertible or an associate of $p$ (hence a multiple of $p$), for the other properties are trivially satisfied.

\item Let $n\geq 1$. Obviously $p\mid p^{n}$ and $p-1\mid p^{n}-1$, and if $g$ is a divisor of $p^{n}$, say $p^{n}=gh$, then $p^{n+1}$ cannot divide $h$ (otherwise we would have, by canceling, that $p$ divides $1$, which contradicts the primality of $p$). Thus, the largest $k$ with $p^{k}$ dividing $h$ must satisfy $k\leq n$. After canceling we get $p^{n-k}=g\bighat{h}$, with $\bighat{h}$ not a multiple of $p$. If $k=n$, then $g$ is invertible; otherwise, $p$ divides $g\bighat{h}$, so necessarily $p$ divides $g$ because $p$ is prime.\qedhere

\end{enumerate}

\end{proof}

\noindent In what follows we will examine the case $S=R[x]$, in order to draw some consequences from the equality $\lpow(x)=\pow(x)$:

\begin{proposition}\label{lpowx=powximplies}

Let $R$ be a ring and consider $R[x]$, the polynomial ring in one variable over $R$. If $x\in\lpow(x)$, then $x$ is irreducible. If in addition one of the inclusions $\lpow(x)\subseteq\pow(x)$ or $\pow(x)\subseteq\lpow(x)$ holds, then $R$ is reduced.

\end{proposition}

\begin{proof}

We always have that $x$ is nonzero and noninvertible. Since $x$ is regular, every divisor of it will also be regular, and so if $x\in\lpow(x)$, then by using the contrapositive of \Cref{ghnotinlpowgh} we can conclude that $x$ is irreducible.

Let $a\in R$ with $a^{n}=0$ for some $n\geq 1$. We want to prove that if, in addition, $\lpow(x)\subseteq\pow(x)$ or $\pow(x)\subseteq\lpow(x)$, then $a=0$, obtaining in this way that $R$ is reduced. Set $u=1-a\cdot(x-1)$. Note that $u$ divides $1-a^{n}\cdot(x-1)^{n}=1$, that is, $u$ is invertible, and also that $x-1$ clearly divides $u-1$. Consequently, by item \subcref{lpowgeneral-b} of \Cref{lpowgeneral} we have $ux\in\lpow(x)$.

If $\lpow(x)\subseteq\pow(x)$, then $ux=x^{m}$ for some $m\geq 1$, which forces to have $m=1$ and $u=1$, and so $a=0$. Moreover, observe that $x-a$ is not invertible and divides $x^{n}-a^{n}=x^{n}$, and therefore, if $\pow(x)\subseteq\lpow(x)$ (in this case the condition $x\in\lpow(x)$ is superfluous), then $x-a$ must be a multiple of $x$, so again $a=0$.
\end{proof}

\noindent Thus, for a ring $R$, in order to have $\lpow(x)=\pow(x)$, it is necessary that $R$ be reduced and the polynomial $x$ be irreducible in $R[x]$. Later we will see (\Cref{lpowx=powx}) that these conditions are also sufficient, and in the course of the reasoning we will show (see \Cref{bivalente}) that irreducibility of the polynomial $x$ in $R[x]$ is equivalent to indecomposability of $R$.

\section{Reduced and indecomposable rings and some of their algebraic properties}\label{sectionredindec}

\noindent In this section we study some algebraic properties of reduced and/or indecomposable rings. We prove, among other things, that just as integral domains, reduced indecomposable rings have characteristic zero or prime, and we exhibit examples of such rings that are not integral domains. Finally, we prove that constant polynomial functions in reduced indecomposable rings can only come from constant polynomials.

\subsection{Expressing reducedness and indecomposability of rings in terms of the corresponding polynomial rings}

\begin{lemma}\label{basic}

Let $R$ be a ring. Let $f,g\in R[x]$ be nonzero polynomials, and denote their degrees by $d$ and $m$, respectively.

\begin{enumlemma}

\item\label{basic-a} Let $h\in R[x]$, and let $k=\deg(h)$. If $f=gh$, with $d<m+k$, then for all integer $i$ with $1\leq i\leq m+k-d$, the $i$-th power of the leading coefficient of $g$ annihilates the $i$ coefficients of $h$ of highest degrees, that is, $h_{k},\ldots,h_{k-i+1}$.

\item\label{basic-b} Given $h\in R[x]$ and $r\geq 0$, if $x^{r}$ divides $gh$, then $x^{r}$ divides $g_{0}^{r}h$. Moreover, $x^{r}=gh$ implies $g_{0}^{r}=g_{0}^{r+1}h_{r}$.

\item\label{basic-c} If $g$ divides $f$ and $m>d$, then $f$ is annihilated by a power of $g_{m}$. More specifically, if $f=gh$ and $k=\deg(h)$, then $g_{m}^{k+1}f=0$. Consequently, if $R$ is reduced and $f$ is regular, then $g\mid f$ implies $m\leq d$. In particular, whenever the coefficient ring is reduced, divisors of regular constant elements are themselves constants.

\item\label{basic-d} Suppose $R$ is reduced and indecomposable and $g$ divides $f$. If the leading coefficient of $f$ is a unit, then that of $g$ must be a unit too.

\end{enumlemma}

\end{lemma}

\begin{proof}\leavevmode

\begin{enumerate}

\item Write $f=gh=(g_{m}x^{m}+\cdots+g_{0})(h_{k}x^{k}+\cdots+h_{0})$, with $g_{m},h_{k}\neq 0$. We proceed by induction to prove that multiplying by $g_{m}^{i}$ annihilates $h_{k},\ldots,h_{k-i+1}$ for all $i=1,\ldots,m+k-d$. For $i=1$, the claim follows from $g_{m}h_{k}=f_{m+k}=0$ (recall that $d<m+k$). Suppose the claim holds for $i$ and suppose $i+1\leq m+k-d$. In this case we have $d<m+k-i$, and therefore $0=f_{m+k-i}=g_{m}h_{k-i}+(g_{m-1}h_{k-i+1}+\cdots+g_{m-i}h_{k})$. By induction hypothesis, the second term of this sum is annihilated by $g_{m}^{i}$, as all coefficient of $h$ appearing in it are. Therefore, multiplying by $g_{m}^{i}$, one gets $g_{m}^{i+1}h_{k-i}=0$, and since $g_{m}^{i+1}$ also annihilates $h_{k},\ldots,h_{k-i+1}$, this completes the induction.

\item The result is obvious for $r=0$. For $r>0$, as $gh$ is a multiple of $x^{r}$, we have that all its coefficients in degrees $0,\ldots,r-1$ vanish, so we may apply a specular reasoning to that used in the previous item and get $0=(gh)_{0}=g_{0}h_{0}$ and, if $r>1,0=(gh)_{1}=g_{0}h_{1}+g_{1}h_{0}$, from which $g_{0}^{2}h_{1}=0$ and thus $g_{0}^{2}$ annihilates $h_{0}$ and $h_{1}$. By proceeding analogously until $r-1$ we obtain that $g_{0}^{r}$ annihilates $h_{0},\ldots,h_{r-1}$ and therefore all coefficients of $g_{0}^{r}h$ vanish until degree $r-1$, which yields the first claim. In the special case where $x^{r}=gh$ we also have $1=(gh)_{r}=g_{0}h_{r}+(g_{1}h_{r-1}+\cdots+g_{r}h_{0})$; after multiplying by $g_{0}^{r}$, the second term of the right side vanishes, giving $g_{0}^{r}=g_{0}^{r+1}h_{r}$.

\item If $d<m$ and $h$ is as in item \subcref{basic-a}, we can apply such result to $i=k+1\leq m+k-d$ and get that $g_m^{k+1}$ annihilates $h_{k},\ldots,h_{0}$ and, consequently, annihilates $h$. Hence $g_{m}^{k+1}f=(g_{m}^{k+1}h)g=0$. For the second assertion, observe that if we had $m>d$, then $f$ would be annihilated by a power of a nonzero constant (the leading coefficient of $g$), which is also nonzero in a reduced ring. Therefore
$f$ would be a zerodivisor, contradicting the hypothesis. The last statement follows immediately.

\item Let $f=gh$, with $h\in R[x]$. In the case $d=m+k$ we have $f_{d}=g_{m}h_{k}$ and therefore, if $f_{d}$ is invertible, then $g_{m}$ invertible as well. In the case $d<m+k$, letting $i=m+k-d$, we may write the leading coefficient of $f$ as $u=f_{d}=f_{m+k-i}=g_mh_{k-i}+L$, where $L=g_{m-1}h_{k-i+1}+\cdots+g_{m-i}h_{k}$. The item \subcref{basic-a} above may be applied to the index $i$ (because $1\leq i\leq m+k-d$), implying that $h_{k},\ldots,h_{k-i+1}$ are annihilated by $g_{m}^{i}$, and therefore $g_{m}^{i}L=0$.

If $u$ is a unit, so is $u^{i}=(g_{m}h_{k-i}+L)^{i}=g_{m}^{i}h_{k-i}^{i}+LM$, for some $M\in R$. Multiplying by $v=u^{-i}$, we have $1=vg_{m}^{i}h_{k-i}^{i}+vLM$. By setting $e=vg_{m}^{i}h_{k-i}^{i}$ and $e'=vLM$ we have written $e+e'=1$, and since $g_{m}^{i}L=0$, it follows that $ee'=0$. Therefore $e$ and $e'$ are idempotent, and since $R$ is indecomposable, one of them must be $1$. We also have $g_{m}^{i}\neq 0$ because $R$ is reduced, and since $g_{m}^{i}e'=(vM)\cdot(g_{m}^{i}L)=0$, we conclude that $e'$ is a zerodivisor and, consequently, $e'\neq 1$. This forces $1=e=vg_{m}^{i}h_{k-i}^{i}$, and thus $g_{m}$ is a unit.\qedhere

\end{enumerate}

\end{proof}

\begin{proposition}\label{redunits}

For a ring $R$, the following conditions are equivalent:

\begin{enumproposition}

\item $R[x]$ is reduced;

\item $R$ is reduced;

\item $R[x]^{*}=R^{*}$.

\end{enumproposition}

\end{proposition}

\begin{proof}

The implication \texttt{a}~$\Rightarrow$~\texttt{b} is obvious. For \texttt{b}~$\Rightarrow$~\texttt{c}, note that units are precisely the divisors of $1$, which is a regular constant element, and apply the last assertion of \Cref{basic-c}. Finally, if $R[x]^{*}=R^{*}$ and $f\in R[x]$ satisfies $f^{m}=0$, with $m\geq 2$, then $(1+xf^{m-1})(1-xf^{m-1})=1$ implies $1+xf^{m-1}\in R[x]^{*}\subseteq R$, so necessarily $f^{m-1}=0$. Iterating this reasoning we conclude that $f=0$, proving that $R[x]$ is reduced.
\end{proof}

\noindent The next result relates indecomposability of a ring $R$ to a property about its polynomial ring $R[x]$:

\begin{proposition}\label{bivalente}

A ring $R$ is indecomposable if and only if the polynomial $x\in R[x]$ is irreducible.

\end{proposition}

\begin{proof}

Obviously $x$ is nonzero and noninvertible. Suppose that $R$ is indecomposable, and assume $x=gh$, with $g,h\in R[x]$; we want to show that either $g$ or $h$ is a unit. Set $e=g_{0}h_{1}$ and $e'=g_{1}h_{0}$. We have $e+e'=g_{0}h_{1}+g_{1}h_{0}=(gh)_{1}=(x)_{1}=1$. Furthermore, by the last part of \Cref{basic-b} with $r=1$ we have $g_{0}^{2}h_{1}=g_{0}$, so $e^{2}=(g_{0}h_{1})^{2}=(g_{0}^{2}h_{1})h_{1}=g_{0}h_{1}=e$, and therefore $e$, being idempotent, must be $0$ or $1$ (in a similar way one can show that $e'$ is idempotent). If $e=1$, then $g_{0}\in R^{*}$; since $g_{0}h_{0}=(gh)_{0}=(x)_{0}=0$, it follows that $h_{0}=0$, so $x$ divides $h$, and dividing out the equality $x=gh$ by the regular element $x$, we get that $g$ is a unit. If $e=0$, then $e'=1$, and proceeding analogously we conclude that $h\in R[x]^{*}$.

For the converse, since the only invertible idempotent $f$ in a ring is $f=f^{2}f^{-1}=ff^{-1}=1$, if $e\in R$ is a nontrivial idempotent (that is, other than $0$ or $1$), then $1-e$ is also a nontrivial idempotent, and therefore both $e$ and $1-e$ are nonunits. Thus, the polynomials $g=ex+(1-e)$ and $h=(1-e)x+e$ have noninvertible constant term, so they cannot be units in $R[x]$. Since $x=gh$, we conclude that $x$ is reducible.
\end{proof}

\begin{remark}\label{SpecBPI}

Notice that the argument above proves that, if $x$ has any nontrivial factorization, then it has one as a product of two linear polynomials. Furthermore, by putting together \Cref{redunits,bivalente}, we obtain a characterization of reduced indecomposable rings in terms of a property of the polynomials $1$ and $x$ in $R[x]$: that they both be not a product of two positive degree polynomials. For those acquainted with algebraic geometry, we recall the special meaning that indecomposability has in terms of the topology of the corresponding Zariski affine scheme: a ring $R$ is indecomposable if and only if its prime spectrum $\Sp(R)$ is connected\footnotemark\!. The reader may feel free to check \Cref{reduindecomp} for more equivalent definitions of indecomposability and/or reducedness.

\end{remark}

\footnotetext{For a proof of this equivalence, see \cite{Eisenbud1995}*{Exercise 2.25}. The proof relies heavily upon the Boolean prime ideal theorem (\textsf{BPI}); see \cite{HowardR1998}*{Form 14}.}

\noindent From the very definition of polynomials and their multiplication, it follows that $0$ is the only polynomial infinitely divisible by $x$. This will be used in the proof of the following result, which shares the same spirit of \Cref{redunits}, but concerning indecomposability:

\begin{proposition}\label{RindecR[x]}

For any ring $R$, a polynomial $e\in R[x]$ is idempotent if and only if $e$ is constant and idempotent in $R$. In particular, $R$ is indecomposable if and only if $R[x]$ is indecomposable.

\end{proposition}

\begin{proof}

Let $e\in R[x]$ be idempotent. Writing $e=e_{0}+gx$, with $g\in R[x]$, the equality $e=e^{2}$ becomes $e_{0}+gx=e_{0}^{2}+2e_{0}gx+g^{2}x^{2}$, yielding $e_{0}=e_{0}^{2}$, and in particular $(1-2e_{0})gx=(gx)^{2}$. Since $(1-2e_{0})^{2}=1$, it follows that $(1-2e_{0})gx=[(1-2e_{0})gx]^{2}$. Thus $(1-2e_{0})gx=[(1-2e_{0})g]^{n}x^{n}$ for all $n\geq 1$, that is, $(1-2e_{0})gx$ is infinitely divisible by $x$, and so necessarily $(1-2e_{0})gx=0$. Since $(1-2e_{0})x$ is regular, it follows that $g=0$, so $e=e_{0}$ is idempotent in $R$.
\end{proof}

\noindent From \Cref{redunits,RindecR[x]} we obtain the following characterization of reducedness/indecomposability for polynomial rings in an arbitrary set of indeterminates:

\begin{proposition}\label{R[Xfat]}

Let $R$ be a ring and let $\XX$ be a set of indeterminates over $R$. If $S=R[\XX]$, then $S$ is reduced \textup{(}resp.~indecomposable\textup{)} if and only if $R$ is reduced \textup{(}resp.~indecomposable\textup{)}.

\end{proposition}

\begin{proof}

Obviously, if $S$ reduced (resp.~indecomposable), then the subring $R$ of $S$ is also reduced (resp.~indecomposable). Conversely, assume that $R$ is reduced (resp.~indecomposable). Given $f\in S$, there exists a finite subset $\XX'$ of $\XX$ such that $f\in S_{0}$, where $S_{0}=R[\XX']$. \Cref{redunits} (resp.~\Cref{RindecR[x]}), together with induction, shows that $S_{0}$ is reduced (resp.~indecomposable) as well, and therefore $f$ nilpotent (resp.~idempotent) implies $f=0$ (resp.~$f=0$ or $1$), which shows that $S$ is reduced (resp.~indecomposable).
\end{proof}

\noindent Notice that, although our class of rings of the form $S=R[x]$ was initially described in terms of properties of $R$, we now have instead an intrinsic characterization of the same class, regardless of the presentation of $S\cong R'[\XX]$ (that is, independent of the subring $R'$ and the set $\XX$ of indeterminates over $R'$). Consequently, provided that a given ring is polynomial (in any set of variables), all other conditions for membership in our class are first-order axiomatizable in the language of rings (\Cref{Th-red-ind}), without extra symbols for the coefficient ring or the indeterminates. The main results of this paper, that is, definability of the prime subring (\Cref{finalboss}), undecidability of the full theory (\Cref{undecidable}) and interpretability of rational integers (\Cref{interpretability}) are therefore true for ``polynomial reduced indecomposable rings''.

Given an element of a ring that is zero or a unit, it trivially has a positive power dividing the previous corresponding power (actually, this happens for \emph{every} positive power of it). For reduced indecomposable rings, the converse holds. This basic result will be used repeatedly, and we prove it below:

\begin{proposition}\label{cm+1dividescm}

For any reduced indecomposable ring $R$ and any $c\in R$, we have:

\begin{enumproposition}

\item\label{cm+1dividescm-a} If $c^{m+1}$ divides $c^{m}$ for some $m\geq 0$, then $c\in\{0\}\cup R^{*}$.

\item\label{cm+1dividescm-b} If $c\notin\{0\}\cup R^{*}$, then all nonnegative powers of $c$ are pairwise distinct.

\item\label{cm+1dividescm-c} If $R$ is finite, then $R$ is a field.

\end{enumproposition}

\end{proposition}

\begin{proof}\leavevmode

\begin{enumerate}

\item If $c^{m}=c^{m+1}d$, then $(cd)^{m}=c^{m}d^{m}=(c^{m+1}d)d^{m}=(cd)^{m+1}$, hence $(cd)^{m}=(cd)^{m+1}=\cdots=(cd)^{2m}$. Therefore $(cd)^{m}$ is idempotent, hence it equals $1$ or $0$ (because $R$ is indecomposable). If $(cd)^{m}=1$, then $c\in R^{*}$. Otherwise, since $R$ is reduced, it follows that $cd=0$, which implies $c^{m}=c^{m+1}d=c^{m}\cdot(cd)=0$, and therefore $c=0$ (again by reducedness of $R$).

\item If two nonnegative powers of an element $t$ coincide, say $t^{m}=t^{n}$, with $0\leq m<n$, then $t^{m}=t^{m+1}t^{n-m-1}$, so $t\in\{0\}\cup R^{*}$ by item \subcref{cm+1dividescm-a}.

\item If $R$ is finite, then item \subcref{cm+1dividescm-b} implies that $R$ coincides with $\{0\}\cup R^{*}$ and is therefore a field.\qedhere

\end{enumerate}

\end{proof}

\pagebreak

\noindent The following result shows that, like integral domains, reduced indecomposable rings can only have zero or prime characteristic:

\begin{proposition}\label{0orprime}

If $R$ is a reduced indecomposable ring of positive characteristic, then $R$ has prime characteristic. In particular, every nonzero integer in $R$ is invertible.

\end{proposition}

\begin{proof}

The prime subring $\ZZ$ of $R$ is reduced and indecomposable, since $R$ is. If $\Char(R)>0$, then $\ZZ$ is finite, so $\ZZ$ is a field by \Cref{cm+1dividescm-c}, and we know that in this case $|\ZZ|=\Char(R)$ is a prime number.
\end{proof}

\subsection{Examples of reduced and indecomposable rings}

\noindent Clearly, any integral domain is reduced and indecomposable. In this subsection we provide some examples of reduced/indecomposable rings that are not integral domains.

\begin{example}\label{Bpq}

Let $B$ be a ring, and let $p,q\in B$. We are going to impose sufficient conditions on $p$ and $q$ in such a manner that the ring $R=B/(pq)$ be reduced, indecomposable, and not an integral domain.

Suppose firstly that $p\nmid q$ and $q\nmid p$. This implies $pq\nmid p$ and $pq\nmid q$, hence the element $pq$ is not prime, and so $R$ is not an integral domain.

Furthermore, if $p$ and $q$ are prime, then $R$ is reduced: for if $a\in B$ and $n\geq 1$ satisfy $pq\mid a^{n}$, then by primality of $p$ and $q$ we have $p\mid a$ and $q\mid a$, say $a=sp=tq$. As $q$ is prime and $q\nmid p$, we necessarily have $q\mid s$, which shows that $pq\mid a$.

If in addition the ideal $Bp+Bq$ in $B$ is proper, then $R$ is also indecomposable. In fact, if $a\in B$ satisfies $pq\mid a(a-1)$, then $p$ must divide $a$ or $a-1$, and the same for $q$. If $p$ and $q$ do not divide the same factor, then $1=a-(a-1)\in Bp+Bq$, which contradicts our assumption. Therefore $p$ and $q$ both divide either $a$ or $a-1$, which implies $pq\mid a^{2}$ or $pq\mid(a-1)^{2}$. As we already proved reducedness of $B/(pq)$, either $pq\mid a$ or $pq\mid a-1$, as desired.

As concrete examples of rings satisfying the conditions above, we can take $B=\Z[t],p=2,q=t$, or $B=\Q[s,t],p=s,q=t$. In the latter case, we obtain an example of reduced indecomposable characteristic zero ring $R$ which is not a field, but such that every nonzero integer is invertible.

As a final remark, we could replace the hypotheses ``$p\nmid q$ and $q\nmid p$'' by ``$q$ is regular and $q\nmid p$'', which, together with the remaining hypotheses, would still imply that $R$ is reduced, indecomposable, and not an integral domain.

\end{example}

\begin{example}\label{C(X|B)}

For a set $X$ with at least two elements and a ring $B$, let $S=B^{X}$ be the set of $B$-valued functions on $X$. Endowed with componentwise addition and product, $S$ is a ring. On the one hand, if $B$ is reduced, then so is any subring of $S$; on the other hand, if $B$ is indecomposable, then the idempotent elements of a given subring of $S$ are precisely those functions that take only the values $0$ and $1$.

If $B$ is a reduced indecomposable topological ring such that its singletons are closed sets (that is, endowed with a $T_{1}$ topology), and $X$ is a connected topological space, then $R=\mathcal{C}(X,B)$, the subring of $S$ of $B$-valued continuous functions on $X$, is indecomposable: for if $f\in R$ is idempotent, then $X=f^{-1}(\{0\})\cup f^{-1}(\{1\})$ is the disjoint union of two closed sets, so by connectedness of $X$ we must have that $f$ is constant.

Consequently, the existence in $R$ of two continuous functions with disjoint supports provides examples of reduced indecomposable rings that are not integral domains. The last condition is guaranteed in many cases: for instance, if $B=\mathbb{R}$, this holds whenever $X$ separates some pair of disjoint closed sets, which is the case if $X$ is a metric space or a completely regular space or, under certain standard assumptions, whenever $X$ is a normal space\footnotemark\!.

\end{example}

\footnotetext{Urysohn's lemma cannot be proved in \textsf{ZF} (\cite{GoodT1995}*{Corollary 2.2}): the usual proof of this result relies on \textsf{DC}. However, as shown in \cite{Blass1979}*{p.~55}, it suffices to use \textsf{DMC}, the axiom of dependent multiple choice (\cite{HowardR1998}*{Form 106}).}

\begin{example}\label{ZxZ}

Consider the subring $R$ of $\Z\times\Z$ consisting of those pairs $(m,n)$ with $m\equiv n\pmod 2$. Since $\Z\times\Z$ is reduced, so is $R$. Moreover, the idempotents in $\Z\times\Z$ are precisely $(0,0),(1,1),(1,0)$ and $(0,1)$; since $(1,0),(0,1)\notin R$, it follows that $R$ is indecomposable.

\end{example}

\noindent Notice that the main result of this paper (\Cref{finalboss}) implies that $\Z$ is definable in the subring $R[x]$ of the ring $(\Z\times\Z)[x]\cong\Z[x]\times\Z[x]$, where $R$ is as described in \Cref{ZxZ}. In this line of thought, the reader may wonder whether $\Z$ is definable in $(\Z\times\Z)[x]$. Nevertheless, one can extract from the proof of \cite{AschenbrennerKNS2018}*{Lemma 4.7} that this is not the case (actually, that $\Z$ is not even definable in $A\times B$, whenever $A$ and $B$ are characteristic zero rings\footnote{See \cite{Arthan2016} for another proof in the case $A=B=\Z$.}\!). In other words, the condition on the subring $R$ in \Cref{ZxZ} is essential for the definability of $\Z$ in $R[x]$ (see \Cref{ZnotdefinAxB} for details).

\Cref{ZxZ} is just a special case of the following more general class of examples:

\begin{example}\label{B^I}

Let $B$ be a reduced indecomposable ring which is not a field (for example, an integral domain such as $\Z$ or $\F_{p}[t],p$ prime), and let $\Fb$ be a nonzero proper ideal in $B$. Given a set $I$ with more than one element, let $R\subseteq B^{I}$ be the set of $I$-tuples whose entries are pairwise congruent modulo $\Fb$. Since $B^{I}$ is reduced, so is $R$. The set of idempotents in $B^{I}$ is precisely $\{0,1\}^{I}$ and, since $\Fb$ is a proper ideal, it follows that $\{0,1\}^{I}\cap R=\{0_{R},1_{R}\}$, which shows that $R$ is indecomposable.

Finally, for each $i\in I$, denote by $e_{i}$ the $i$-th canonical $I$-tuple in $B^{I}$ taking value $1$ at position $i$ and $0$ elsewhere. If $c$ is a nonzero element in $\Fb$, then $R$ contains two nonzero elements of the form $ce_{i}$ and $ce_{j}$, with $i,j\in I$ and $i\neq j$, whose product is $0$, and this shows that $R$ is not an integral domain.

\end{example}

\noindent Unless $I$ is finite\footnote{If $\Fb=Rc$ is principal and $I=\{1,\ldots,n\}$, then $R$ is the image of the ring $B[x_{1},\ldots,x_{n}]$ under the ring homomorphism $f\mapsto\bigl(f(ce_{1}),\ldots,f(ce_{n})\bigr)$. Thus, $B$ being Noetherian implies that $R$ is Noetherian as well.}\!, the ring $R$ in \Cref{B^I} is not, in general, Noetherian: indeed, if $I$ contains a denumerable subset $\{i_{n}\colon n\in\N\}$ (that is, if $I$ is Dedekind-infinite), $c$ is a nonzero element of $\Fb$ and $\Fc_{n}\subseteq R$ is the ideal generated by $ce_{i_{0}},\ldots,ce_{i_{n}}$, then the ascending chain of ideals $(\Fc_{n})_{n\in\N}$ is not stationary\footnote{If $I$ is merely infinite, then we can only prove that $R$ has a non-finitely generated ideal, namely, that one generated by all the $I$-tuples $ce_{i}$. See \cite{Hodges1974}*{Section 3} for a comparison, in \textsf{ZF}, of the various notions of Noetherianity.}\!.

The reader may notice that the technique shown in \Cref{B^I} also provides examples in positive characteristic (which is necessarily prime, by \Cref{0orprime}). More specifically, for each $p$ prime, the following ring is reduced and indecomposable, has characteristic $p$, and it is not an integral domain:
\[R=\bigl\{(f,g)\in\F_{p}[t]\times\F_{p}[t]\colon t\mid f-g\bigr\}\,.\]

\begin{example}\label{stronglocal}

Let $R$ be a local ring (see \Cref{ourlocal}). If $a\in R$ is idempotent, then we have $(a-1)a=0$; since one of $a-1$ or $a$ is a unit, it follows that $a=0$ or $a-1=0$, which proves that $R$ is indecomposable. This provides more examples of reduced indecomposable rings which are not integral domains, obtained as suitable localizations of further rings at prime ideals\footnotemark\!, such as the germs of rational functions at points lying in more than one irreducible component of a (reduced) algebraic set (e.g.~$R=\bigl(\,\mathbb{C}[x,y]\bigl/(xy)\,\bigr)_{(\overline{x},\overline{y})}$\,).

\end{example}

\footnotetext{If $A$ is a ring and $\Fp$ is a prime ideal in $A$, then the localization $R=A_{\Fp}$ is local. In fact, if $a\in A$ and $s\in A\smallsetminus\Fp$ are such that $a/s$ is not a unit in $R$, then necessarily $a\in\Fp$, and thus $a+s\in A\smallsetminus\Fp$. Therefore $(a/s)+1=(a+s)/s$ is invertible in $R$.}

\subsection{Polynomials \texorpdfstring{\emph{versus}}{\unichar{"1D463}\unichar{"1D452}\unichar{"1D45F}\unichar{"1D460}\unichar{"1D462}\unichar{"1D460}} polynomial functions}

\noindent In this subsection we address the relationship between polynomials in one variable and their corresponding polynomial functions. More specifically, we want to provide a sufficient condition on the coefficient ring that ensures that polynomial constant functions can only come from constant polynomials.

\setcounter{footnote}{0}

If $R$ is a finite ring, then the nonzero polynomial $\prod_{r\in R}(x-r)$ is zero as a function on $R$, so we may restrict our discussion to infinite rings. If $D$ is an integral domain, then any nonzero polynomial $f\in D[x]$ can only have finitely many roots; in particular, if $D$ is infinite, then $f$ does not vanish identically on $R$ (as a polynomial function). For infinite reduced indecomposable rings, the set of roots of a nonzero polynomial may be infinite (take for instance the reduced indecomposable ring of characteristic zero $R=\Z[t]\bigl/(2t)$ in \Cref{Bpq}, and consider the polynomial $\overline{t}\cdot(x^{2}+x)\in R[x]$, vanishing at all integers), yet it can never be all of $R$, as the following result shows\footnote{See \cite{Sawin2014} for a general condition, and \cite{VanDobben2017} for a second-order topological proof, which relies on different notions of Noetherianity (whose equivalence depends on \textsf{DC}) and connectedness of the prime spectrum (which depends on \textsf{BPI}).}\!.

\begin{theorem}\label{constantfpoly}

Let $R$ be a reduced indecomposable ring. Assume that $R$ is infinite, and let $f\in R[x]$. If $f(c)=0$ for all $c\in R$, then $f=0$.

\end{theorem}

\begin{proof}

The case of integral domains was just discussed, so we may assume that $R$ is not a field. Write $f=f_{m}x^{m}+\cdots+f_{0}\in R[x]$, with $m\geq 0$. For $c_{0},\ldots,c_{m}\in R$, let $V(c_{0},\ldots,c_{m})$ be the Vandermonde matrix associated to these elements, that is, the matrix with rows indexed from $0$ to $m$, the $i$-th row being equal to $(1,c_{i}^{\phantom{1}},c_{i}^{2},\ldots,c_{i}^{m-1},c_{i}^{m})$, and for $a\in R$, let $V_{a}=V(a^{0},a^{1},\ldots,a^{m})$.

We claim that if $a\det(V_{a})=0$, then $a\in\{0\}\cup R^{*}$: in fact, recall that $\det(V_{a})=\prod_{0\leq i<j\leq m}(a^{i}-a^{j})$. Since $a^{i}-a^{j}=a^{i}\cdot(1-a^{j-i})$ for each $i$ and $j$ with $0\leq i<j\leq m$, it follows that $\det(V_{a})=a^{k}[1-ag(a)]$ for some $k\geq 1$ and some $g\in\Z[x]$. Therefore $a\det(V_{a})=0$ becomes $a^{k+1}=a^{k+2}g(a)$, and the claim follows from item \subcref{cm+1dividescm-a} of \Cref{cm+1dividescm}.

If $w$ denotes the column vector with entries $f_{0},\ldots,f_{m}$, and $V=V(c_{0},\ldots,c_{m})$, where $c_{0},\ldots,c_{m}$ are arbitrary constants, then $Vw$ is the column vector with entries $f(c_{0}),\ldots,f(c_{m})$, so that $Vw=0$. Multiplying this equality by the adjugate of $V$ yields $\det(V)w=0$, and so for each $i$ we have $f_{i}\det(V)=0$ for any choice of elements $c_{0},\ldots,c_{m}\in R$; in particular $f_{i}\det(V_{f_{i}})=0$, and consequently $f_{i}\in\{0\}\cup R^{*}$. Thus, to prove that all coefficients of $f$ are zero, it suffices to show that none of them is invertible.

If some $f_{i}$ were invertible, then $\det(V)=f_{i}^{-1}\cdot[f_{i}\det(V)]=0$ for all $c_{0},\dots,c_{m}\in R$. Consequently $a\det(V_{a})=0$ for all $a\in R$, so $R=\{0\}\cup R^{*}$, contradicting the assumption that $R$ is not a field.
\end{proof}

\noindent Notice that, in the previous result, none of the two conditions (indecomposability and reducedness) can be removed from the hypothesis. We provide counterexamples in both directions. On the one hand, infinite Boolean rings such as $R=\F_{2}^{\N}$ are reduced but not indecomposable and the nonzero polynomial $x^{2}-x$ vanishes everywhere as a function. On the other hand, $R=\F_{2}[\{x_{i}\}_{i\in\N}]\bigl/(x_{i}x_{j})_{i,j\in\N}$ is indecomposable but not reduced, and the polynomial $(x^{2}-x)^{2}$ is null as a function. The examples above are treated in detail in \Cref{nonconstantfpoly}.

\section{Logical powers in reduced and indecomposable polynomial rings}\label{sectionlpowR[x]}

\noindent In this section we study the properties of the logical powers (see \Cref{deflpow}) of a polynomial for reduced and/or indecomposable coefficient rings.

\subsection{Powers \texorpdfstring{\emph{versus}}{\unichar{"1D463}\unichar{"1D452}\unichar{"1D45F}\unichar{"1D460}\unichar{"1D462}\unichar{"1D460}} logical powers}

\begin{lemma}\label{redinftyp}

Let $R$ be a reduced ring. If $p\in R[x]$ is nonconstant, then no element of $\lpow(p)$ can be infinitely divisible by $p$. If in addition the leading coefficient of $p$ is regular, then $\lpow(p)\subseteq\pow(p)$.

\end{lemma}

\begin{proof}

Let $d=\deg(p)$ and $c\neq 0$ be the leading coefficient of $p$. Since $R$ is reduced, the leading coefficient of $p^{r}$ is $c^{r}\neq 0$, for all $r\geq 1$. Moreover, as $d>0$, for any given $f\in\lpow(p)$ we may find $r\geq 1$ such that $\deg(p^{r})=rd>\deg(f)$.

Suppose by contradiction that $f$ be infinitely divisible by $p$, and hence divisible by $p^{r}$. Item \subcref{basic-c} of \Cref{basic} ensures then that $f$ is annihilated by some power of $c$, say $c^{s}$. Setting $\ell=1+c^{s}p$, we find that $\ell$ divides $f$ (as $\ell f=f$) and that $p$ does not divide $\ell$ (otherwise $p$ would be a nonconstant invertible polynomial, contradicting \Cref{redunits}). Therefore, as $f\in\lpow(p)$, we must have that $\ell$ is invertible. However, we have that $\ell=c^{s}p+1$ is nonconstant, having coefficient $c^{s+1}\neq 0$ in degree $d>0$, and so it cannot be invertible (again by \Cref{redunits}), a contradiction.

After proving that any $f\in\lpow(p)$ cannot be infinitely divisible by $p$, item \subcref{lpowgeneral-a} of \Cref{lpowgeneral} guarantees that $f$ has the form $up^{n}$, for some integer $n\geq 1$ and a unit $u$ satisfying $p-1\mid u-1$. As $u$ is constant (\Cref{redunits}) and $p-1$ has positive degree and leading coefficient $c$, then again by \Cref{basic-c} we have that $u-1$ is annihilated by a power of $c$. Finally, if $c$ is regular, then $u-1$ must be zero and $f\in\pow(p)$, proving the second assertion.
\end{proof}

\begin{corollary}\label{lp=pindom}

If $R$ is reduced and $p\in R[x]$ is nonconstant, prime, and it has a regular leading coefficient, then $\lpow(p)=\pow(p)$. In particular $\lpow(x)=\pow(x)$ when $R$ is an integral domain.

\end{corollary}

\begin{proof}

The fact that the leading coefficient of $p$ is regular implies that $p$ is regular, and therefore we can apply \Cref{lpowgeneral-d} to obtain $\pow(p)\subseteq\lpow(p)$. The reverse inclusion follows from \Cref{redinftyp}.
\end{proof}

\noindent The requirement that $\lpow(x)=\pow(x)$, together with the technique shown in \Cref{singlepOydes}, could be at the base of a specific strategy for definability of integers in polynomial rings. However, \Cref{lp=pindom} above only guarantees that $\lpow(x)=\pow(x)$ for integral domains, where the issue of definability of integers has already been worked out, in a Diophantine way (\cite{Shlapentokh1990}*{Theorem 5.1}). Fortunately, we now have all the tools to characterize the rings $R$ such that, in the polynomial ring $R[x]$, the equality $\lpow(x)=\pow(x)$ holds, obtaining in this way the converse of \Cref{lpowx=powximplies}:

\begin{theorem}\label{lpowx=powx}

Let $R$ be a ring and consider $R[x]$, the polynomial ring in one variable over $R$.

\begin{enumtheorem}

\item\label{lpowx=powx-a} If $R$ is reduced, then $\lpow(x)\subseteq\pow(x)$.

\item\label{lpowx=powx-b} $\pow(x)=\lpow(x)$ if, and only if, $R$ is reduced and indecomposable.

\end{enumtheorem}

\end{theorem}

\begin{proof}\leavevmode

\begin{enumerate}

\item This follows immediately from \Cref{redinftyp}.

\item If $\lpow(x)=\pow(x)$, then \Cref{lpowx=powximplies,bivalente} together imply that $R$ is reduced and indecomposable.

Conversely, suppose that $R$ is reduced and indecomposable. For every $r\geq 1$ we have that $x$ divides $x^{r}$ and $x-1$ divides $x^{r}-1$. Suppose that $x^{r}=gh$, with $g,h\in R[x]$. Following notation as in the beginning of \Cref{definitions}, we are denoting by $g_{0}$ the constant term of $g$ and by $h_{r}$ the coefficient of $x^{r}$ in $h$. Using \Cref{basic-b}, we get that $x^{r}$ divides $g_{0}^{r}h$ and $g_{0}^{r}=g_{0}^{r+1}h_{r}$, hence $g_{0}\in\{0\}\cup R^{*}$ by \Cref{cm+1dividescm-a}.

If $g_{0}=0$, then $x$ divides $g$. Otherwise, $x^{r}$ divides $g_{0}^{-r}\cdot(g_{0}^{r}h)=h$, say $h=x^{r}\bighat{h}$, hence $x^{r}=gh=x^{r}g\bighat{h}$; canceling out $x^{r}$ we conclude that $g$ is invertible. This shows that $x^{r}\in\lpow(x)$ for all $r\geq 1$, that is, $\pow(x)\subseteq\lpow(x)$, and the reverse inclusion follows from item \subcref{lpowx=powx-a}.\qedhere

\end{enumerate}

\end{proof}

\noindent Next, we try to distinguish by a logical formula some elements of $R[x]$ whose logical powers coincide with their positive powers. To this end, it is necessary to exclude elements exhibiting logical powers infinitely divisible by them. One way of doing so, which will be presented in the following subsection, relies on producing a first-order equivalent of the concept of ``powers of two given elements have the same exponent''\footnote{This concept is somewhat outlined in the description of the set $U$ appearing in \Cref{defsTandUPV}, and more explicitly exploited in the proof of \Cref{onlypowconst}. Finally, we are able to fully express it in the first-order language of rings, in a definitive way, when dealing with interpretability of the structure $(\Zm,+,\mid\,)$ in the rings of our class (namely, the formulas $=_{\Gamma}$ in \Cref{interprnonfield,interprcharzero,interpretcharp}).}\!, and exploits and extends the fact that, under reasonable conditions, for polynomials $p$ and $q$ we have that $p-q$ divides $p^{m}-q^{n}$ forces $m=n$.

\subsection{Some convenient sets whose elements have definable sets of powers}

\noindent The goal of this subsection is to construct special definable subsets of a ring $S$, which will end up being useful throughout the paper. When $S=R[x]$, with $R$ reduced and indecomposable, the elements of such sets will turn out to have definable sets of powers. If in addition $R$ is not a field, we are able to show that every constant element in $S$ also has a definable set of powers.

\enlargethispage*{5mm}

\begin{definition}\label{defsTandUPV}

For a ring $S$, we define the following sets:

\begin{itemize}[labelindent=13pt,itemindent=0em,leftmargin=!,itemsep=3mm]

\item $T$ is the set of elements $p\in S$ such that $p$ is irreducible and $ph\in\lpow(p)$ whenever $h\in\lpow(p)$.

\item $U$ is the set of elements $p\in T$ such that:

\begin{itemize}

\item For every $q\in T$ and every $f\in\lpow(p)$, there exists $g\in\lpow(q)$ such that $p-q\mid f-g$;

\item If $a\in S^{*}$ satisfies $p-1\mid a-1$, then $a=1$.

\end{itemize}

\item $P$ is the set of elements $p\in U$ such that $p-1$ is regular.

\item $V$ is the set of elements $p\in U$ such that:

\begin{itemize}

\item $p$ is regular;

\item For any $y,z\in\{1\}\cup\lpow(p)$, if $y-1\mid z-1$ and $z-1\mid y-1$, then $y=z$.

\end{itemize}

\end{itemize}

\end{definition}

\noindent Observe that the sets $T,U,P$ and $V$ are first-order definable and $P,V\subseteq U\subseteq T$. Since irreducible elements are noninvertible by definition, it follows that the sets in \Cref{defsTandUPV} consist of nonunits. Although we will not use the properties of the sets $P$ and $V$ until the following section, we have opted for introducing them altogether in the definition above.

\begin{theorem}\label{TandUPV}

Let $S$ be a ring and let $T,U,P$ and $V$ as in \Cref{defsTandUPV}. For all $q\in T$ we have $\pow(q)\subseteq\lpow(q)$. In addition, if $S=R[x]$, with $R$ reduced and indecomposable, then the following hold:

\begin{enumtheorem}

\item\label{TandUPV-a} $x\in U$.

\item\label{TandUPV-b} $\lpow(p)=\pow(p)$ for every $p\in U$.

\item\label{TandUPV-c} $P$ and $V$ are nonempty; more specifically, we have $x\in P\cap V$.

\end{enumtheorem}

\end{theorem}

\begin{proof}

If $q\in T$, then $q$ is irreducible, hence $q\in\lpow(q)$ by \Cref{lpowgeneral-c}, and if we assume inductively that $m\geq 1$ satisfies $h=q^{m}\in\lpow(q)$, then $q^{m+1}=qh\in\lpow(q)$, by the definition of $T$. This shows that $\pow(q)\subseteq\lpow(q)$.

Suppose that $S=R[x]$, with $R$ reduced and indecomposable.

\begin{enumerate}

\item First, we prove that $x\in T$. Since $R$ is indecomposable, $x$ is irreducible by \Cref{bivalente}. Moreover, as $R$ is also reduced, it follows from \Cref{lpowx=powx-b} that $\lpow(x)=\pow(x)$. Therefore $x\in\pow(x)=\lpow(x)$, and if $h\in\lpow(x)=\pow(x)$, then $h=x^{k}$ for some $k\geq 1$, hence $xh=x^{k+1}\in\pow(x)=\lpow(x)$. Thus, $x\in T$, as desired.

Regarding the remaining conditions for membership in $U$, given $q\in T$ and $f\in\lpow(x)$, we want to find $g\in\lpow(q)$ such that $x-q$ divides $f-g$. Since $\lpow(x)=\pow(x)$, we have $f=x^{n}$ for some $n\geq 1$. Moreover, we already know that $q\in T$ implies $\pow(q)\subseteq\lpow(q)$, and so by taking $g=q^{n}$, we get $g\in\lpow(q)$, and clearly $x-q$ divides $x^{n}-q^{n}=f-g$. Finally, let $a\in S^{*}$ be such that $x-1$ divides $a-1$. By \Cref{redunits} we have $S^{*}\subseteq R$, and therefore $a$ is constant. Writing $a-1=(x-1)\ell$, we can evaluate at $x=1$ to conclude $a=1$. Consequently, $x\in U$.

\item If $p\in U$, then $p\in T$, and consequently $\pow(p)\subseteq\lpow(p)$. For the reverse inclusion, let $f\in\lpow(p)$ and set $q=x\in T$. The first condition in the definition of $U$ guarantees the existence of an element $g\in\lpow(q)=\lpow(x)=\pow(x)$ such that $p-x\mid f-g$, say $g=x^{n}$, with $n\geq 1$.

If $f$ were infinitely divisible by $p$, then $p$ would be constant by \Cref{redinftyp}. Evaluating at $p$ and using that $x-p$ divides $x^{n}-f$, we conclude that $f(p)=p^{n}$. Since $f$ is infinitely divisible by $p$, there is an $h$ such that $f=p^{n+1}h$; in particular we have $f(p)=p^{n+1}h(p)$, so $p^{n+1}$ divides $p^{n}$. \Cref{cm+1dividescm-a} would imply then that $p\in\{0\}\cup R^{*}$, which is absurd since $p$ is irreducible.

The contradiction above, together with item \subcref{lpowgeneral-a} of \Cref{lpowgeneral}, shows that $f=ap^{k}$ for some $k\geq 1$ and some $a\in R[x]^{*}$ with $p-1\mid a-1$; the second condition of the definition of $U$ forces $a=1$ and, consequently, $f=p^{k}\in\pow(p)$.

\item By item \subcref{TandUPV-a} we have $x\in U$, and clearly $x-1$ is regular. Therefore $x\in P$. Concerning the proof of membership of $x$ in $V$, first notice that $x$ is regular. For the remaining condition, we proceed by adapting, to our context, an argument from \cite{RobinsonR1951}*{\S 4b} in what follows.

Let $k,j\geq 0$ be such that $x^{k}-1\mid x^{j}-1$. We claim that $k\mid j$. In fact, if $k=0$, then $0\mid x^{j}-1$, hence $x^{j}=1$, so $j=0$ and thus $k\mid j$. Otherwise we may write $j=qk+r$, with $q\geq 0$ and $0\leq r<k$. Since
\[x^{k}-1\mid x^{j}-1=x^{r}(x^{qk}-1)+x^{r}-1\]
and $x^{k}-1\mid x^{qk}-1$, it follows that $x^{k}-1\mid x^{r}-1$. From general ring theory, if $f\in R[x]$ has regular leading coefficient, then for all $g\in R[x]\smallsetminus\{0\}$ we have $fg\neq 0$ and $\deg(fg)=\deg(f)+\deg(g)$. Taking $f=x^{k}-1$ and taking into account that $r<k$, we get that $x^{k}-1\mid x^{r}-1$ can only occur if $x^{r}-1=0$. Therefore $r=0$, which shows that $k\mid j$ in this case as well.

Finally, let $y,z\in\{1\}\cup\lpow(x)$ be such that $y-1\mid z-1$ and $z-1\mid y-1$. We want to show that $y=z$. Since $\lpow(x)=\pow(x)$ by items \subcref{TandUPV-a} and \subcref{TandUPV-b}, we have $y=x^{k}$ and $z=x^{j}$ for some $k,j\geq 0$, and therefore $x^{k}-1\mid x^{j}-1$ and $x^{j}-1\mid x^{k}-1$. The previous reasoning shows then that $k\mid j$ and $j\mid k$, hence $k=j$, and so $y=z$, as desired.\qedhere

\end{enumerate}

\end{proof}

\begin{remark}\label{automorphU}

Let $S$ be any ring. If $\theta$ is a ring automorphism of $S$, then $\theta$ preserves the logical structure, and therefore the definable sets $T,U,P$ and $V$ of \Cref{defsTandUPV} are invariant under $\theta$, that is, $\theta(T)=T,\theta(U)=U,\theta(P)=P$ and $\theta(V)=V$. If $S=R[x],v\in R^{*}$ and $r\in R$, then the mapping $\theta\colon S\to S$ given by $\theta(f)=f(vx+r)$ is a ring automorphism ($g\mapsto g\bigl(v^{-1}\cdot(x-r)\bigr)$ being its inverse). If $R$ is reduced and indecomposable, then $x\in P\cap V$ by \Cref{TandUPV-c}, and therefore we have $vx+r\in P\cap V\subseteq U$ in this case. In other words, if $L$ denotes the set of automorphic images of $x$, namely
\[L=\{vx+r\colon v\in R^{*},r\in R\}\,,\]
then we have $x\in L\subseteq V\cap P$. In \Cref{interpretability} we will see that $L$ is definable whenever $R$ is a field, and this is crucial to provide a proof of interpretability when $R$ is a field of positive characteristic.

\end{remark}

\noindent The last result of this subsection (\Cref{onlypowconst}) ensures definability of sets of powers of any fixed constant, using the corresponding constant as a parameter, for reduced indecomposable coefficient rings that are not fields. Before proceeding, we need the following technical result:

\begin{lemma}\label{sumnonunits}

Let $R$ be a ring. If $R$ is not a field, then at least one of the following holds:

\begin{itemize}

\item There exists a unit $u$ with $u-1\notin\{0\}\cup R^{*}$.

\item Every element of $R$ is the sum of two nonunits.

\end{itemize}

\end{lemma}

\begin{proof}

If $R$ is local (see \Cref{ourlocal}), then, as it is not a field, we may take $z\notin\{0\}\cup R^{*}$, so that $u=z+1$ must be a unit, satisfying the first property. If $R$ is not local, then nonunits are not closed under sum. Hence, some unit $w$ must be the sum of two nonunits, say $x$ and $y$, and therefore for any $r\in R$ we have that $r=rw^{-1}w=(rw^{-1}x)+(rw^{-1}y)$ is the sum of two nonunits.
\end{proof}

\begin{theorem}\label{onlypowconst}

Let $S=R[x]$, with $R$ being a reduced indecomposable ring that is not a field, and let $U$ be as in \Cref{defsTandUPV}. Given $f\in S$ and $a\in R$, we have that $f\in\pow(a)$ if, and only if, for all $p,q\in U$, there exist $y\in\pow(p)$ and $z\in\pow(q)$, such that:

\begin{itemize}

\item $p-a\mid y-f$;

\item $q-a\mid z-f$;

\item $p-q\mid y-z$.

\end{itemize}

\end{theorem}

\begin{proof}

If $f=a^{n}$, with $n\in\Zm$, then for any $p,q\in U$, by taking $y=p^{n}$ and $z=q^{n}$, one clearly has $p-a\mid y-f,q-a\mid z-f$ and $p-q\mid y-z$. Conversely, let $f\in R[x]$ satisfy the properties listed. We will prove that $f$ is constant as a function on $R$.

Given any two $\rho,\sigma\in R$ and any $\upsilon\in R^{*}$, define the polynomials $p=x-\rho+a$ and $q=\upsilon x-\sigma+a$ and observe that both $p$ and $q$ lie in $L\subseteq U$ (see \Cref{automorphU}). By the properties listed in the hypothesis, there exist elements $y=p^{m}=(x-\rho+a)^{m}$ and $z=q^{n}=(\upsilon x-\sigma+a)^{n}$, where $m$ and $n$ are suitable positive integers depending on $p$ and $q$ (and, of course, on $a$), satisfying:

\begin{itemize}

\item $x-\rho+a-a\mid (x-\rho+a)^{m}-f$;

\item $\upsilon x-\sigma+a-a\mid (\upsilon x-\sigma+a)^{n}-f$;

\item $(x-\rho+a)-(\upsilon x-\sigma+a)\mid p^{m}-q^{n}$;

\end{itemize}

\noindent which yields:

\begin{itemize}

\item $f(\rho)=a^{m}$;

\item $f(\upsilon^{-1}\sigma)=a^{n}$;

\item $(1-\upsilon)x+(\sigma-\rho)\mid (x-\rho+a)^{m}-(\upsilon x-\sigma+a)^{n}$.

\end{itemize}

\noindent In particular we have $f(0)\in\pow(a)$ (just take $\rho=0,\sigma=0$ and $\upsilon=1$).

Fix a triplet $(\rho,\sigma,\upsilon)$ and take any $m=m(\rho,\sigma,\upsilon)$ and $n=n(\rho,\sigma,\upsilon)$ satisfying the conditions above. If $m\neq n$, then $(x-\rho+a)^{m}-(\upsilon x-\sigma+a)^{n}$ has invertible leading coefficient, being $1$ or $-\upsilon^{n}$, and therefore, by \Cref{basic-d}, the last condition can only be satisfied if the leading coefficient of $(1-\upsilon)x+(\sigma-\rho)$ is also invertible. If this does not happen, then we must have $m=n$ and therefore $f(\rho)=a^{m}=a^{n}=f(\upsilon^{-1}\sigma)$.

The above reasoning amounts to saying that, given any $\rho,\sigma\in R$ and any $\upsilon\in R^{*}$, if any of the following conditions holds:

\begin{enumerate}[label=\texttt{(\alph*)}]

\item $\upsilon\neq 1$ and $\upsilon-1\notin R^{*}$;

\item $\upsilon=1$ and $\rho-\sigma\notin R^{*}$,

\end{enumerate}

\noindent then $f(\rho)=f(\upsilon^{-1}\sigma)$.

Take any $r\in R$: we want to prove that $f(r)=f(0)$. By \Cref{sumnonunits}, either there exists a unit $u$ with $u-1\notin\{0\}\cup R^{*}$ or any element of $R$ is the sum of two nonunits. In the first case, condition \texttt{(a)} is satisfied for $\upsilon=u$; taking $\rho=r$ and $\sigma=0$ we conclude that $f(r)=f(\rho)=f(\upsilon^{-1}\sigma)=f(0)$. In the second case, there are two nonunits $s$ and $t$ such that $r=s+t$. Set $\upsilon=1$. Considering $\rho=r$ and $\sigma=s$, we can use \texttt{(b)} to prove that $f(r)=f(1^{-1}\cdot s)=f(s)$. Analogously, considering $\rho=s$ and $\sigma=0$, we can use \texttt{(b)} again to prove that $f(s)=f(1^{-1}\cdot 0)=f(0)$. Thus, $f(r)=f(s)=f(0)$.

We have proven that, in both cases, $f(r)=f(0)$. As $r$ was arbitrarily taken, it follows that $f$ is constant as a function on $R$. Since $R$ is reduced and indecomposable but not a field, it follows from \Cref{cm+1dividescm-c} that $R$ is infinite, and thus \Cref{constantfpoly} ensures that $f=f(0)\in\pow(a)$.
\end{proof}

\begin{remark}

Let $S=R[x]$ be as in \Cref{onlypowconst}. We have that the sets of powers of elements of $U$ coincide with their corresponding sets of logical powers (\Cref{TandUPV-b}), and therefore they are definable, using the corresponding elements as parameters; see \Cref{lpowformula}. Since the condition in the statement of \Cref{onlypowconst} involves quantification over the definable set $U$, we get that the set of positive powers of any constant $a\in R$ is definable in $R[x]$ using $a$ as a parameter. In other words, we proved the following:

\begin{corollary}\label{Phipowerconstant}

Let $S=R[x]$, with $R$ being a reduced indecomposable ring that is not a field. There is a two-variable first-order formula $\Phi(\cdot,\cdot)$ such that, for each $a\in R$, the formula $\Phi(\cdot,a)$ defines the set $\pow(a)$ in $S$. More explicitly, we can take
\begin{align*}
\Phi(t,a)\colon\ \ \forall p\,\forall q\,\bigl(\,&[\,p\in U\ \wedge\ q\in U\,]\rightarrow\exists y\,\exists z\,[\,y\in\lpow(p)\ \wedge\ z\in\lpow(q)\\
&\hspace{-3pt}\wedge\ p-a\mid y-t\ \wedge\ q-a\mid z-t\ \wedge\ p-q\mid y-z\,]\,\bigr)\,.
\end{align*}

\end{corollary}

\end{remark}

\section{The main results}\label{sectionZdefin}

\noindent We end this paper by proving both the undecidability of the full theory and the definability of the prime subring of $R[x]$, whenever $R$ is a reduced indecomposable ring. Undecidability will be obtained by generalizing a method from \cite{RobinsonR1951}. As for definability, clearly it is sufficient to define just the subset $\ZZ^{+}$ of positive integers in $S$. We will initially express the class of reduced indecomposable coefficient rings as a union of two subclasses, for each of which we produce a uniform formula defining $\ZZ^{+}$. Once this is done, we manipulate the two formulas obtained and merge them, in a convenient way, into a unified formula that covers the whole class.

\subsection{Undecidability of the full theory of reduced indecomposable polynomial rings (\emph{d'après} Raphael Robinson)}\label{undecidable}

\noindent Let $S=R[x]$, with $R$ a reduced indecomposable ring, and let $V$ be as in \Cref{defsTandUPV}. In this subsection we exploit the properties of $V$ to prove undecidability of the full theory of $S$, following the reasoning in \cite{RobinsonR1951}*{\S\S 4b,4c}. Our main definability result (\Cref{finalboss}) is a sufficient condition for undecidability in the case of characteristic zero. Nonetheless, the method described in this subsection works regardless of the characteristic.

Let $\beta$ be a two-variable formula in the language of rings, and let $S$ be a ring. If $p\in S$ satisfies

\begin{enumerate}[label=\texttt{(C\arabic*)}]

\item $p$ has all its nonnegative powers distinct,

\item $\beta(\cdot,p)$ defines the set $\{1\}\cup\pow(p)$ in $S$, and

\item For any nonnegative powers $f$ and $g$ of $p$, if $f-1\mid g-1$ and $g-1\mid f-1$, then $f=g$,

\end{enumerate}

\noindent then Robinson is able to translate the structure $(\N,+,\cdot,0,1)$ into the structure $(S,+,\cdot,0,1)$, using $p$ as a parameter, in the following way (see \cite{RobinsonR1951}*{\S 4b}): first, the product in $\N$ is expressed in terms of addition and divisibility in $\N$. Next, elements of $\N$ are encoded as the corresponding nonnegative powers of $p$; note that this uses \texttt{(C1}). In particular, $0_{\N}$ is encoded as $1_{S}$ and $1_{\N}$ is encoded as $p$, which explains the presence of $p$ as a parameter. Moreover, if a variable, say $x$, occurs in a formula, then we translate it as ``$x\in\{1\}\cup\pow(p)$'' (or, equivalently, as ``$\beta(x,p)$''). Finally, addition in $\N$ is realized using the product in $S$ (``law of exponents''), and divisibility $k\mid j$, with $k,j\in\N$, is realized as the divisibility $p^{k}-1\mid p^{j}-1$ in $S$.

Notice that, in the presence of \texttt{(C2)}, condition \texttt{(C3)} is first-order expressible in the theory of rings with parameter $p$. Moreover, Robinson proves that, in the presence of \texttt{(C1)}, condition \texttt{(C3)} guarantees the equivalence ``$k\mid j\iff p^{k}-1\mid p^{j}-1$'' (we use this reasoning in the proof of \Cref{TandUPV-c} above).

Thus, given a formula $\varphi(-)$ in the language $(+,\cdot,0,1)$, we have associated a formula $\varphi_{\beta}(-,p)$ in the language $(+,\cdot,0,1,p)$, with $p$ as a parameter. Under this association, whenever $\varphi$ is a sentence, we have that $\varphi$ holds in the structure $(\N,+,\cdot,0,1)$ if and only if $\varphi_{\beta}(-,p)$ holds in the structure $(S,+,\cdot,0,1,p)$.

In order to get rid of the dependence of the parameter $p$, suppose that there are a two-variable formula $\beta$ and a nonempty definable subset $\VV$ of $S$ such that the pair $(\beta,p)$ satisfies conditions \texttt{(C1)-(C3)} above, for each $p\in\VV$. If $\varphi$ is a sentence, then the formula
\[\bighat{\varphi}\colon\ \ \forall p\,[\,p\in\VV\rightarrow\varphi_{\beta}(p)\,]\]
is also a sentence. Notice that $\varphi$ holds in $(\N,+,\cdot,0,1)$ if and only if $\varphi_{\beta}(p)$ holds in $(S,+,\cdot,0,1,p)$ for each $p\in\VV$. As a consequence, the sentence $\varphi$ holds in the semiring $\N$ if and only if the sentence $\bighat{\varphi}$ holds in the ring $S$. From this, the undecidability of the full theory of $(S,+,\cdot,0,1)$ would follow from the undecidability of the full theory of $(\N,+,\cdot,0,1)$.

\begin{theorem}\label{undecidability}

Let $S=R[x]$, with $R$ reduced and indecomposable. Then $S$ is undecidable.

\end{theorem}

\begin{proof}

It is sufficient to define $\VV$ and $\beta$, as previously discussed. We may take $\VV$ to be the set $V$ introduced in \Cref{defsTandUPV}, and set $\beta$ to be
\[\beta(t,p)\colon\ \ t\in\lpow(p)\ \vee\ t=1\,.\]
We have $V\neq\varnothing$ by \Cref{TandUPV-c}. We claim that every element $p$ of $V$ satisfies conditions \texttt{(C1)-(C3)}. In fact, If $p\in V$, then $p$ is regular and noninvertible (see the commentary right after \Cref{defsTandUPV}), which in turn implies that $p$ satisfies \texttt{(C1)}: indeed, if $p^{k}=p^{j}$ for some $k,j$ with $0\leq j<k$, then canceling out $p^{j}$ on both sides (which is possible by regularity of $p$) would imply $p^{k-j}=1$. This, together with the fact that $k-j>0$, would imply in turn that $p$ is invertible, which is absurd.

For each $p\in U$, with $U$ as in \Cref{defsTandUPV}, we have $\pow(p)=\lpow(p)$ by \Cref{TandUPV-b}. Since $V\subseteq U$ by the definition of $V$, it follows that the formula $\beta(\cdot,p)$ defines the set $\{1\}\cup\pow(p)$ of all nonnegative powers of $p$, and so the pair $(\beta,p)$ satisfies \texttt{(C2)} for each $p\in V$. Finally, the very definition of the set $V$ implies that each element $p$ of $V$ satisfies \texttt{(C3)}.
\end{proof}

\noindent The technique shown in \cite{RobinsonR1951}*{\S 4c} provides one such definable set $\VV$ only in the case in which the coefficient ring $R$ is a field; namely, $\VV$ consists of the nonconstant prime elements of $R[x]$. To compensate this lack of generality, Robinson devises an alternative method (\cite{RobinsonR1951}*{\S 4d}), based on the notion of \emph{essential undecidability}, to prove the undecidability of every polynomial integral domain. Our result supersedes the undecidability results of \cite{RobinsonR1951}*{\S\S 4c,4d}.

It is important to point out that the mapping $\varphi\mapsto\bighat{\varphi}$ \emph{does not} supply an interpretation of $(\N,+,\mid\,,0,1)$ in $(S,+,\cdot,0,1)$ (see \cite{Hodges1993}*{Chapter 5} for a general discussion on interpretations). This happens because, among other things, it does not provide an interpretation for equality of exponents (of powers of elements in $\VV$). In other words, no formula $=_{\Gamma}$ associated with an interpretation $\Gamma$ is provided such that, for any $y,z$ of the form $y=p^{k},z=q^{j}$, with $k,j\in\N$ and $p,q\in\VV$, it is the case that $k=j$ if and only if $=_{\Gamma}(y,z)$ holds.

The methods described at the end of this section, however, do provide a two-dimensional interpretation of $(\Zm,+,\mid\,)$ in any ring of our class, which will automatically produce an interpretation of $(\Zm,+,\,\cdot\,)$.

\subsection{Defining sets of exponents: the first steps}\label{approxexpo}

\noindent In this subsection we provide a first-order technique for extracting ``approximate'' exponents from sets of powers, in the sense that, given a suitable element $p$ in a ring $S$, the (images in $\ZZ^{+}$ of the) exponents of its powers are determined modulo $p-1$. Of course, we are interested in extracting the (actual) images in $\ZZ^{+}$ of the exponents. This will be done in the two next subsections in two different ways, according to whether every nonzero element of the prime subring is invertible, or the coefficient ring is a nonfield of characteristic zero.

We remind the reader that, if $n$ is a positive integer (for example, when appearing as an exponent), then the symbol $n$ is also conventionally used in this work to denote the element $n\cdot 1_{S}$ in $S$, as discussed in \Cref{primesubring}.

\pagebreak

\begin{definition}\label{deflogpB}

Let $S$ be a ring, $p\in S$ and $B\subseteq\pow(p)$. We define the sets
\vspace{2mm}
\begin{align*}
\log_{p}B=&\,\{n\in\ZZ^{+}\colon p^{n}\in B\}\,,\\[2mm]
\log_{p}B+(p-1)S=&\,\{n+(p-1)s\colon n\in\log_{p}B,s\in S\}\,.
\end{align*}

\end{definition}

\noindent Notice that $\log_{p}B+(p-1)S$ is precisely the set of elements $t\in S$, such that $p-1$ divides $t-n$, for some $n\in\log_{p}B$.

In what follows, given a formula defining a set $B$ of powers of a fixed element $p$ such that $p-1$ is regular, we provide a formula that defines the set $\log_{p}B+(p-1)S$. Before we state our preliminary result we define, for $p\in S$ and $n\in\Zm$, the element
\begin{equation}\label{wnp}
w_{n}(p)=p^{n-1}+p^{n-2}+\cdots+p+1\in S\,.
\end{equation}
Observe that $w_{n}(p)$ satisfies the equality $(p-1)w_{n}(p)=p^{n}-1$. Moreover, writing $w_{n}(p)$ as
\begin{equation}\label{wnpcongn}
w_{n}(p)=\begin{cases}
1,\ &\textnormal{if}\ n=1;\\
n+(p-1)\sum_{k=0}^{n-2}(n-1-k)p^{k},\ &\textnormal{otherwise},
\end{cases}
\end{equation}
it follows immediately that $p-1$ divides $w_{n}(p)-n$. These relations are used crucially to prove the main results of this section. We begin our reasoning by introducing a formula, together with a lemma that makes its meaning clearer.

\begin{definition}\label{Lbeta}

For a two-variable formula $\beta$, we define the four-variable formula
\[L_{\beta}(t,p,y,w)\colon\ \ \beta(y,p)\ \wedge\ y-1=(p-1)w\ \wedge\ p-1\mid w-t\,.\]
Given a ring $S$, we denote by $B_{p}$ the subset of $S$ defined by $\beta(\cdot,p)$.

\end{definition}

\begin{lemma}\label{singlepOydes}

Let $S$ be a ring, and let $p\in S$ with $p-1$ regular. With notation as in \Cref{Lbeta}, suppose that $B_{p}\subseteq\pow(p)$.

\begin{enumlemma}

\item\label{singlepOydes-a} Given $t,y,w\in S$, we have that $L_{\beta}(t,p,y,w)$ holds if, and only if, there exists $n\in\log_{p}B_{p}$ such that

\begin{itemize}

\item $y=p^{n}$,

\item $w=w_{n}(p)$, and

\item $p-1$ divides $t-n$.

\end{itemize}

\item\label{singlepOydes-b} The formula $\exists y\,\exists w\,L_{\beta}(\cdot,p,y,w)$ defines the set $\log_{p}B_{p}+(p-1)S$ of elements $t\in S$, such that $p-1$ divides $t-n$ for some $n\in\log_{p}B_{p}$ \textup{(}see \Cref{deflogpB}\textup{)}.

\end{enumlemma}

\end{lemma}

\begin{proof}

We will use the fact that the element $w_{n}(p)=(p^{n}-1)/(p-1)$ is congruent to $n$ modulo $p-1$, which, together with the hypotheses, will allow us to recover the value $n$ modulo $p-1$ from the expression $p^{n}-1$ in a definable way.

\begin{enumerate}

\item Observe that $L_{\beta}(t,p,y,w)$ holds if and only if there exists a positive integer $n$ satisfying:

\begin{itemize}

\item $y=p^{n}\in B_{p}$ (recall that $B_{p}\subseteq\pow(p)$ by hypothesis),

\item $y-1=(p-1)w$, and

\item $p-1$ divides $w-t$.

\end{itemize}

\noindent The chain of equalities
\[(p-1)w_{n}(p)=p^{n}-1=y-1=(p-1)w\,,\]
together with the regularity of $p-1$, implies that the only possible such value of $w$ is $w_{n}(p)$. Thus, $L_{\beta}(t,p,y,w)$ holds if and only if there exists $n\in\log_{p}B_{p}$ such that

\begin{itemize}

\item $y=p^{n}$,

\item $w=w_{n}(p)$, and

\item $p-1$ divides $w_{n}(p)-t$.

\end{itemize}

\noindent Finally, recall that $p-1$ divides $w_{n}(p)-n$, so $p-1$ divides $w_{n}(p)-t$ if and only if $p-1$ divides $[w_{n}(p)-n]-[w_{n}(p)-t]=t-n$.

\item If $\exists y\,\exists w\,L_{\beta}(t,p,y,w)$ holds, then item \subcref{singlepOydes-a} implies that $p-1$ divides $t-n$, for some $n\in\log_{p}B_{p}$, and therefore $t=n+(t-n)\in\log_{p}B_{p}+(p-1)S$. Conversely, if $t=m+(p-1)s$, with $m\in\log_{p}B_{p}$ and $s\in S$, then $L_{\beta}(t,p,y,w)$ is satisfied by taking $y=p^{m}$ and $w=w_{m}(p)$.\qedhere

\end{enumerate}
\end{proof}

\noindent In our setting we have $S=R[x]$, with $R$ reduced and indecomposable. Given $r\in R$, the element $p=x-r+1$ is such that $p-1$ is regular. If $L_{\beta}(t,p,y,w)$ holds, then \Cref{singlepOydes-a} implies $x-r=p-1\mid t-n$, for some $n\in\ZZ^{+}$ possibly depending on $r$. This amounts to saying that $t$, considered as a polynomial function, satisfies $t(r)=n$.

In order to obtain from $L_{\beta}$ a formula that corresponds to ``$t\in\ZZ^{+}$'', we must necessarily bind the variables $y,z$ and $p$. First, we quantify existentially over $y$ and $w$, obtaining an auxiliary value $n\in\log_{p}B_{p}$, and afterwards we vary $p$ in a suitable definable subset containing all the linear polynomials $x-r+1$, with $r\in R$. The first step, besides leaving $n$ dependent on $p$, only specifies it modulo $p-1$. To fix this issue, we will express the class of reduced indecomposable polynomial rings as the union of two subclasses, for each of which a different technique defining $\ZZ^{+}$ is introduced. Both techniques involve making further restrictions on $t$. This will allow us, all in all, to cover our whole class of rings. We point out that the two subclasses considered do indeed overlap, so in particular some of our rings may be treated by any of the two techniques.

The first technique consists of imposing a restriction on $t$ that implies that $t$ is constant, that is, $t\in R$. In this case, $t=t(r)$ for all $r\in R$, and since we already have $t(r)\in\ZZ^{+}$, we are done.

The second technique adds a condition on $t$ implying that the value $t(r)=n$ does not depend on $p$ (equivalently, on $r$; recall that we are taking $p=x-r+1$). In other words, we want to force $t$ to be a constant polynomial function. By doing this, and assuming that the ring $R$ is infinite, we can apply \Cref{constantfpoly} to get $t\in R$, and again we obtain $t\in\ZZ^{+}$.

It is reasonable to expect that the technique showed in \Cref{singlepOydes} can be adapted in order to obtain the definability of the prime subring in other types of rings.

\subsection{The case in which every nonzero integer is invertible}

\noindent In this subsection we develop the first strategy discussed above. More concretely, we obtain the definability of $\ZZ^{+}$ in $R[x]$ when $R$ is a reduced indecomposable ring, provided the definability of a set between $\ZZ^{+}$ and $R$. Particularly, if we take this set as the set of units of $R[x]$ together with zero, this method accounts for all cases in which every nonzero integer in the ring is invertible. This improves the result of \cite{RobinsonR1951}*{\S 2}, which requires that $R$ be a characteristic zero integral domain that is first-order definable in the ring $R[x]$\footnote{This is the case if $R$ is a field or a local domain (see \Cref{ourlocal}): in the first case we have $R=\{0\}\cup R[x]^{*}$; in the second case, $R=\{p\in R[x]\colon p\in R[x]^{*}\textnormal{ or }p+1\in R[x]^{*}\}$.}\!.

\enlargethispage*{5mm}

\begin{proposition}\label{ZisdefinA}

Let $S=R[x]$, with $R$ a reduced indecomposable ring, and let $P$ be as in \Cref{defsTandUPV}. Given a definable subset $A$ of $S$ with $A\subseteq R$, we have that
\begin{align*}
\Theta_{A}(t)\colon\ \ t\in A\ \wedge\ \forall p\,\bigl(\,&p\in P\rightarrow\exists y\,\exists w\,[\,y\in\lpow(p)\\
&\hspace{-3pt}\wedge\ y-1=(p-1)w\ \wedge\ p-1\mid w-t\,]\,\bigr)
\end{align*}
defines the subset $\ZZ^{+}\cap A$. In particular, $\Theta_{A}$ defines $\ZZ^{+}$ whenever $A\supseteq\ZZ^{+}$.

\end{proposition}

\begin{proof}

With notation as in \Cref{Lbeta}, let $\beta=\psi$, where $\psi$ is given by \Cref{lpowformula}, so the subset $B_{p}$ of $S$ defined by $\beta(\cdot,p)$ is equal to $\lpow(p)$. Therefore, the subformula
\[\exists y\,\exists w\,[\,y\in\lpow(p)\ \wedge\ y-1=(p-1)w\ \wedge\ p-1\mid w-t\,]\]
of $\Theta_{A}$ is precisely the formula $\exists y\,\exists w\,L_{\beta}(t,p,y,w)$, with $L_{\beta}(t,p,y,w)$ as in \Cref{Lbeta}.

If $p\in P$, then $B_{p}=\lpow(p)=\pow(p)$ by item \subcref{TandUPV-b} of \Cref{TandUPV}; in particular, $\log_{p}B_{p}=\ZZ^{+}$, regardless of $p\in P$. Moreover, we have that $p-1$ is regular, by the definition of $P$. Thus, we are in the hypotheses of \Cref{singlepOydes-b}, which implies that $\exists y\,\exists w\,L_{\beta}(t,p,y,w)$ holds if and only if the following condition is satisfied:
\begin{equation}\label{Lbiff}
\textnormal{There exists}\ \ n_{p}\in\log_{p}B_{p}=\ZZ^{+}\ \ \textnormal{such that}\ \ p-1\mid t-n_{p}\,.\tag{$\ast$}
\end{equation}
If $t$ satisfies $\Theta_{A}$, then $t\in A$ by definition. Moreover, taking $p=x\in P$ and using \labelcref{Lbiff} we get some $n_{x}\in\ZZ^{+}$ and some $\ell\in R[x]$ such that $t-n_{x}=(x-1)\ell$\ . However, $t-n_x\in R$, because $t\in A\subseteq R$. Thus, evaluating at $x=1$ we conclude that necessarily $t-n_{x}=0$, and consequently $t=n_{x}\in\ZZ^{+}$.

Conversely, let $t=n\in\ZZ^{+}\cap A$. We want to show that $\Theta_{A}(t)$ holds. Obviously $t\in A$, and if $p\in P$, then the element $n_{p}=n$ satisfies $n_{p}\in\ZZ^{+}=\log_{p}B_{p}$ and $p-1\mid 0=t-n_{p}$, so that \labelcref{Lbiff} holds, and therefore $\exists y\,\exists w\,L_{\beta}(t,p,y,w)$ holds as well.\qedhere
\end{proof}

\begin{theorem}\label{Z+invertible}

Let $S=R[x]$, with $R$ a reduced indecomposable ring, and let $P$ be as in \Cref{defsTandUPV}. The formula
\begin{align*}
\Theta(t)\colon\ \ t\in\{0\}\cup S^{*}\ \wedge\ \forall p\,\bigl(\,&p\in P\rightarrow\exists y\,\exists w\,[\,y\in\lpow(p)\\
&\hspace{-3pt}\wedge\ y-1=(p-1)w\ \wedge\ p-1\mid w-t\,]\,\bigr)\,.
\end{align*}
defines the set $\ZZ^{+}\cap(\{0\}\cup S^{*})$, which contains $1_{S}$. In particular, $\Theta$ defines $\ZZ^{+}$ if and only if every nonzero element of $\ZZ^{+}$ is invertible.

\end{theorem}

\begin{proof}

Let $A=\{0\}\cup S^{*}$. \Cref{redunits} implies indeed that $A\subseteq R$, and therefore we can apply \Cref{ZisdefinA}, after observing that $\Theta=\Theta_{A}$.
\end{proof}

\begin{remark}\label{1inCsubsetZ+}

The fact that $\Theta$ defines a subset of $\ZZ^{+}$ containing $1_{S}$ in \emph{arbitrary} reduced indecomposable polynomial rings will play a crucial role at the end of the section, in the construction of a unified formula that works for all such rings.

\end{remark}

\subsection{The case of nonfields of characteristic zero}

\noindent In this subsection we develop the second strategy for defining $\ZZ^{+}$ discussed at the end of \Cref{approxexpo}, which works successfully for the case where the coefficient ring is a (reduced, indecomposable) nonfield of characteristic zero. Since \Cref{Z+invertible} covers, among others, the case in which the coefficient ring is a field or has positive characteristic (the latter by \Cref{0orprime}), the result of this subsection will settle all remaining cases.

By using definability of powers of constants with the constants themselves as parameters (\Cref{Phipowerconstant}), we can strengthen the formula $L_{\beta}$ (see \Cref{Lbeta}), as was made in the previous subsection, but in another manner, in order to get rid of the requirement of having a suitable definable set of constants in $R[x]$ for defining $\ZZ^{+}$.

Notice that this result implies, in particular, the definability of $\Z$ in the ring $\Z[x]$, which is announced in \cite{RobinsonR1951}*{\S\S 3a,3b}, but not directly proved\footnote{The author proves the definability of integers in quadratic rings, and claims that the method of his proof can be slightly modified in order to obtain the corresponding definability result in polynomial rings over the integers or over quadratic rings.} (see \cite{Nies2007}*{Theorem 7.13} for an alternative proof).

\begin{proposition}\label{lambdat=k}

Let $S=R[x]$, with $R$ a reduced indecomposable ring, and let $P$ be as in \Cref{defsTandUPV}. Let $\lambda$ be the three-variable formula defined by
\begin{align*}
\lambda(t,a,b)\colon\ \ \forall p\,[\,&p\in P\rightarrow\exists y\,\exists w\,(\,y\in\lpow(p)\\
&\hspace{-3pt}\wedge\ y-1=(p-1)w\ \wedge\ p-1\mid w-t\ \wedge\ p-a\mid y-b\,)\,]\,.
\end{align*}
Let $a\in R$ be such that all powers of $a$ are distinct. If $k\in\Zm$ is such that $\lambda(t,a,a^{k})$ holds, then $t=k$.

\end{proposition}

\begin{proof}

Our argument resembles closely that of the proof of \Cref{ZisdefinA}: with notation as in \Cref{Lbeta}, let $\beta=\psi$, where $\psi$ is given by \Cref{lpowformula}, so that the subset $B_{p}$ of $R[x]$ defined by $\beta(\cdot,p)$ is precisely $\lpow(p)$. Therefore, the subformula
\[\exists y\,\exists w\,(\,y\in\lpow(p)\ \wedge\ y-1=(p-1)w\ \wedge\ p-1\mid w-t\ \wedge\ p-a\mid y-b\,)\]
of $\lambda(t,a,b)$ is precisely the formula
\[\exists y\,\exists w\,[\,L_{\beta}(t,p,y,w)\ \wedge\ p-a\mid y-b\,]\,,\]
with $L_{\beta}(t,p,y,w)$ as in \Cref{Lbeta}. If $p\in P$, then $B_{p}=\lpow(p)=\pow(p)$ by item \subcref{TandUPV-b} of \Cref{TandUPV}; in particular, $\log_{p}B_{p}=\ZZ^{+}$. Moreover, we have that $p-1$ is regular, by the definition of $P$. Thus, we are in the hypotheses of \Cref{singlepOydes}.

Let $r\in R$ be fixed. We will show that $t(r)=k$. If $p=x-r+1$, then $p\in P$ by \Cref{automorphU}. Since $\lambda(t,a,a^{k})$ holds, there exist $y,w\in R[x]$ such that $p$ satisfies both the formula $L_{\beta}(t,p,y,w)$ and the condition $p-a\mid y-a^{k}$. In particular, \Cref{singlepOydes-a} grants the existence of an element $n\in\log_{p}B_{p}=\ZZ^{+}$ ($n$ possibly depends on $r$) such that $p-1\mid t-n$ and $y=p^{n}$.

Since $p-1=x-r$, the condition $p-1\mid t-n$ becomes $x-r\mid t-n$, which in turn is equivalent to have $t(r)=n$. Since we also have $p-a\mid y-a^{k}$ and obviously $p-a\mid p^{n}-a^{n}$ always holds, we conclude that $p-a$ divides $(y-a^{k})-(p^{n}-a^{n})=a^{n}-a^{k}$ (recall that $y=p^{n}$). Thus, there exists $\ell\in R[x]$ such that $a^{n}-a^{k}=(p-a)\ell=(x-r+1-a)\ell$. After evaluating at $x=r-1+a$ and taking into account that $a^{n}-a^{k}\in R$ (because $a\in R$), we get $a^{n}-a^{k}=0$. As all powers of $a$ are distinct, the equality $a^{n}=a^{k}$ forces $n=k$, hence $t(r)=n=k$, as desired.

Since $k$ is fixed and therefore does not depend on $r$, we have proven that if $\lambda(t,a,a^{k})$ holds, then the polynomial function induced by $t$ has constant value $k$. As all powers of $a$ are distinct, it follows that $R$ is infinite, so we can apply \Cref{constantfpoly} to conclude that $t=k$.
\end{proof}

\begin{theorem}\label{nonfieldZdef}

Let $S=R[x]$, with $R$ being a reduced indecomposable characteristic zero ring which is not a field. Let $U$ be as in \Cref{defsTandUPV}, and let $\Phi(\cdot,\cdot)$ be the formula given in \Cref{Phipowerconstant}, defining powers of constant elements, namely,
\begin{align*}
\Phi(t,a)\colon\ \ \forall p\,\forall q\,\bigl(\,&[\,p\in U\ \wedge\ q\in U\,]\rightarrow\exists y\,\exists z\,[\,y\in\lpow(p)\ \wedge\ z\in\lpow(q)\\
&\hspace{-3pt}\wedge\ p-a\mid y-t\ \wedge\ q-a\mid z-t\ \wedge\ p-q\mid y-z\,]\,\bigr)\,.
\end{align*}
If
\[\Upsilon(t)\colon\ \ \exists b\,[\,\Phi(b,2)\ \wedge\ \lambda(t,2,b)\,]\,,\]
with $\lambda$ as in \Cref{lambdat=k}, then $\Upsilon$ defines $\ZZ^{+}$ in $S$.

\end{theorem}

\begin{proof}

We have, by \Cref{Phipowerconstant}, that for any $a\in R$ the formula $\Phi(\cdot,a)$ defines the set $\pow(a)$. Therefore, if $\Upsilon(t)$ holds, then there exists a positive integer $k$ such that formula $\lambda(t,2,2^{k})$ holds. Since $R$ has characteristic zero, all powers of $2$ are distinct, and therefore we may take $a=2$ in \Cref{lambdat=k}, obtaining $t=k\in\ZZ^{+}$.

Conversely, if $t\in\ZZ^{+}=\Zm$ (recall that $\Char(R)=0$), say $t=n$, then it is easy to see that $\Upsilon(t)$ holds for the choice $b=2^{n}$: more specifically, the reader may check that the formula $\lambda(n,2,2^{n})$ holds by taking, for each $p\in P$ (where $P$ is defined as in \Cref{defsTandUPV}), the values $y=p^{n}$ and $w=w_{n}(p)$.
\end{proof}

\subsection{The unified formula (\texorpdfstring{``}{"}One Formula to define them all'')}\label{definethemall}

\noindent In the previous two subsections we have provided two techniques that define $\ZZ^{+}$ in two different cases (\Cref{Z+invertible,nonfieldZdef}). To sum up, let $\CH$ be the class of reduced indecomposable polynomial rings. Let $\CH_{1}$ be the subclass of rings in $\CH$ where every nonzero integer is invertible, and let $\CH_{2}$ be the subclass of rings in $\CH$ that may be expressed as $R[x]$, where $R$ is a nonfield of characteristic zero. By \Cref{0orprime}, if $S$ is a member of $\CH$ not belonging to $\CH_{1}$, then $S$ belongs to $\CH_{2}$, and this is equivalent to the following identity of classes:
\[\CH=\CH_{1}\cup\CH_{2}\,.\]
We remark that these subclasses do overlap: for example, the ring $R=\Q[s,t]/(st)$ (\Cref{Bpq}) is a reduced indecomposable nonintegral domain (hence a nonfield) of characteristic zero in which every nonzero integer is invertible. Therefore, any of the two techniques developed could be used to define $\ZZ^{+}$ in $R[x]$.

At this point of the paper we have already proven that $\ZZ^{+}$ (and, consequently, the whole prime subring) is definable in all reduced indecomposable polynomial rings. However, depending on whether we work over $\CH_{1}$ or $\CH_{2}$, we resorted to distinct formulas, that were denoted by $\Theta$ and $\Upsilon$, respectively, in order to write out the definition sought.

In what follows we merge $\Theta$ and $\Upsilon$ into a single formula, defining $\ZZ^{+}$ in any reduced indecomposable polynomial ring, covering this way the whole class $\CH$ uniformly. To this end, we begin by constructing an auxiliary sentence characterizing nonmembership in $\CH_{1}$, and therefore forcing membership in $\CH_{2}$.

\begin{lemma}\label{XiXi}

Let $S=R[x]$, with $R$ a reduced indecomposable ring. Let $C=\ZZ^{+}\cap(\{0\}\cup S^{*})$ be the set defined by the formula $\Theta$ as in \Cref{Z+invertible}, and define
\[\Xi\,\colon\ \ \exists t\,(\,t\in C\ \wedge\ t+1\notin C\,)\,.\]
Then $C=\ZZ^{+}$ if and only if $\Xi$ does not hold. Moreover, if $\Xi$ holds in $S$, then $R$ is a nonfield of characteristic zero.

\end{lemma}

\begin{proof}

By \Cref{Z+invertible} we have that $C$ is a subset of $\ZZ^{+}$ containing $1$, and therefore $C=\ZZ^{+}$ if and only if $C$ is closed under the successor function $t\mapsto t+1$, which is equivalent to negating $\Xi$, proving the first assertion. For the second assertion, if $R$ is a field or $R$ has positive characteristic, then every nonzero integer in $S$ is invertible (by \Cref{0orprime} in the latter case). Therefore $C$ coincides with $\ZZ^{+}$ in these cases, and so $\Xi$ is false.
\end{proof}

\noindent What follows is the main result of our work: there is a formula defining the prime subring in all reduced indecomposable rings $R[x]$, regardless of the coefficient ring $R$. As mentioned in \Cref{1inCsubsetZ+}, we stress how the result of \Cref{Z+invertible} plays a critical role in the proof of our final claim, for it guarantees that $1\in C\subseteq\ZZ^{+}$, regardless of the coefficient ring $R$.

\begin{theorem}\label{finalboss}

Let $S=R[x]$, with $R$ a reduced indecomposable ring. Let
\[\Omega(t)\colon\ \ [\,\neg\,\Xi\ \wedge\ \Theta(t)\,]\ \vee\ [\,\Xi\ \wedge\ \Upsilon(t)\,]\,,\]
where $\Theta$ and $\Upsilon$ are the formulas given by \Cref{Z+invertible,nonfieldZdef}, respectively, and $\Xi$ is given by \Cref{XiXi}. We have that $\Omega$ defines the set $\ZZ^{+}$ in $S$.

\end{theorem}

\begin{proof}

Observe that
\[\Omega(t)\colon\begin{cases}
t\in C,&\ \textnormal{if }\,\Xi\,\textnormal{ is false}\,;\\
\Upsilon(t),&\ \textnormal{if }\,\Xi\,\textnormal{ is true}\,,
\end{cases}\]
with $C$ as in \Cref{XiXi}. If $\Xi$ is false, then $C=\ZZ^{+}$ by \Cref{XiXi}. Otherwise, $R$ is a nonfield of characteristic zero, again by \Cref{XiXi}, hence $\Upsilon$ defines $\ZZ^{+}$ by \Cref{nonfieldZdef}. In either case, we have proven that $\Omega(t)$ holds if and only if $t\in\ZZ^{+}$.
\end{proof}

\subsection{Interpretability of positive integers}\label{interpretability}

\noindent Undecidability is often obtained in literature as a consequence of interpretability. In \Cref{undecidable}, however, we have shown a relatively simple technique that proves undecidability without resorting to interpretability. While undecidability of the full theory of $R[x]$ is a concrete computational goal, interpretability of the ring $\Z$ in $R[x]$ is more of a theoretical issue. Since the latter is a stronger property than the former, we devote this subsection to its proof.

Beyond the machinery considered in \Cref{undecidable}, here we further need to find a first-order property that detects whether two powers of elements of a certain definable set, possibly of different bases, have the same exponent. We show how to exploit the notion of logical powers to provide a two-dimensional interpretation of the structure consisting of the set $\Zm$ of positive integers with the usual sum and divisibility, in $S=R[x]$, whenever $R$ is a reduced indecomposable ring. As we already observed in \Cref{undecidable} while proving undecidability of the full theory of our rings, the product can be written, on $\Zm$, in terms of divisibility and sum. Therefore, interpreting equality, sum and divisibility will suffice for our purposes.

We will use three different techniques, according to whether $R$ is either $1)$ a nonfield, $2)$ a characteristic zero field or $3)$ a field of positive characteristic, the case $2)$ being actually a simple adaptation of the exponent-extracting technique developed in \Cref{approxexpo}.

We follow the notation of $n$-dimensional interpretations that can be found in \cite{Hodges1993}*{Section 5.3}, which consists, in our context, of a surjective map $f$ from a suitable definable subset of $S^{n}$ to $\Zm$ and the definition of formulas that are, on this subset, equivalent to the relations on $\Zm$ given by $x=y,x+y=z$ and $x\mid y$, respectively.

The reader may notice that the result on definability of $\ZZ^{+}$, just proven in the previous subsection, would itself provide a one-dimensional interpretation of positive integers in the characteristic zero case, where $\ZZ^{+}$ coincides with $\Zm$. This may simply be achieved by taking, for a surjection, the identity on the definable subset $\ZZ^{+}=\Zm$ of $S$. However, reduced indecomposable rings of characteristic $p>0$ are clearly not covered by this technique, which is only able to define integers modulo $p$, and other methods are therefore necessary to achieve our goal.

Formulas interpreting sum and divisibility will be defined along the lines of \Cref{undecidable} and exploited in our proof of interpretability, but the really nontrivial argument introduced in this subsection is the interpretation of equality in $S$, especially for the subclasses $1)$ and $3)$ mentioned above. In a nutshell, for powers $q^{n}$ of the elements of suitable sets, we need a way of extracting exponents $n$ without ``bringing them down'', that is, talking of $n\in\Zm$ in some abstract level, without necessarily ending up talking about $n\cdot 1_{S}\in S$. We start the subsection by introducing two important subsets of our polynomial rings.

\begin{definition}\label{defWL}

Let $S=R[x]$, with $R$ reduced indecomposable. Consider the set $V$ as in \Cref{defsTandUPV}, and define the sets
\[W=\{p\in R[x]\colon p-u\in V\textnormal{ for all }u\in \{ 0\}\cup R^{*}\}\]
and
\[L=\{vx+r\colon v\in R^{*},r\in R\}\,.\]

\end{definition}

\noindent In the case where $R$ is a field, the reader may notice that $L$ coincides with the set of degree $1$ polynomials.

Observe that the set $W$ is definable, because $V$ is definable and $R^{*}=S^{*}$ (the latter by \Cref{redunits}, since $R$ is reduced), and that we may easily rewrite $W$ in the form $W=\{p\in V\colon p-u\in V\textnormal{ for all }u\in R^{*}\}$; in particular, elements $p\in W\subseteq V$ inherit (see \Cref{undecidable}) the properties

\begin{itemize}

\item $\lpow(p)=\pow(p)$;

\item All the powers of $p$ are distinct; and

\item For any positive integers $m$ and $n$, we have $p^{m}-1\mid p^{n}-1$ if and only if $m\mid n$.

\end{itemize}
\vspace{3mm}
\noindent The next result shows the relationship between the sets $W$ and $L$.

\begin{proposition}\label{WandL}

With notation as in \Cref{defWL}, we have the following:

\begin{enumproposition}

\item\label{WandL-a} $L\subseteq W$.

\item\label{WandL-b} If $p\in W\smallsetminus L$, then the constant term of $p$ does not belong to $\{0\}\cup R^{*}$.

\item\label{WandL-c} If $R$ is a field, then $W=L$. In particular, $L$ is definable in this case.

\end{enumproposition}

\end{proposition}

\begin{proof}\leavevmode

\begin{enumerate}

\item Just recall that $L\subseteq V$ by \Cref{automorphU} and observe that $L$ is closed under translations by a constant. Alternatively, as $W$ is definable, the reader may observe that elements of $L$ are automorphic images of $x\in W$, but in order to prove that $x\in W$ the argument in \Cref{automorphU} must be applied anyway.

\item Let $p\in W$ and write $p=p_{0}+xg$, with $p_{0}\in R$.
We want to show that, if $r=p_{0}\in\{0\}\cup R^{*}$, then $p\in L$. Since $p\in W$, supposing $r\in \{0\}\cup R^{*}$, we have $xg=p-r\in V$, and since elements of $V$ are irreducible, $g$ must be a unit, say $g=v$, proving $p=vx+r\in L$.

\item The claim follows from item \subcref{WandL-b} and the fact that, if $R$ is a field, then $R=\{0\}\cup R^{*}$.\qedhere

\end{enumerate}

\end{proof}

\noindent Next, we define a class of first-order expressible equivalence relations which will be useful for our purposes:
\begin{definition}

Let $W$ be defined as above and $p,q\in W$. Recall that for elements of $U$ (and therefore for elements of $W$) positive powers coincide with logical powers. We say that $p$ and $q$ are \textbf{$\bm{1}$-connected} (and write $p\sim q$) if, taking any $y=p^{m}\in\lpow(p)$ and $z=q^{n}\in\lpow(q)$ such that $p-q\mid y-z$, we must have $m=n$. We say that $p$ and $q$ are \textbf{$\bm{k}$-connected} (and write $p\sim_{k}q$) if there exist $p_{0},\ldots,p_{k}\in W$ with $p_{0}=p$ and $p_{k}=q$, such that $p_{i-1}\sim p_{i}$ for $i=1,\ldots,k$. We say that $p$ and $q$ are \textbf{connected} if they are $k$-connected for some positive integer $k$.

\end{definition}

\begin{remark}\label{conectivityproperties}

It is trivial to see that being $1$-connected is by definition (and by properties of divisibility) a symmetric relation, and it is also reflexive because all elements of $V$ (and therefore all elements of $W$) have distinct positive powers. This makes $k$-connectivity reflexive (take $p_{0}=p_{1}=\cdots=p_{k}=p$) and symmetric (take $p'_{i} =p_{k-i}$ and use symmetry of $\sim$), which in turn implies that connectivity is reflexive and symmetric. Furthermore, transitivity of the connectivity relation may be proven by merging connecting sequences, concluding that being connected is an equivalence relation. Finally, observe also that, if $p\sim_{k_{1}}q$ and $k_{1}\leq k_{2}$, then $p\sim_{k_{2}}q$, by extending $p_{0}=p,\ldots,p_{k_{1}}=q$ identically by $p_{j}=q$ on the right, for $j=k_{1},\ldots,k_{2}$.

\end{remark}

\begin{lemma}\label{Antilider}

Let $R$ be a reduced indecomposable ring and let $p,f\in R[x]$ be such that $p\mid f$. If the lowest degree nonzero coefficient of $f$ is a unit, then the lowest degree nonzero coefficient of $p$ must be a unit as well.

\end{lemma}

\begin{proof}

Notice that this result is an analogous version of \Cref{basic-d}, obtained by replacing ``leading coefficients'' by ``lowest degree nonzero coefficients'', and the corresponding proof can be adapted to meet our purpose. Alternatively, for any polynomial $g=g_{m}x^{m}+g_{m-1}x^{m-1}+\cdots+g_{k}x^{k}\in R[x]$, where $m\geq k$, $g_{k}$ and $g_{m}$ being nonzero, we define the \textbf{reciprocal polynomial} of $g$ as
\[\bighat{g}=g_{m}+g_{m-1}x+\cdots+g_{k}x^{m-k}\,.\]
One can easily check the equality $\bighat{g}=x^{m}g(\frac{1}{x})$ in the ring $R[x,\frac{1}{x}]$ of Laurent polynomials over $R$ (\cite{Eisenbud1995}*{Exercise 2.17}),
and observe that the lowest degree nonzero coefficient of $g$ is precisely the leading coefficient of $\bighat{g}$. Therefore, if $p$ divides $f$, say $f=pq$, then multiplying the equality $f(\frac{1}{x})=p(\frac{1}{x})q(\frac{1}{x})$ by $x^{\deg(f)+\deg(p)+\deg(q)}$ yields the equality
\[x^{\deg(p)+\deg(q)}\cdot\bighat{f}=x^{\deg(f)}\cdot\bighat{p}\bighat{q}\]
in $R[x]$. Now, the lowest degree nonzero coefficient of $f$, which we are supposing to be a unit, coincides with the leading coefficient of $\bighat{f}$, which is also the leading coefficient of $x^{\deg(p)+\deg(q)}\cdot\bighat{f}$, and the latter polynomial is a multiple of $\bighat{p}$. Therefore, by \Cref{basic-d}, the leading coefficient of $\bighat{p}$ must be a unit, and since this is the lowest degree nonzero coefficient of $p$, the result follows.
\end{proof}

\begin{lemma}\label{nonunitconnected}

Let $p,q\in R[x]$, with $R$ being a reduced indecomposable nonfield. If $p\in W$ and the constant term of $p$ does not belong to $\{0\}\cup R^{*}$, then $p\sim x$. In particular, all elements of $W\smallsetminus L$ are $1$-connected with $x$ \textup{(}and therefore, trivially, they are $2$-connected with $x$\textup{)}.

\end{lemma}

\begin{proof}

$p-x\mid p^{m}-x^{n}$, together with $p-x\mid p^{m}-x^{m}$, implies $p-x\mid x^{m}-x^{n}$. The lowest degree nonzero coefficient of $p-x$ coincides with its constant term (as the latter is nonzero) and therefore is a nonunit. Therefore, by \Cref{Antilider} the lowest degree nonzero coefficient of $x^{m}-x^{n}$ must be also a nonunit, and this can only happen when $m=n$ (otherwise the required coefficient would be $\pm1$). The last part of the statement follows from \Cref{WandL-b} and from the last part of \Cref{conectivityproperties}.
\end{proof}

\begin{theorem}\label{connectivityfinal}

If $R$ is not a field, then all elements of $W$ are $2$-connected with the element $x$, and consequently $4$-connected with each other.

\end{theorem}

\begin{proof}

If $p\in W\smallsetminus L$, then we already know that $p\sim x$, by \Cref{nonunitconnected}. The same result also accounts for the case $p=vx+r\in L$, when $r$ is a nonzero nonunit. Therefore we are left with the case
\[p=vx+w,\textnormal{ with }v\in R^{*}\textnormal{ and }w\in\{0\}\cup R^{*}\,.\]
Fix a nonzero nonunit $r$ of $R$ (recall that $R$ is not a field). Following \Cref{sumnonunits}, we have two possibilities: either there exists $u\in R^{*}$ with $u-1\notin\{0\}\cup R^{*}$, or else any element of $R$ is a sum of two nonunits.

In the first case define $g=uvx+r$: we claim that $p\sim g$ and $g\sim x$, so that $p\sim_{2}x$. For the first connection, if $m\neq n$, then the leading coefficient of $p^{m}-g^{n}$ is either $v^{m}$ or $-(uv)^{n}$, hence a unit. Since the leading coefficient of $p-g$ is a (nonzero) nonunit, namely $(1-u)v$, it follows from \Cref{basic-d} that $p-g$ cannot divide $p^{m}-g^{n}$. The second connection follows from \Cref{nonunitconnected}.

In the second case, write $w=a+b$, where $a$ and $b$ are constant nonunits. We can also suppose that $a$ and $b$ are both nonzero: if $w$ is a unit, this is automatic; otherwise, $w=0$ and we may take $a=r,b=-r$. Defining $g=vx+a$, we claim that $p\sim g$ and $g\sim x$, so that $p\sim_{2}x$. For the first connection notice that, if $m\neq n$, then the constant nonunit $p-g=b$ cannot divide $p^{m}-g^{n}$, because otherwise it would divide all its coefficients, in particular its leading coefficient, which is a unit, namely $v^{m}$ or $-v^{n}$: a contradiction. The second connection follows from \Cref{nonunitconnected}.

We have proven that, in any case, $p\sim_{2}x$. The remaining part of the claim follows easily by gluing, for any two $p,q\in W$, connecting sequences for $p\sim_{2}x$ and $x\sim_{2}q$ to get $p\sim_{4}q$.
\end{proof}

\noindent Next, we are going to define an important set, which will be used at a starting point to interpret (actual) positive integers in $R[x]$.

\begin{definition}\label{defDelta}

Let $R$ be a reduced indecomposable ring, and consider $W\subseteq R[x]$, the set introduced in \Cref{defWL}. We define $\Delta$ as the following subset of $R[x]\times R[x]$:

\[\Delta=\{(p,y)\colon p\in W,y\in\lpow(p)\}\,.\]

\end{definition}

\noindent Notice that $\Delta$ is a definable subset of $R[x]\times R[x]$ (because $W$ is definable, and the sets $\lpow(p)$ are definable using $p$ as a parameter), and it coincides with the set $\{(p,p^{n})\colon p\in W,n\in\Zm\}$, because $W\subseteq V\subseteq U$, and elements $p\in U$ satisfy $\lpow(p)=\pow(p)$ (\Cref{TandUPV-b}).

When $R$ is not a field, we have proven that any two elements of $W$ are $4$-connected, and this is sufficient to construct a two-dimensional interpretation of $(\Zm,+,\mid\,)$ in $R[x]$ with domain $\Delta\subseteq R[x]\times R[x]$:

\begin{theorem}\label{interprnonfield}

Let $R$ be a reduced indecomposable ring which is not a field. There is a two-dimensional interpretation $\Gamma$ of $(\Zm,+,\mid\,)$ in the ring $R[x]$ consisting in the following:
\vspace{3mm}
\begin{itemize}[labelindent=13pt,itemindent=0em,itemsep=2mm,leftmargin=!]

\item The domain $\Delta\subseteq R[x]\times R[x]$ as described in \Cref{defDelta};

\item The surjective map $f\colon\Delta\to\Zm$ given by $(p,p^{n})\mapsto n$;

\item For interpreting the equality on $\Zm$, the four-variable formula
{\small
\begin{alignat*}{3}
(p,y)=_{\Gamma}(q,z)\ \colon&&\ \ \forall p_{0}\,\forall p_{1}\,\forall p_{2}\,\forall p_{3}\,\forall p_{4}\,\Biggl(\ &\biggl[\,p_{0}=p\ \wedge\ p_{4}=q\ \ &\wedge&\ \ \bigwedge_{i=1}^{3}p_{i}\in W\,\biggr]\xrightarrow{\hspace{5mm}}\\
&&\,\exists y_{0}\,\exists y_{1}\,\exists y_{2}\,\exists y_{3}\,\exists y_{4}\,&\biggl[\,y_{0}=y\ \wedge\ y_{4}=z\ \ &\wedge&\ \ \Bigl(\,\bigwedge_{i=1}^{4}p_{i}-p_{i-1}\mid y_{i}-y_{i-1}\,\Bigr)\\
&&&&\wedge&\ \ \Bigl(\,\bigwedge_{i=1}^{3}y_{i}\in\lpow(p_{i})\,\Bigr)\,\biggr]\ \Biggr)\,;
\end{alignat*}
}

\item For interpreting the sum on $\Zm$, the six-variable formula
{\small
\begin{align*}
\bigr[(p,y)+(q,z)=(r,w)\bigl]_{\Gamma}\ \colon\ \ \exists z'\,\exists w'\,[\,&z'\in\lpow(p)\ \wedge\ w'\in\lpow(p)\\
&\hspace{-3pt}\wedge\ (p,z')=_{\Gamma}(q,z)\ \wedge\ (p,w')=_{\Gamma}(r,w)\ \wedge\ yz'=w'\,]\,;
\end{align*}
}

\item For interpreting the divisibility on $\Zm$, the four-variable formula
{\small
\[(p,y)\bigm|_{\Gamma}(q,z)\ \colon\ \ \exists z'\,[\,z'\in\lpow(p)\ \wedge\ (p,z')=_{\Gamma}(q,z)\ \wedge\ y-1\mid z'-1\,]\,.\]
}

\end{itemize}

\end{theorem}

\begin{proof}

Clearly, as powers coincide with logical powers for elements of $W$, and are all distinct (see the proof of \Cref{undecidability}), the map $f$ is well defined and surjective. We claim that, for all $(p,y),(q,z)\in\Delta$, the formula $(p,y)=_{\Gamma}(q,z)$ holds true if and only if $f(p,y)=f(q,z)$. This amounts to saying that, for $p,q\in W$ and $m,n$ positive integers, $(p,p^{m})=_{\Gamma}(q,q^{n})$ if and only if $m=n$.

Let $p,q\in W$. If $m=n$ and $p_{0},\ldots,p_{4}\in W$ are such that $p=p_{0}$ and $q=p_{4}$, then we may take $y_{i}=p_{i}^{n}$ and it can be easily checked that $(p,p^{n})=_{\Gamma}(q,q^{n})$ holds true through such choices. Conversely, suppose $(p,p^{m})=_{\Gamma}(q,q^{n})$ is true for $p,q\in W$. Since we know by \Cref{connectivityfinal} that $p\sim_{4}q$, we may take a connecting sequence $p_{0},\ldots,p_{4}$, with $p_{0}=p,p_{4}=q$, that is, with the property that for any $i=1,\ldots,4$, if $p_{i}-p_{i-1}$ divides $p_{i}^{n_{i}}-p_{i-1}^{n_{i-1}}$, then $n_{i}=n_{i-1}$. Since $(p,p^{m})=_{\Gamma}(q,q^{n})$ is true, we have that for this connecting sequence there exist $y_{i}=p_{i}^{n_{i}}$, for $i=0,\ldots,4$, with $y_{0}=p^{m}$ (and therefore $n_{0}=m$), $y_{4}=q^{n}$ (and therefore $n_{4}=n$) and $p_{i}-p_{i-1}\mid p_{i}^{n_{i}}-p_{i-1}^{n_{i-1}}$ for $i=1,\ldots,4$. Connectedness implies $n_{i}=n_{i-1}$ for $i=1,\ldots,4$, proving $n=n_{0}=\cdots=n_{4}=m$, as required.

Once we proved that the relation $=_{\Gamma}$ works for our purpose, it is easy to check that the same occurs for the sum: for $p,q,r\in W$ and positive integers $m,n$ and $k$, we have that $m+n=k$ if and only if $\bigl[(p,p^{m})+(q,q^{n})=(r,r^{k})\bigl]_{\Gamma}$ holds. To prove this, suppose first $m+n=k$ and write $y=p^{m},z=q^{n},w=r^{k}$. Taking $z'=p^{n},w'=p^{k}$, it is straightforward to see that $(p,z')=_{\Gamma}(q,q^{n}),\ (p,w')=_{\Gamma}(r,r^{k})$ and $yz'=w'$. For the converse, observe that if $\bigl[(p,p^{m})+(q,q^{n})=(r,r^{k})\bigl]_{\Gamma}$ is true, then there must exist $z'=p^{n_{0}},w'=p^{k_{0}}$ such that $(p,z')=_{\Gamma}(q,q^{n}),\ (p,w')=_{\Gamma}(r,r^{k})$ and $p^{m}z'=w$. Since the relation $=_{\Gamma}$ was proven to be equivalent to equality of exponents, we get $n_{0}=n$ and $k_{0}=k$, and therefore the condition $p^{m}z'=w'$ becomes $p^{m+n}=p^{k}$, which proves $m+n=k$, since the powers of $p$ are distinct.

Finally, in order to prove that the relation $\bigm|_{\Gamma}$ plays the role meant for it, we shall take $m,n\in\Zm$ and $p,q\in W$, and prove that $m\mid n$ if and only if $(p,p^{m})\bigm|_{\Gamma}(q,q^{n})$. To this end, we use the fact that for any $h\in W\subseteq V$, we have $m\mid n$ if and only if $h^{m}-1\mid h^{n}-1$ (see \Cref{undecidable}). If $m$ divides $n$, we may take $z'=p^{n}$: by what was proven, we have $(p,z')=_{\Gamma}(q,q^{n})$ and clearly $p^{m}-1$ divides $z'-1=p^{n}-1$, making $(p,p^{m})\bigm|_{\Gamma}(q,q^{n})$ true. Conversely, if $(p,p^{m})\bigm|_{\Gamma}(q,q^{n})$ is true, then there exists $z'=p^{k}\in\lpow(p)$ such that $(p,z')=_{\Gamma}(q,q^{n})$ and $p^{m}-1\mid z'-1$. The first condition, as we proved, leads to $k=n$, whereas the second condition implies $m\mid k$.
\end{proof}

\noindent We have therefore proved interpretability of the structure $(\Zm,+,\mid\,)$ in reduced indecomposable polynomial rings over a nonfield, by defining and exploiting the connectivity relation on $W$. We now turn our attention to the case where $R$ is a field.\\[3mm]

\begin{remark}\label{aspas}

In the definability section, the case of fields was accounted for by the formula $\Theta$ defined in \Cref{Z+invertible}. Let $R$ be a field and $S=R[x]$. Given $\ell\in L=W\subseteq P$, fix $y=\ell^{n}\in\lpow(\ell)$, with $n\in\Zm$. By reasoning as in the proof of \Cref{ZisdefinA}, we conclude that $n\cdot 1_{S}$ is the unique element $t$ in $S$ satisfying the conditions $t\in\{0\}\cup S^{*}$ and 
\begin{equation}\label{ambiguo}
\textnormal{there exists}\ \ w\in S\ \ \textnormal{such that}\ \ \ell-1\mid w-t\ \ \textnormal{and}\ \ y-1=(\ell-1)w\,.\tag{$\spadesuit$}
\end{equation}
Regularity of $\ell-1$ implies that $w$ is uniquely determined by the equality $y-1=(\ell-1)w$, namely $w=w_{n}(\ell)$; see \Cref{wnp}. Therefore we may rewrite \labelcref{ambiguo} unambiguously in the form ``$\ell-1\bigm|\frac{y-1}{\ell-1}-t$''.

\end{remark}

\noindent Recall that $L=W$ is definable whenever $R$ is a field (\Cref{WandL-c}), and that $\ZZ^{+}=\Zm$ in characteristic zero. These facts, together with \Cref{aspas}, allow to construct an interpretation in the case of fields of characteristic zero, and allows to prove definability, with parameter varying in the definable set $L$, of certain special sets that will help setting interpretability in the positive characteristic case too. First, let us observe how, in the case where $R$ is a field of characteristic zero, one can easily construct a two-dimensional interpretation of the structure $(\Zm,+,\mid\,)$ whose domain and associated surjection are the same used for the case of nonfields.

\begin{theorem}\label{interprcharzero}

Let $R$ be a field of characteristic zero and let $\Delta\subseteq R[x]\times R[x]$ be as in \Cref{defDelta}. There is a two-dimensional interpretation of $(\Zm,+,\mid\,)$ in $R[x]$ with domain $\Delta$, and with associated surjective map sending $(\ell,\ell^{n})$ to $n\in\Zm$.

\end{theorem}

\begin{proof}

We set the interpretation of equality to be
\[(\ell_{1},y_{1})=_{\Gamma}(\ell_{2},y_{2})\ \colon\ \ \exists t\,\biggl[\,t\in\{0\}\cup S^{*}\footnote{Although the condition ``$t\in\{0\}\cup S^*$'' seems odd to interpret $\Zm$ in the characteristic zero case (it could be changed to ``$t\in S^*$''), such uniform writing also works in the positive characteristic case, where the set $\{0\}\cup S^{*}$ (and not $S^{*}$) contains the images of all integers in the prime subring.}\!\ \wedge\ \biggl(\,\ell_{1}-1\biggm|\frac{y_{1}-1}{\ell_{1}-1}-t\,\biggr)\ \wedge\ \biggl(\,\ell_{2}-1\biggm|\frac{y_{2}-1}{\ell_{2}-1}-t\,\biggr)\,\biggr]\,.\]
By \Cref{aspas}, this relation is equivalent to saying that $y_{1}$ and $y_{2}$ have the same exponents as powers of $\ell_{1}$ and $\ell_{2}$, respectively.
Interpretations of sum and divisibility are defined by using the same formulas as in the case of nonfields, with the definition just given replacing the former definition of $=_{\Gamma}$\footnote{Since this difference does not affect the proof of interpretability, we may safely consider the interpretations for sum and divisibility as ``essentially equal'' to those of \Cref{interprnonfield}.}\!: the proof of their equivalence to sum and divisibility on $\Zm$ is identical to that provided in \Cref{interprnonfield}.
\end{proof}

\noindent We now move to the case where $R$ is a field of positive characteristic $p>0$, where a two-dimensional interpretation will be given on a subset of $\Delta$ and for which the definition of equality requires more effort to be established.

\begin{leftbar}
\noindent\textbf{Notational warning.} In the rest of this subsection, since $p$ is the name chosen for the characteristic, and $L=W$ is our definable subset of $R[x]$ of reference, we will refer to the general element of $L=W$ as $\ell$ instead of $p$. Notice that this convention is already used in the previous case (fields of characteristic zero).
\end{leftbar}

\noindent We begin by introducing, for any fixed polynomial $\ell$, two important sets of polynomials, which will turn out to be definable (with parameter) whenever $\ell$ is a degree $1$ polynomial, that is, an element of $L$.

\begin{definition}\label{defmandppow}

Let $R$ be a field of characteristic $p>0$. For $\ell\in R[x]$, we define the following sets:
\vspace{2mm}
\begin{itemize}[labelindent=18pt,itemindent=0em,leftmargin=!,itemsep=2mm]

\item $\mpow(\ell)$ is the set of powers of $\ell$ whose exponent is a multiple of $p$;

\item $\ppow(\ell)$ is the set of powers of $\ell$ whose exponent is a positive power of $p$.

\end{itemize}

\end{definition}

\pagebreak

\noindent Notice that $\ppow(\ell)\subseteq\mpow(\ell)\subseteq\pow(\ell)$ holds trivially for any $\ell\in R[x]$, and recall that $\lpow(\ell)=\pow(\ell)$ whenever $\ell\in L$.

\begin{lemma}\label{MandPPOW}

Let $R$ be a field of characteristic $p>0$. Given $y\in R[x]$ and $\ell\in L$, we have:
\vspace{2mm}
\begin{enumlemma}[labelindent=23pt,itemindent=0em,leftmargin=!,itemsep=3mm]

\item\label{MandPPOW-a} $y\in\mpow(\ell)$ if and only if $y\in\lpow(\ell)$ and $(\ell-1)^{2}\mid y-1$.

\item\label{MandPPOW-b} $y\in\ppow(\ell)$ if and only if
\vspace{1mm}
\begin{itemize}[labelindent=0pt,itemindent=0em,itemsep=1mm,leftmargin=!]

\item $y\in\mpow(\ell)$, and

\item For any $y'\in\lpow(\ell)$ with $y'\neq y$ and $y'-1\mid y-1$, we have $y'\in\mpow(\ell)$.

\end{itemize}

\end{enumlemma}
\vspace{2mm}
Consequently, given $\ell\in L$, we have that both $\mpow(\ell)$ and $\ppow(\ell)$ are definable using $\ell$ as parameter.

\end{lemma}

\begin{proof}\leavevmode

\begin{enumerate}

\item As $\ell\in L$, we have observed that $\mpow(\ell)\subseteq\lpow(\ell)$. Therefore it is enough to show that, for $y\in\lpow(\ell)$, we have that $y\in\mpow(\ell)$ if and only if $(\ell-1)^{2}$ divides $y-1$. Take any $n\in\Zm$. By recalling, as in \Cref{wnpcongn}, that $\frac{\ell^{n}-1}{\ell-1}=w_{n}(\ell)=(\ell-1)[\ell^{n-2}+2\ell^{n-3}+\cdots+(n-2)\ell+(n-1)]\footnote{For $n=1$, the expression in brackets is understood to be an empty sum.}\!+n\cdot 1_{S}$, or the argument in \Cref{aspas}, we get that
$\ell-1$ divides $\frac{\ell^{n}-1}{\ell-1}-n\cdot 1_{S}$, where regularity of $\ell-1$ allows the slight abuse of notation. This fact, together with the fact that linear polynomials divide no nonzero constant, allows us to argue that
\[(\ell-1)^{2}\mid\ell^{n}-1\iff\ell-1\biggm| \dfrac{\ell^{n}-1}{\ell-1}\iff\ell-1\mid n\cdot 1_{S}\iff n\cdot 1_{S}=0\ ,\]
and to conclude, for any $y=\ell^{n}\in\lpow(\ell)$, that $(\ell-1)^{2}\mid y-1$ if and only if $p\mid n$.

\item Observe that, for all $m,n\in\Zm$, we have $\ell^{m}-1\mid\ell^{n}-1$ if and only if $m\mid n$ (because $\ell\in V$), and that powers of $p$ can be characterized as those multiples of $p$ whose only positive divisors, except for $1$, are multiple of $p$. These facts, together with the result of item \subcref{MandPPOW-a}, yield the claim.\qedhere

\end{enumerate}

\end{proof}

\noindent Before giving the interpretation of the structure $(\Zm,+,\mid\,)$ in the last case, a technical lemma is needed, expressing a relationship between two powers of linear polynomials, sharing a suitable exponent, in the prime characteristic case.

\begin{lemma}\label{ele1k=uele2k}

Let $R$ be a field of characteristic $p>0$, and suppose that $k=p^{m}$, with $m\geq 1$. Given $\ell_{1},\ell_{2}\in L\subseteq R[x]$, there exist $u\in R^{*}$ and $\rho\in R$ such that $\ell_{1}^{k}=u\ell_{2}^{k}+\rho$.

\end{lemma}

\begin{proof}

We may write $\ell_{i}=a_{i}x+b_{i}$, with $a_{1},a_{2}\neq 0$. By setting $s=b_{1}-\frac{a_{1}b_{2}}{a_{2}}$ and $v=\frac{a_{1}}{a_{2}}$, we get $\ell_{1}=v\ell_{2}+s$. Notice that raising elements of $R[x]$ to $k$ is the $m$-th iterate of the Frobenius endomorphism, hence an additive map. Therefore we may write $\ell_{1}^{k}=u\ell_{2}^{k}+\rho$, where $u=v^{k}\in R^{*}$ and $\rho=s^{k}$.
\end{proof}

\noindent The following result is a hint to the interpretation of equality that we intend to build for the positive characteristic field case.

\pagebreak

\begin{proposition}\label{interprcharpequality}

Let $R$ be a field of characteristic $p>0$ and let $\ell_{1},\ell_{2}\in L$. Suppose that $k_{1},k_{2}$ are powers of $p$ with positive exponents, and set $y_{i}=\ell_{i}^{k_{i}}$ \textup{(}that is, $y_{i}\in\ppow(\ell_{i})$ for $i=1,2$\textup{)}. We have $k_{1}=k_{2}$ if and only if there exist $u\in R^{*}$ and $\rho\in R$ such that $y_{1}=uy_{2}+\rho$.

\end{proposition}

\begin{proof}

The ``only if'' part follows from \Cref{ele1k=uele2k}. For the converse, just observe that, if two polynomials differ by a constant, then they must have the same degree, and that multiplying by a unit does not alter the degree.
\end{proof}

\begin{theorem}\label{interpretcharp}

Let $R$ be a field of characteristic $p>0$ and consider the polynomial ring $S=R[x]$. There is a two-dimensional interpretation of $\Zm$ in $S$ with domain
\[\nabla=\bigl\{(\ell,y)\colon\ell\in L,y\in\ppow(\ell)\bigr\}\,,\]
and with associated surjective map sending $(\ell,\ell^{p^{n}})$ to $n\in\Zm$.

\end{theorem}

\begin{proof}

The interpretation of equality is given by
\[(\ell_{1},y_{1})=_{\Gamma}(\ell_{2},y_{2})\ \colon\ \ \exists u\,\exists\rho\,(\,u\in S^{*}\ \wedge\ \rho\in\{0\}\cup S^{*}\ \wedge\ y_{1}=uy_{2}+\rho\,)\,.\]
The subset $\nabla\subseteq S\times S$ is definable without parameters, because the sets $\ppow(\ell)$ are definable using $\ell$ as a parameter (by \Cref{MandPPOW-b}) and $L=W$ is definable. Moreover, the map $(\ell,\ell^{p^{n}})\mapsto n$ is well defined because the powers of elements in $L$ are distinct, and surjective by the definition of the sets $\ppow(\ell)$.

Equivalence between $=_{\Gamma}$ and the equality of positive integers is just the content of \Cref{interprcharpequality} and, once this is granted, the equivalence of sum and divisibility on $\Zm$ with the maps defined follows exactly as in the proof of \Cref{interprnonfield}.
\end{proof}

\noindent By gathering \Cref{interprnonfield,interprcharzero,interpretcharp} together with the fact that the product on $\Zm$ may be defined in terms of sum and divisibility, we may conclude that for any reduced indecomposable ring $R$ there exists a two-dimensional interpretation of the structure $(\Zm,+,\,\cdot\,)$ in $R[x]$. From this interpretation, it is straightforward to build an interpretation of the ring $\Z$ in such cases.

\section{Appendix: miscellaneous considerations}\label{sectionAppendix}

\noindent In this section we place additional findings concerning our work, which are not strictly necessary to prove the main result but have appeared as side-results, by-products and optimal generalizations and may be useful in future attempts at extending our claims to wider classes of rings or in more general contexts.

\subsection{A generalization of the uniform exponent-extracting technique and an application in the noncommutative context}\label{Feixistico}

\noindent In this subsection we will give another version of \Cref{singlepOydes}, whose claim is presented in terms of maps between sets of first-order formulas, mapping simple formulas into complex ones, the latter being built out of the former. The theorem is stated for unital rings, possibly noncommutative. In the commutative case, the technique works as an alternative argument for defining the prime subring in rings of our subclass $\CH_{1}$ (see \Cref{definethemall}), for it relies crucially on the existence of a definable set containing the positive integers and, for at least one element $r\in R$, no two polynomials taking the same value at $r$. If $S$ is in $\CH_{1}$, then the set $\{0\}\cup S^{*}$ has this property because it is a set of constant elements. We also provide, as an application of this more general criterion, an example of a noncommutative ring in which the prime subring is definable by using an extra constant symbol.

\begin{theorem}\label{Oydesfeixe}

Let $F_{1}$ \textup{(}resp.~$F_{2}$\textup{)} denote the set of one-variable \textup{(}resp.~two-variable\textup{)} first-order formulas in the language of unital rings. Consider the function $L\colon F_{1}\times F_{2}\to F_{2}$ given by
\[[L(\alpha,\beta)](t,s)\colon\ \ \alpha(t)\ \wedge\ \exists y\,\exists w\,[\,\beta(y,s)\ \wedge\ y-1=(s-1)w\ \wedge\ s-1\mid w-t\,]\,,\]
where ``$\mid$'' denotes left divisibility. Let $S$ be a unital \textup{(}not necessarily commutative\textup{)} ring and let $Q$ be a subset of $S$ such that for all elements $p\in Q$:

\begin{enumerate}[label=\texttt{\textup{(\arabic*)}}]

\item $p-1$ is left cancelable.

\item There is a definable set $A_{p}$ such that for all $n\in\ZZ^{+}$ there is a unique element $s_{p}(n)\in A_{p}$ right congruent\footnotemark\! to $n$ modulo $p-1$.

\item There is a two variable formula $\CB(p)\in F_{2}$ such that $[\CB(p)](\cdot,p)$ defines a subset $B_{p}$ of $\pow(p)$.

\end{enumerate}

\noindent The following hold:

\begin{enumtheorem}

\item\label{Oydesfeixe-a} For all $p\in Q$, the set $s_{p}(\log_{p}B_{p})$ is definable, by using an extra symbol for $p$.

\item\label{Oydesfeixe-b} If $Q$ is definable, then the sets $\cap_{p\in Q}s_{p}(\log_{p}B_{p})$ and $\cup_{p\in Q}s_{p}(\log_{p}B_{p})$ are definable in the language of rings.

\item\label{Oydesfeixe-c} For any $p\in Q$ that further satisfies $\ZZ^{+}\subseteq A_{p}$, we have that $\log_{p}B_{p}$ is definable in the language $(0,1,+,\cdot,p)$ with an extra symbol for $p$. Moreover, if $B_{p}=\pow(p)$, then $\ZZ^{+}$ is definable in $S$ in the language $(0,1,+,\cdot,p)$.

\item\label{Oydesfeixe-d} If $Q$ is definable, all elements of $Q$ satisfy $\ZZ^{+}\subseteq A_{p}$, and at least one of them satisfies $B_{p}=\pow(p)$, then $\ZZ^{+}$ is definable in $S$ in the language of rings. Moreover, if all $p\in Q$ satisfy $B_{p}=\pow(p)$, then we may define $\ZZ^{+}$ either as a union or as an intersection of equal copies of $\ZZ^{+}$, quantified over $Q$.

\end{enumtheorem}

\end{theorem}

\footnotetext{By right congruence we mean congruence modulo the right ideal generated by $p-1$.}

\begin{proof}\leavevmode

\begin{enumerate}

\item By properties \texttt{\textup{(2)}} and \texttt{\textup{(3)}}, if $\CA(p)$ denotes the formula defining $A_{p}$, then the maps $\CA\colon Q\to F_{1}$ and $\CB\colon Q\to F_{2}$ may be merged into their direct product function $(\CA,\CB)\colon Q\to F_{1}\times F_{2}$, sending $p$ to $(\CA(p),\CB(p))$, and the composition $[L\circ(\CA,\CB)(p)]$ may partly evaluate at $p$ in the second variable. It is easy to observe that $[L\circ(\CA,\CB)(p)](\cdot,p)=[L(\CA(p),\CB(p))](\cdot,p)$ defines the set of elements $t$ of $A_{p}$ such that $\exists y\,\exists w\,L_{\CB(p)}(t,p,y,w)$ holds, where $L_{\CB(p)}$ is the formula $L_{\beta}$ defined in \Cref{Lbeta}, and where ``$\mid$'' stands for left divisibility.

As $B_{p}$ is a subset of $\pow(p)$, the exponent-extracting function $\log_{p}$ may be defined on this set and we are in the condition to apply \Cref{singlepOydes} (the reader may check that the proof of this result works identically for noncommutative rings, for it only requires condition \texttt{\textup{(1)}} on cancelability of $p-1$ from the left). By \Cref{singlepOydes-b} the above formula defines the set $A_{p}\cap[\log_{p}B_{p}+(p-1)S]$. An element $t\in S$ belongs to this set if and only if it lies in $A_{p}$ and it is right congruent to some element of $\log_{p}B_{p}$ modulo $p-1$, that is, there exists a positive integer $n$ in $\log_{p}B_{p}$ (elements of $\log_{p}B_{p}$ are always positive integers) right congruent to $t$ modulo $p-1$. However, by our hypothesis on $s_{p}$, the only such element is $s_{p}(n)$. Therefore this set is given by the images of elements of $\log_{p}B_{p}$ under $s_{p}\colon\ZZ^{+}\to A_{p}$, as required.

\item If $Q$ is definable, then quantifying universally or existentially on $p\in Q$ gives definitions of the required sets in the language of rings, for it eliminates the need of a symbol for $p$ in the defining language.

\pagebreak

\item As any integer is trivially congruent with itself modulo $p-1$, if an integer belongs to $A_{p}$, then it must coincide with its image under $s_{p}$ by condition \texttt{\textup{(2)}}. Therefore $s_{p}$ coincides with the inclusion of $\ZZ^{+}$ in $A_{p}$ and the first part of the claim follows by item \subcref{Oydesfeixe-b}. Clearly, if $B_{p}=\pow(p)$, then $\log_{p}B_{p}=\ZZ^{+}$ and the second part of the claim follows as well.

\item By merging items \subcref{Oydesfeixe-b} and \subcref{Oydesfeixe-c} we get a formula defining the union, over $Q$, of the sets $\log_{p}B_{p}$. But these are all subsets of $\ZZ^{+}$ and at least one of them coincides with it. Therefore such union must be $\ZZ^{+}$. The last claim is straightforward, as one is able to use either the universal or the existential quantification described in the proof of item \subcref{Oydesfeixe-b}.\qedhere

\end{enumerate}

\end{proof}

\begin{remark}

The reader may notice how items \subcref{Oydesfeixe-a} and \subcref{Oydesfeixe-c} of the previous theorem, as well as its hypotheses, have a totally elementwise formulation, making sense for one element $p$ of $S$ at a time, with no need for defining a set $Q$. However, we preferred to include the set $Q$ (and items \subcref{Oydesfeixe-b} and \subcref{Oydesfeixe-d}) within the same result and let the reader think of the elementwise formulation as the case $Q=\{p\}$ (notice that in \subcref{Oydesfeixe-a} and \subcref{Oydesfeixe-c} the set $Q$ needs not be, \emph{a priori}, definable).

\end{remark}

\noindent The result above can be interpreted as a statement on a diagram of trivial set bundles on $S$. More specifically, if we consider the following setting:

\begin{itemize}[itemsep=2mm]

\item $L\colon F_{1}\times F_{2}\to F_{2}$ is defined as in the previous theorem;

\item $\CP_{\Def}(S)$ denotes the set of definable subsets of $S$;

\item $\TS\colon F_{1}\to\CP_{\Def}(S)$ is the ``truth set'' function, sending a one-variable formula $\alpha$ to the set $\alpha^{S}=\{s\in S\colon\alpha(s)\textnormal{ is true}\}$;

\item $\Psi\colon F_{2}\times S\to\CP_{\Def}(S)$ is the function given by $\Psi(\varphi,s)=\TS\bigl(\varphi(\cdot,s)\bigr)=\varphi(\cdot,s)^{S}$,

\end{itemize}

\noindent then we may look at the following diagram of trivial set bundles on $S$:
\[\begin{tikzcd}
F_{1}\times F_{2}\times S \arrow[r, "L\,\times\,\operatorname{id}_{S}"] \arrow[d] & F_{2}\times S \arrow[d] \\
\CP_{\Def}(S)\times\CP_{\Def}(S)\times S \arrow[r,dashrightarrow] & \CP_{\Def}(S)\times S
\end{tikzcd}
\]
where the right downward arrow is given by the map $(\Psi,\pi_{S})\colon F_{2}\times S\to\CP_{\Def}(S)\times S$, sending $(\varphi,s)$ into $(\Psi(\varphi,s),s)=(\varphi(\cdot,s)^{S},s)$ (that is, the direct product map of $\Psi$ with the projection on the factor $S$), and the left downward arrow is $\TS\times\,(\Psi,\pi_{S})$.

Under this point of view, the theorem can be interpreted in terms of the corresponding maps of sections of these bundles over a subset $Q\subseteq S$: it amounts to saying that, if we restrict the diagram to $Q\subseteq S$ and to the section $(\CA,\CB)$, that is, if we compose from the top left with $\bigl(\CA,\CB,\subseteq\bigr)\colon Q\to F_{1}\times F_{2}\times S$ sending $p$ to $(\CA,\CB,p)$, then we may take the dashed arrow to be $(X,Y,p)\mapsto s_{p}\bigl(\log_{p}(Y)\bigr)$, making the restricted diagram commute, provided that, for all $p\in P,\CB(p)\subseteq\pow(p)$ and the quotient map $S\to S/(p-1)$ coequalizes the inclusion map $\ZZ^{+}\subseteq S$ and $j_{p}\circ s_{p}$, where $j_{p}$ is the inclusion map $A_{p}\subseteq S$. The last condition is equivalent to saying that the lower triangle in the diagram below commutes for all $p\in Q$\footnote{In the diagram we preferred the notation $\pi_{Q}$ to $\pi_{3}$ to denote the projection onto the factor $Q$ of $F_{1}\times F_{2}\times Q$ (as well as all projections onto $Q$) because $Q$ appears as the second, instead of third, factor of the product $F_{2}\times Q$, in the rightmost top corner of the diagram, and we wanted to use the same notation for the rightmost downward arrow. Nevertheless, we had to opt for $\pi_{2}$, rather than $\pi_{\CP_{\Def}(S)}$ for projection onto the second factor of $\CP_{\Def}(S)\times\CP_{\Def}(S)\times Q$, since the first two factors in this case are equal.}\!.
\newlength{\col}
\setlength{\col}{\widthof{\,$\CP_{\Def}(S)\times\CP_{\Def}\bigl(\pow(p)\bigr)\times Q$\,}}
\[\scalebox{0.65}{%
\begin{tikzcd}[column sep=large, font=\large,every label/.append style = {font = \large}, ampersand replacement=\&]
\parbox[t]{\widthof{\,$Q$\,}}{\centering $Q$ \\[-1ex] \rotatebox{90}{\small $\in$} \\[-1ex]$p$} \arrow[rr, "{\bigl(\CA,\CB,\operatorname{id}_{Q}\bigr)}"'] \arrow[rrdd, bend right, mapsto, start anchor={[xshift=2.5ex]south west}, end anchor={[xshift=7ex,yshift=3ex]south west}] \& \& F_{1}\times F_{2}\times Q \arrow[ll, bend right, dashrightarrow, rotate=90, start anchor={north west}, end anchor={north east}, "\pi_{Q}"'] \arrow[rrr, "L\times\operatorname{id}_{Q}"] \arrow[dd, end anchor={north}, "\TS\times\,(\Psi\bigl|_{F_{2}\times Q}{,}\pi_{Q})"] \& \& \& F_{2}\times Q \arrow[dd,"(\Psi\bigl|_{F_{2}\times Q}{,}\pi_{Q})"] \\\\
\& \& \parbox[t]{\col}{\centering $\CP_{\Def}(S)\times\CP_{\Def}(S)\times Q$ \\[0.8ex] \rotatebox{90}{\huge $\subseteq$} \\ ${\CP_{\Def}(S)\times\CP_{\Def}\bigl(\pow(p)\bigr)\times Q}$ \\[0.8ex] \rotatebox{90}{\huge $\in$} \\ $\bigr(\underbracket[1pt][1pt]{A_{p}},{B_{p}},p)$} \arrow[rrr, dashrightarrow, "\exists", "{\bigl(\,s_{\pi_{Q}(\cdot)}\,\circ\,\log_{\pi_{Q}(\cdot)}\,\circ\,\pi_{2}\,,\pi_{Q}\bigr)}"'] \arrow[rd, start anchor={[xshift=-1.0ex]south}, end anchor={[yshift=0ex]}, "\subseteq\,(j_{p})"] \& \& \& \CP_{\Def}(S)\times Q \\
\& \& \mbox{\huge $\circlearrowleft$} \& S \arrow[rd, end anchor={[xshift=-2.5ex]north}, "\pi"] \& \& \\
\& \ZZ^{+} \arrow[ruu, end anchor={[xshift=8.5ex]}, "s_{p}"] \arrow[rr,"\subseteq"] \& \& S \arrow[r, end anchor={[xshift=7.5ex]}, "\pi"] \& \hspace{8ex}S/(p-1) \&
\end{tikzcd}
}\]
In some sense, if we find a way of measuring some homotopic-like information about how much the square diagram fails to be commutative (or better, to admit a bottom arrow ``closing it'' and making it commute), a set $Q$ whose elements all have the properties listed in \Cref{Oydesfeixe} is one that is ``small enough to trivialize the bundle diagram''. Properties of \Cref{Oydesfeixe} can be expressed in terms of the section map induced by the leftmost downward arrow landing into a subset of the sections over $Q$, say $\mathcal{E}(Q)\subseteq(\CP_{\Def}(S)\times\CP_{\Def}(S)\times S)(Q)$, characterized in turn by the fact that images of its sections, considered as maps $Q\to\CP_{\Def}\times\CP_{\Def}\times S$, belong to a special subset whose elements $(\CA_{p},\CB_{p},p)$ share the properties described in the hypotheses of \Cref{Oydesfeixe}. We can build up this way the subbundle $\mathcal{E}\subseteq \CP_{\Def}(S)\times\CP_{\Def}(S)$ and claim that if, locally (over $Q$), the leftmost arrow in the square lands inside $\mathcal{E}$, then the rightmost arrow also lands inside a subbundle $\mathcal{F}\subseteq\CP_{\Def}\times S$ whose definable sets on the first coordinate are all contained in $\ZZ^{+}$, and the diagram below can be closed by bottom arrows defined by means of the $\log$ functions:
\[\begin{tikzcd}
(F_{1}\times F_{2}\times S)\,(Q) \arrow[r, "L\,\times\,\operatorname{id}_{S}"] \arrow[d] & (F_{2}\times S)\,(Q) \arrow[d] \\
\mathcal{E}\,(Q) \arrow[r,dashrightarrow] & \mathcal{F}\,(Q)
\end{tikzcd}
\]

\begin{corollary}

Let $S=R[x]$ be a reduced irreducible \textup{(}unital, commutative\textup{)} polynomial ring belonging to the subclass $\CH_{1}$ defined in \Cref{definethemall}, that is, suppose that all nonzero integers are invertible in $S$. Denote by $A$ the definable set $\{0\}\cup S^{*}$ and set
\[Q_{A}=\{t\in P\colon t\textnormal{ divides no nonzero difference of two elements of }A\}\,,\]
where $P$ is the set defined in \Cref{TandUPV}. Then $Q_{A}$ is clearly definable and we may use it to define $\ZZ^{+}$, like in \Cref{Oydesfeixe}, as either a union or an intersection of identical copies of itself quantified over $Q_{A}$, that is, by using either one of the formulas
\[\forall p\,\bigl(\,p\in Q_{A}\rightarrow[\,t\in A\ \wedge\ \exists y\,\exists w\,L_{\psi}(t,p,y,w)\,]\,\bigr)\]
or
\[\exists p\,\bigl(\,p\in Q_{A}\ \wedge\ [\,t\in A\ \wedge\ \exists y\,\exists w\,L_{\psi}(t,p,y,w)\,]\,\bigr)\,,\]
where $L_{\psi}$ is defined as in \Cref{Lbeta} and $\psi$ is given by \Cref{lpowformula} and defines logical powers.

\end{corollary}

\begin{proof}

We want to show that we are in the last case of item \subcref{Oydesfeixe-d} in \Cref{Oydesfeixe} and we may apply the theorem, by taking constant functions $\CA(p)=A$ and $\CB(p)=\beta$ and by using $Q_{A}$ in place of the set $Q$ in \Cref{Oydesfeixe-d}. Let $p\in Q_{A}$. Clearly $p-1$ is regular, and therefore cancelable, because $p\in P$. As $\ZZ^{+}\subseteq A_{p}=\{0\}\cup S^{*}$, the additional property on $Q_{A}$ that $p-1$ divides no difference of elements of $A=A_{p}$ implies that $p-1$ cannot divide a difference between a positive integer and an element of $A_{p}$, yielding condition \texttt{\textup{(2)}}. Finally, $B_{p}=\psi(\cdot,p)^{S}=\lpow(p)=\pow(p)$, since $p\in P$ and it is therefore contained in $\ZZ^{+}$, which, together with the fact that $Q_{A}$ is definable, also grants the further specific hypotheses of \Cref{Oydesfeixe-d}.
\end{proof}

\noindent As an application of \Cref{Oydesfeixe-c} we provide an example of a noncommutative ring in which the prime subring is definable in the language of rings expanded with an extra constant symbol.

\begin{example}

Let $D$ be an integral domain, and let $q\in D\smallsetminus\{0\}$. The \textbf{quantum plane over $\bm{D}$ with parameter $\bm{q}$}, denoted by $S=D_{q}[x,y]$, is defined as the quotient of the free noncommutative $D$-algebra over two generators $x$ and $y$, by the unique relation $yx=qxy$. Alternatively, the ring $S$ is the free $D$-module generated by the monomials $x^{m}y^{n}$, with $m,n\in\N$, so their elements are of the form $f=\sum_{(m,n)\in\N^{2}}f_{(m,n)}x^{m}y^{n}$, with $f_{(m,n)}\in D$ and $f_{(m,n)}\neq 0$ for finitely many pairs $(m,n)$. Multiplication is given by $(x^{m}y^{n})(x^{r}y^{s})=q^{rn}x^{m+r}y^{n+s}$, and extended by $D$-linearity. (see \cite{Kassel1995}*{Chapter IV} for details on the case in which $D$ is a field.)

We claim that $\ZZ^{+}$ is definable in $(S,0,1,+,\cdot,x)$, provided that every nonzero integer is invertible in $S$: for example when $D$ is a field or it has positive characteristic (because in the latter case the characteristic is a prime number, as $D$ is an integral domain), or in other cases such as $D=\Q[t]$. Toward our aim, we first endow $\N\times\N$ with the \textbf{degree lexicographic order}: $(m,n)\prec(m',n')$ if either $m+n<m'+n'$, or both $m+n=m'+n'$ and $m<m'$. We observe that $\preccurlyeq$ defines a well-ordering of $\N\times\N$, which satisfies the property
\begin{equation}\label{prec}
(m',n')+(r',s')\prec(m,n)+(r,s),\ \ \textnormal{whenever}\ \ (m',n')\preccurlyeq(m,n)\ \ \textnormal{and}\ \ (r',s')\prec(r,s).\tag{$\dagger$}
\end{equation}
For $f\in S\smallsetminus\{0\}$, let $\max(f)$ (resp.~$\min(f)$) be the maximum (resp.~the minimum) pair $(m,n)$, with respect to $\preccurlyeq$, with $f_{(m,n)}\neq 0$. Given $g,h\in S\smallsetminus\{0\}$, if $(m',n'),(r',s')$ are such that $g_{(m',n')}\neq 0$ and $h_{(r',s')}\neq 0$, then $(m',n')\preccurlyeq\max(g)$ and $(r',s')\preccurlyeq\max(h)$, so by \labelcref{prec} we have $(m',n')+(r',s')\preccurlyeq\max(g)+\max(h)$, and equality occurs precisely when $(m',n')=\max(g)$ and $(r',s')=\max(h)$.

The above reasoning implies that, for $\vec{b}\in\N\times\N$, all summands of
\[(gh)_{\vec{b}}=\sum_{(m',n')+(r',s')=\vec{b}} q^{r'n'}g_{(m',n')}h_{(r',s')}\]
\pagebreak

\noindent vanish if $\vec{b}\succ\max(g)+\max(h)$, and moreover, in the same way, if we set $\max(g)=(m,n),\max(h)=(r,s)$, then
\[(gh)_{\max(g)+\max(h)}=q^{rn}g_{\max(g)}h_{\max(h)}\neq 0\,.\]
This shows that $gh\neq 0$ and $\max(gh)=\max(g)+\max(h)$ (the equality $\min(gh)=\min(g)+\min(h)$ is proven in a similar way); in particular, every nonzero element in $S$ is left cancelable. Consequently, if $g$ left-divides $f$, say $f=gh$, then $g,h\neq 0$ (as $f\neq 0$) and therefore $\max(g)\leq\max(g)+\max(h)=\max(f)$.

Let $x^{\ell}=gh$, with $\ell\geq 0$ and $g,h\in S$, and let $\max(h)=(r_{0},s_{0})$. If $(m,n)$ satisfies $g_{(m,n)}\neq 0$, then we have
\begin{align*}
\min(x^{\ell})=&\,\min(g)+\min(h)\\
\preccurlyeq&\,(m,n)+\max(h)\\
=&\,(m+r_{0},n+s_{0})\\
\preccurlyeq&\,\max(g)+\max(h)\\
=&\,\max(x^{\ell})\,;
\end{align*}
but $\max(x^{\ell})=\min(x^{\ell})=(\ell,0)$, and so $(m+r_{0},n+s_{0})=(\ell,0)$, which forces to have $m=\ell-r_{0}$ and $n=0$. This proves that any left divisor $g$ of $x^{\ell}$ is of the form $g=ax^{m}$, for some $a\in D\smallsetminus\{0\}$ and some $m$ with $0\leq m\leq\ell$. The same result holds for right divisors (with an entirely similar proof), and so $h=bx^{j}$ for some $b\in D\smallsetminus\{0\}$ and some $j$ with $0\leq j\leq\ell$. Therefore $x^{\ell}=(ax^{m})(bx^{j})=abx^{m+j}$, which implies $ab=1$ (and $m+j=\ell$), hence $a\in D^*$. In particular, the case $\ell=0$ implies that every left invertible element in $S$ belongs to $D^{*}$.

We claim that $f\in\pow(x)$ precisely when $\psi(f,x)$ holds, $\psi$ being given by \cref{lpowformula}, and where all the clauses of divisibility are interpreted as left divisibility (under this convention, left divisors of $1$ are precisely the right units, also called \emph{left invertible}). In fact, if $f\in\pow(x)$, say $f=x^{\ell}$, with $\ell\geq 1$, then the reasoning above shows that every left divisor of $f$ is of the form $ax^{m}$, with $a\in D^{*}$ and $0\leq m\leq\ell$, and therefore is either left invertible ($m=0$) or else a left multiple of $x$ ($m>0$); since we also have $x\mid f$ and $x-1\mid f-1$, one of the implications follows.

Conversely, let $f\in S$ be such that $\psi(f,x)$ holds. Since $\max(x-1)=(1,0)\succ(0,0)$, it follows that $x-1$ does not left-divide any element of $D\smallsetminus\{0\}$; in particular we must have $f\neq 0$ (otherwise we would have $x-1\mid f-1=-1$). Let $\max(f)=(p,q)$. If $x^{t}\mid f$, then $(t,0)=\max(x^{t})\preccurlyeq\max(f)=(p,q)$, hence $t\leq p+q$. Thus, there exists a greatest $k\geq 1$ with $x^{k}\mid f$, so by proceeding analogously to the proof of \Cref{lpowgeneral-a} we conclude that $f=ux^{k}$, with $u\in S$ being a left invertible element satisfying $x-1\mid u-1$, hence $u\in D^{*}$, and as we already showed that $x-1$ does not left-divide any nonzero constant, it follows that $u=1$, yielding $f=x^{k}$.

With notation as in \Cref{Oydesfeixe}, let $B_{x}=\pow(x)$. The proof above shows that $B_{x}$ is a definable subset of $\pow(x)$, using $x$ as a parameter; moreover, the element $x-1$ is left cancelable (as every nonzero element of $S$). As mentioned in the previous paragraph, we have that $A_{x}\subseteq D$ and that $x-1$ does not left-divide any element of $D\smallsetminus\{0\}$, so in particular $x-1$ left-divides no nonzero difference between a positive integer and an element of $A_{x}$, so that all conditions of \Cref{Oydesfeixe} are satisfied. As for the further conditions of item \subcref{Oydesfeixe-c}, since we are assuming that all nonzero integers in $S$ are invertible, the definable set $A_{x}=\{0\}\cup S^{*}$ satisfies $\ZZ^{+}\subseteq A_{x}$, and $B_{x}=\pow(x)$ by definition. Applying \Cref{Oydesfeixe-c} with $p=x$, we conclude that the set $\log_{x}B_{x}=\ZZ^{+}$ is definable in $(S,0,1,+,\cdot,x)$, that is, using $x$ as a parameter.

\end{example}

\pagebreak

\subsection{Further properties of logical powers}

\begin{proposition}\label{lpowfurther}

Let $S$ be a ring, and let $p\in S$.

\begin{enumproposition}

\item\label{lpowfurther-a} We have $\lpow(0)=\{0\}$ if $S$ is a field and $\lpow(0)=\varnothing$ otherwise.

\item\label{lpowfurther-b} We have $\lpow(p)\neq\varnothing$ if and only if $p\in\lpow(p)$. If $S$ is not a field and $\lpow(p)\neq\varnothing$, then $p\neq 0$.

\item\label{lpowfurther-c} If $p$ is a unit, then $\lpow(p)$ is the set $\{(p-1)g+1\colon g\in S\}$, which contains every integer power of $p$.

\end{enumproposition}

\end{proposition}

\begin{proof}\leavevmode

\begin{enumerate}

\item If $f\in\lpow(0)$, then $0$ divides $f$, so necessarily $f=0$, which shows that $\lpow(0)\subseteq\{0\}$. Moreover, we have $0\in\lpow(0)$ if and only if every divisor of $0$ is a unit or a multiple of $0$. Since (trivially) every element of $S$ divides $0$, it follows that $0\in\lpow(0)$ precisely when every element of $S$ is a unit or $0$, that is when $S$ is a field.

\item One implication is clear; for the converse, let $f\in\lpow(p)$ be fixed. We trivially have $p\mid p$ and $p-1\mid p-1$, and if $g\in S$ divides $p$, then it also divides $f$ (because $p\mid f$), so $g$ is either a unit or a multiple of $p$. Thus $p\in\lpow(p)$, and if $S$ is not a field, then $p\neq 0$ by item \subcref{lpowfurther-a}.

\item If $p$ is a unit, then all the conditions for membership in $\lpow(p)$ are automatically fulfilled by any element $f$, except possibly for $p-1\mid f-1$, and so $f\in\lpow(p)$ precisely when $f-1=(p-1)g$ for some $g\in S$; in particular, since for all natural $n$ we have that $p-1$ divides both $p^{n}-1$ and $-p^{-n}(p^{n}-1)=p^{-n}-1$, it follows that $\lpow(p)$ contains every integer power of $p$.\qedhere

\end{enumerate}

\end{proof}

\noindent Given a unit $p$ in a ring $S$, let us denote its set of integer powers by $\pZ$. If $p$ has infinite multiplicative order, then $\lpow(p)$ strictly contains every set of the form $\{p^{j}\colon j\geq n_{0}\}$, with $n_{0}\in\Z$ (because $\lpow(p)\supseteq\pZ$, by \Cref{lpowfurther-c} above).

The following result shows that, under certain conditions, the stronger strict inclusion $\lpow(p)\supset\pZ$ holds:

\begin{proposition}\label{lpowunitlarge}

If $p$ is a unit in a ring $S$ such that $p-1$ is regular, then $\lpow(p)$ strictly contains the set of all integer powers of $p$.

\end{proposition}

\begin{proof}

By \Cref{lpowfurther-c} we already have $\lpow(p)=\{(p-1)s+1\colon s\in S\}\supseteq\pZ$. We prove below that $\lpow(p)=\pZ$ implies that $S$ is finite. As regular elements in finite rings are invertible, the equality $\lpow(p)=\pZ$ would imply
\[0=(p-1)[-(p-1)^{-1}]+1\in\lpow(p)=\pZ\,,\]
and this contradiction shows the desired result.

For $n\geq 1$, let $w_{n}(p)$ be given by \Cref{wnp}, and let $w_{0}(p)=0$. Assume that $\{(p-1)s+1\colon s\in S\}=\pZ$. We claim that
\begin{equation}\label{S=wn}
S=\{w_{n}(p)\colon n\geq 0\}\cup\{-p^{-n}w_{n}(p)\colon n\geq 1\}\,.\tag{$\vardiamondsuit$}
\end{equation}
In fact, given $s\in S$ we have $(p-1)s+1=p^{\pm n}$ for some $n\geq 0$. Since $p-1$ is regular, the mapping $s\mapsto(p-1)s+1$ is injective. This fact, together with the equalities
\begin{align*}
(p-1)\cdot w_{n}(p)+1=&\,p^{n};\\
(p-1)\cdot [-p^{-n}w_{n}(p)]+1=&\,p^{-n}\,,
\end{align*}
valid for $n\geq 0$, shows that \labelcref{S=wn} holds.

In particular we have $p=w_{n}(p)$ for some $n\geq 0$ or $p=-p^{-n}w_{n}(p)$ for some $n\geq 1$. The equalities $p=w_{0}(p)=0$ and $p=w_{2}(p)=p+1$ are clearly impossible, and $p=w_{1}(p)=1$ would imply $p-1=0$, contradicting the regularity of $p-1$. Therefore we have $w_{n}(p)-p=0$ for some $n\geq 3$, or $p^{n+1}+w_{n}(p)=0$ for some $n\geq 1$.

In both cases, there exists $f\in\Z[x]$ monic such that $pf(p)+1=0$. Writing $p=(p-1)+1$, and expanding and rearranging terms, we find that $p-1$ is a root of some monic polynomial $g\in\Z[x]$, that is, $g(p-1)=0$. We have $g=x^{k}\bighat{g}$ for some $k\geq 0$ and some $\bighat{g}\in\Z[x]$ monic with nonzero constant term, so $\bighat{g}$ is of the form $xh+1-j$, with $h\in\Z[x]$ monic and $j\in\Z$ with $j\neq 1$. Regularity of $p-1$, together with $0=g(p-1)=(p-1)^{k}\bighat{g}(p-1)$, yields
\[0=\bighat{g}(p-1)=\underbrace{[(p-1)\cdot h(p-1)+1]}_{\in\,\lpow(p)=\pZ}-j\,.\]
Consequently we have $p^{m}=j$ for some $m\neq 0$ (because $j\neq 1$). If $j=-1$, then $p^{2m}=1$, and therefore $\pZ$ is finite. As already mentioned, the mapping $s\in S\mapsto(p-1)s+1$ is injective, hence $S$ has the same cardinality as $\lpow(p)=\pZ$, which shows that $S$ is finite in this case.

Finally, assume $j\neq 0,1,-1$, and let $d=|m|\geq 1$. Any nonconstant monic factor $Q$ in $\Z[x]$ of $x^{d}-j$ is the product of factors $x-\mu_{i}$, with $\mu_{i}\in\mathbb{C}$ satisfying $\mu_{i}^{d}=j$. In particular we have $|\mu_{i}|=|j|^{1/d}>1$, so the constant term $Q(0)$ of $Q$ satisfies $|Q(0)|=|\prod_{i}\mu_{i}|>1$. This implies that $Q$ does not divide $xf+1$ in $\Z[x]$ (because divisibility of polynomials in $\Z[x]$ implies divisibility, in $\Z$, of the corresponding constant terms).

The reasoning above shows that $x^{d}-j$ and $xf+1$ have no nonconstant common factors in $\Z[x]$, and a well-known consequence of this is that they are coprime in $\Q[x]$, so we may write $(x^{d}-j)G+(xf+1)H=1$, with $G,H\in\Q[x]$. On the other hand, since the constant and leading coefficients of $xf+1$ are equal to $1$, it follows that the reciprocal polynomial of $xf+1$ (see the proof of \Cref{Antilider}) is also of the form $x\bighat{f}+1$, with $\bighat{f}\in\Z[x]$ monic. We may perform exactly the same reasoning with $x\bighat{f}+1$ instead of $xf+1$ in the previous paragraph, obtaining $(x^{d}-j)\bighat{G}+(x\bighat{f}+1)\bighat{H}=1$, for some $\bighat{G},\bighat{H}\in\Q[x]$.

Recall that $p$ is root of $xf+1$, so that $p^{-1}$ is root of its reciprocal polynomial, that is $p^{-1}\bighat{f}(p^{-1})+1=0$. Multiplying the equalities obtained in the previous paragraph by some $b\in\Zm$ such that $bG,bH,b\bighat{G},b\bighat{H}\in\Z[x]$, and evaluating at $p$ and $p^{-1}$, respectively, we obtain
\begin{align*}
b\cdot 1_{S}=&\,(p^{d}-j)\cdot\bigl(bG\bigr)(p)+[pf(p)+1]\cdot\bigl(bH\bigr)(p)\\
=&\,(p^{d}-j)\cdot\bigl(bG\bigr)(p)\,;\\
b\cdot 1_{S}=&\,(p^{-d}-j)\cdot\bigl(b\bighat{G}\bigr)(p^{-1})+[p^{-1}\bighat{f}(p^{-1})+1]\cdot\bigl(b\bighat{H}\bigr)(p^{-1})\\
=&\,(p^{-d}-j)\cdot\bigl(b\bighat{G}\bigr)(p^{-1})\,.
\end{align*}
Recalling that $p^{m}=j$ and $m=\pm d$, we conclude that $b\cdot 1_{S}=0$, because at least one of the two expressions above vanishes; this shows that $\ZZ$ is finite. Since $p$ is a root of the monic polynomial $xf+1\in\Z[x]$, it follows that $p$ is $\ZZ$-integral, and therefore $\ZZ[p]$ is a finitely generated $\ZZ$-module (\cite{Hungerford1980}*{Theorem VIII.5.3}), hence a finite set. Finally, from $p^{-1}=-f(p)\in\ZZ[p]$ we get $\ZZ[p,p^{-1}]=\ZZ[p]$, and since $S=\ZZ[p,p^{-1}]$ by \labelcref{S=wn}, we conclude that $S$ is a finite ring.
\end{proof}

\noindent We remark that the hypothesis ``$p-1$ is regular'' cannot be dropped in the statement of \Cref{lpowunitlarge}: take $S=\Z[t]\bigl/\Fa$, where $\Fa=\bigl(2(t-1),t^{2}-1\bigr)$, and consider $p=\overline t$. Then $p$ is invertible, for $p^{2}=1$, which incidentally implies that every element $g\in S$ is of the form $g=a+(p+1)b$, with $a,b\in\Z$; using that $(p-1)(p+1)=0$ we get $(p-1)g=(p-1)a$. Writing $a=2k+r$, with $k\in\Z$ and $r=0$ or $1$, and using that $2(p-1)=0$, we conclude that $(p-1)g+1=1$ or $p$. Thus, $\lpow(p)=\{1,p\}=\pow(p)\cup\{1\}$ (recall that $p^{2}=1$). Note that in this case $p-1\neq 0$ (there are no $f,g\in\Z[t]$ such that $t-1=2(t-1)f+(t^{2}-1)g$) but $(p-1)^{2}=(p^{2}-1)-2(p-1)=0$, so $p-1$ is a zerodivisor.

The following result deals with properties of logical powers in polynomial rings in one variable over reduced/indecomposable rings:

\begin{proposition}\label{p-reduinde}

Let $S=R[x]$, with $R$ a ring, and let $p\in R[x]$.

\begin{enumproposition}

\item\label{p-reduinde-a} If $R$ is reduced or indecomposable, and if $\lpow(p)$ contains an element that is multiple of its own square, then $p$ is invertible.

\item\label{p-reduinde-b} If $R$ is reduced or indecomposable, and if $p\in\lpow(p)$, then $p$ is either invertible or irreducible.

\item\label{p-reduinde-c} If $R$ is reduced and $\lpow(p)$ contains a zerodivisor, then $p$ is invertible.

\end{enumproposition}

\end{proposition}

\begin{proof}\leavevmode

\begin{enumerate}

\item Let $f\in\lpow(p)$ be a multiple of its own square, say $f=f^{2}\ell$. We have that the element $e=f\ell$ is idempotent and multiple of $p$. We also have $f=fe$, which implies $f=fe^{n}$ for all $n\geq 1$, that is, $f$ is infinitely divisible by $e$, and consequently $f$ is infinitely divisible by $p$; in particular, if $R$ is reduced, then $p$ is constant by \Cref{redinftyp}. Moreover, defining $h=1+(1-e)x$ we have that $f=fh$ because $(1-e)f=0$.

If $e=0$, then $f=fe=0$, and therefore $p$ is a unit, by \Cref{0inlpow}. If $e=1$, then $f$ is a unit because $e=f\ell$, which implies that $p$ is a unit as well (since $p$ divides $f$). Since $e=0$ or $1$ in a indecomposable ring, this reasoning settles such a case.

If $R$ is reduced and $e\neq 0,1$, then $h$ is not constant, hence a noninvertible divisor of $f$ (by \Cref{redunits}). Since $f\in\lpow(p)$, the element $p$ necessarily divides $h=1+x-ex$, so $p$ divides $1+x$ (because $p$ divides $e$). As we already observed, $p$ is constant in this case, so $p$ divides all the coefficients of $1+x$, and again we conclude that $p$ is invertible.

\item Since $0$ is not a unit, it follows from \Cref{0inlpow} that $0\notin\lpow(0)$, and so $p\neq 0$. If $p$ is invertible, then we are done; otherwise, if $p=gh$, then $p$ cannot divide both $g$ and $h$ (otherwise $p$ would be a multiple of its square, contradicting item \subcref{p-reduinde-a} above); since both $g$ and $h$ are divisors of $p$ and $p\in\lpow(p)$, it follows that one of $g$ or $h$ is a unit, which shows that $p$ is irreducible.

\item Assume that $p$ is noninvertible, and let $f,\ell\in R[x]$ be such that $f\in\lpow(p)$ and $f\ell=0$; our objective is to show that $\ell=0$. We have $f=fh$, with $h=1+p\ell x$. Since $p$ does not divide $h$ (otherwise $p$ would divide $h-p\ell x=1$), $h$ is a divisor of $f$ and $f\in\lpow(p)$, it follows that $h$ must be invertible. As $R$ is reduced, then $h$ is constant by \Cref{redunits}, so $p\ell=0$, which in turn implies $p=(1+\ell x)p$.

Since we are assuming $\lpow(p)\neq\varnothing$, it follows that $p\in\lpow(p)$ by \Cref{lpowfurther-b}. By item \subcref{p-reduinde-a} we cannot have that $p$ is multiple of $p^{2}$, and consequently the element $1+\ell x$, which is a divisor of $p$, cannot be a multiple of it. This, together with the fact that $p\in\lpow(p)$, implies that $1+\ell x$ must be invertible. Therefore $1+\ell x$ is constant (again by \Cref{redunits}), so $\ell=0$, as desired.\qedhere

\end{enumerate}

\end{proof}

\noindent One may wonder, following items \subcref{p-reduinde-a} and \subcref{p-reduinde-b} of \Cref{p-reduinde}, whether reducedness could be replaced by indecomposability in the hypothesis of item \subcref{p-reduinde-c} of the same proposition. As the counterexample below shows, this is not possible.

\pagebreak

Let $R$ be a local ring (see \Cref{ourlocal}) such that the ideal $\Fm$ of nonunits in $R$ is generated by a nonzero element $p$ with $p^{2}=0$ (as a concrete example, take $R=k[z]\bigl/(z^{2})$, $k$ being a field, and $p=\overline{z}$\,). We have that $R$ is indecomposable (\Cref{stronglocal}), and obviously $p$ is not invertible. We claim that $p$ is irreducible in $R[x]$, so we have $p\in\lpow(p)$ by \Cref{lpowgeneral-c}, and therefore the set $\lpow(p)$, with $p$ noninvertible, contains the zerodivisor $p$.

In order to prove our claim, it is suffices to show that $p=gh$ implies that one of $g$ or $h$ is a unit (because clearly $p\notin\{0\}\cup R[x]^{*}$). If $g=g_{0}+sx,h=h_{0}+tx$, with $s,t\in R[x]$, then $p=g_{0}h_{0}$. We have that $p^{2}=0$ does not divide $p$ (because $p\neq 0$), so one of $g_{0}$ or $h_{0}$ is not a multiple of $p$, hence it is invertible, say $g_{0}\in R^{*}$ and $h_{0}=g_{0}^{-1}p$.

Taking images in the integral domain $(R/\Fm)[x]$\footnote{If $a,b\in R$ are such that $ab$ is a nonunit, then $a$ or $b$ is a nonunit; since the set $\Fm$ of nonunits is already an ideal, it follows that $\Fm$ is indeed a prime ideal, so the quotient ring $R/\Fm$ is an integral domain. (Alternatively, in \Cref{nonstronglocal} is proved that $\Fm$ is maximal, hence prime.)}\! we get $0=(\overline{g_{0}}+\overline{s}x)\overline{t}x$. We have $(\overline{g_{0}}+\overline{s}x)x\neq 0$ because $\overline{g_{0}}\neq 0$, hence $\overline{t}=0$, that is $t=p\ell$ for some $\ell\in R[x]$, and consequently $p=p(g_{0}+sx)(g_{0}^{-1}+\ell x)$.

If $s=s_{n}x^{n}+\cdots+s_{0}$, with $n\geq 0$, then by item \subcref{basic-a} of \Cref{basic} we have $s_{n}^{k}p=0$ for some $k\geq 1$. As $p\neq 0$, it follows that $s_{n}$ cannot be invertible, hence it is a multiple of $p$, and in particular $ps_{n}=0$ (recall that $p^{2}=0$). If $\bighat{s}=s-s_{n}x^{n}$, then $p\bighat{s}=ps-ps_{n}x^{n}=ps$, and so we have $p=p(g_{0}+\bighat{s}x)(g_{0}^{-1}+tx)$. Iterating this argument we conclude that $p$ divides every coefficient of $s$, that is $p$ divides $s$, so $s^{2}=0$, and therefore $(g_{0}+sx)(g_{0}-sx)=g_{0}^{2}$ is a unit, which shows that $g$ is a unit.

\subsection{Algebraic equivalences for reducedness/indecomposability}\label{reduindecomp}

\begin{proposition}\label{pentavalente}

For a ring $R$ the following are equivalent:

\begin{enumproposition}

\item $R$ contains an idempotent element other than $0$ and $1$.

\item\label{pentavalente-b} $R$ is isomorphic to the direct product of two nonzero rings.

\item The polynomial $x$ in $R[x]$ is a product of two noninvertible polynomials of degree $1$.

\item\label{pentavalente-d} The polynomial $x$ in $R[x]$ is a product of two noninvertible polynomials of positive degree.

\item The polynomial $x$ in $R[x]$ is a product of two noninvertible polynomials. Equivalently, $x$ is reducible in $R[x]$.

\end{enumproposition}

\end{proposition}

\begin{proof}\leavevmode

\begin{enumerate}[labelindent=23pt,labelwidth=\widthof{\texttt{e}~$\Rightarrow$~\texttt{a}~$\Rightarrow$~\texttt{c}},itemindent=0em,leftmargin=!]

\item[\texttt{(a}~$\Rightarrow$~\texttt{b):}] If $e\in R$ is idempotent, then $f=1-e$ is too; moreover, the ideal $R_{1}=Re$ (resp.~$R_{2}=Rf$) has $e$ (resp.~$f$) as a multiplicative unit. Therefore both $R_{1}$ and $R_{2}$ are unital rings, with $1_{R_{1}}=e$ and $1_{R_{2}}=f$. If $S=R_{1}\times R_{2}$, then the mapping $R\to S$ given by $r\mapsto(re,rf)$ is a bijective ring homomorphism (it respects sums, products and sends $1_{R}$ to $(e,f)=1_{S}$), whose inverse given by $(ae,bf)\mapsto ae+bf$. This shows that $R\cong S$. If $e$ is a nontrivial idempotent, then both $R_{1}$ and $R_{2}$ are nonzero, which proves the implication.

\item[\texttt{(b}~$\Rightarrow$~\texttt{a):}] If $R_{1}$ and $R_{2}$ are nonzero rings, then the element $(1,0)$ is a nontrivial idempotent in the ring $R_{1}\times R_{2}$.

\item[\texttt{(c}~$\Rightarrow$~\texttt{d}~$\Rightarrow$~\texttt{e):}] Obvious.

\item[\texttt{(e}~$\Rightarrow$~\texttt{a}~$\Rightarrow$~\texttt{c):}] See the proof of \Cref{bivalente}.\qedhere

\end{enumerate}
\end{proof}

\noindent We may observe that, like integral domains, which are characterized by the property that $x$ is a prime element in $R[x]$, the class of indecomposable rings also corresponds to a specific property of the algebra generator $x$, namely, the polynomial $x$ is irreducible in $R[x]$ (by \Cref{bivalente}). In the case of reduced rings, since all positive degree polynomials are noninvertible by \Cref{redunits}, this characterization can be specialized in the following form:

\begin{proposition}

A reduced ring $R$ is indecomposable if and only if the polynomial $x$ in $R[x]$ is not a product of two polynomials of positive degree.

\end{proposition}

\noindent Finally, in order to express the class of rings $R$ we are interested in, in terms of properties of $R[x]$, we may synthesize as follows:

\begin{proposition}\label{redindec}\leavevmode

\begin{enumproposition}

\item\label{redindec-a} A ring $R$ is reduced if and only if the polynomial $1$ in $R[x]$ is not a product of two polynomials of positive degree.

\item A ring $R$ is reduced and indecomposable if and only if the polynomials $1$ and $x$ in $R[x]$ are not a product of two polynomials of positive degree.

\end{enumproposition}

\end{proposition}

\begin{proof}\leavevmode

\begin{enumerate}

\item If $R$ is reduced, then invertible elements of $R[x]$ are constant by \Cref{redunits}. For the converse, if $a\in R$ and $n\geq 2$ satisfy $a^{n}=0$ and $a^{n-1}\neq 0$, then $1=(1+a^{n-1}x)(1-a^{n-1}x)$.

\item Once item \subcref{redindec-a} above is given, this follows from condition \subcref{pentavalente-d} in \Cref{pentavalente}, as the requirement of noninvertibility of nonconstant elements becomes redundant in a reduced ring by \Cref{redunits}.\qedhere

\end{enumerate}
\end{proof}

\noindent We remark that the result of \Cref{cm+1dividescm-a} actually characterizes reduced indecomposable rings: in fact, let $R$ be a ring such that whenever $c^{m+1}$ divides $c^{m}$, then $c$ is either zero or a unit. On the one hand, if $e\in R$ is idempotent, then obviously $e^{2}$ divides $e$, and therefore $e=0$ or $e$ is a unit, and in the latter case we have $1=ee^{-1}=e^{2}e^{-1}=e$, which shows that $R$ is indecomposable. On the other hand, if $a\in R$ is nilpotent, say $a^{n}=0$, with $n\geq 1$, then obviously $a$ cannot be a unit (recall that $R$ is a nonzero ring), and since $0=a^{n+1}$ trivially divides $0=a^{n}$, it follows that $a=0$, and consequently $R$ is reduced.

\Cref{basic} exhibits some properties of polynomials in one variable, over reduced and/or indecomposable coefficient rings. These properties hold trivially in the particular case in which the coefficient ring is an integral domain, for the product of the leading coefficients of two given polynomials becomes necessarily the leading coefficient of their product. As we show below, these properties also characterize reducedness and/or indecomposability.

Suppose that a ring $R$ satisfies the last conclusion of \Cref{basic-c}, namely: in the ring $R[x]$, divisors of regular constant elements are themselves constant. Since units in $R[x]$ are divisors of the regular element $1$, we conclude that $R[x]^{*}=R^{*}$, so $R$ is reduced by \Cref{redunits}. On the other hand, let $R$ be a ring satisfying the conclusion of \Cref{basic-d}, namely: whenever $f,g\in R[x]$ are nonzero polynomials such that $g\mid f$ and the leading coefficient of $f$ is a unit, then that of $g$ is a unit too. We claim that $R$ is reduced and indecomposable.

In fact, if $a\in R$ satisfies $a^{2}=0$ and we take $f=1$ and $g=1+ax$, then the equality $(1-ax)g=f$ and the hypothesis over $R$ implies that $a$ cannot be the leading coefficient of $g$, so necessarily $a=0$, and this shows that $R$ is reduced. Moreover, if $e\in R$ is idempotent, then by taking $f=x,g=(1-e)x+e$ and using the equality $f=[ex+(1-e)]g$, we conclude that the leading coefficient of $g$ is a unit. Since this leading coefficient is one of $1-e$ or $e$, which are idempotent, and the only invertible idempotent in a ring is $1$, we conclude that $1-e=1$ or $e=1$, showing that $R$ is indecomposable as well.

One may be tempted to prove that the result of \Cref{basic-d} holds if ``unit'' is replaced by ``regular''. Let $R$ be a ring such that, whenever $f,g\in R[x]$ are nonzero polynomials such that $g\mid f$ and the leading coefficient of $f$ is regular, it is the case that the leading coefficient of $g$ is regular as well. By a reasoning entirely similar to that made in the previous paragraph, one concludes that $R$ is reduced and indecomposable; the converse, however, is not true: if $B,p,q$ and $R$ are as in \Cref{Bpq}, then $R$ is reduced and indecomposable. Now $(\overline{p}x+1)(\overline{q}x+1)=(\overline{p}+\overline{q})x+1$ in $R[x]$. Both $\overline{p}$ and $\overline{q}$ are zerodivisors, but $\overline{p}+\overline{q}$ is regular: for if $b\in B$ satisfies $pq\mid(p+q)b$, then $p\mid qb$, and since $p$ is prime and $p\nmid q$, it follows that $p\mid b$, say $b=ps$. Similarly we have $q\mid pb=p^{2}s$, and since $q$ is prime and $q\nmid p$, we conclude that $q\mid s$. Therefore $pq\mid b$, which proves our claim.

\setcounter{footnote}{0}

\subsection{More about constant polynomial functions}\label{nonconstantfpoly}

\noindent Let $R$ be a ring such that the only polynomials in $R[x]$ inducing constant polynomial functions on $R$ are the constant polynomials\footnote{For example, \Cref{constantfpoly} implies that this is the case when $R$ is infinite, reduced and indecomposable.}\!. We claim that if $R$ is also reduced and $g\in R[x]$ takes finitely many values, then $g\in R$ (and, \emph{a posteriori}, $g$ takes only one value).

To prove this, assume the contrary, and let $n\geq 2$ be minimal such that there exists a polynomial $g$ taking exactly $n$ values. If $a$ and $b$ are two such (distinct) values, then the polynomial $f=(a+b-g)g$ takes the value $ab$ when $g$ takes the values $a$ or $b$, and therefore $f$ takes at most $n-1$ values. On the one hand, by minimality of $n$ we necessarily have that $f$ is constant as a polynomial function, so $f\in R$ by the initial hypothesis. On the other hand, if $c$ is the leading coefficient of the nonconstant polynomial $g$, then $a+b-g$ also has positive degree and its leading coefficient equals $-c$. Since $R$ is reduced, we have $-c^{2}\neq 0$, and consequently $f=(a+b-g)g$ has positive degree, a contradiction.

The following examples show that neither reducedness nor indecomposability can be removed from the hypotheses of \Cref{constantfpoly}.

\begin{example}

If $R$ is a Boolean ring (that is, $a^{2}=a$ for all $a\in R$), then $R$ is decomposable unless $R=\F_{2}$, the field with two elements. On the other hand, the discussion in \Cref{Th-red-ind} implies immediately that $R$ is reduced. Finally, by definition the nonconstant polynomial $f=x^{2}-x\in R[x]$ vanishes on all of $R$. As a concrete example of infinite Boolean ring we may take $R$ as the direct product $\F_{2}^{\N}$.

\end{example}

\begin{example}

Let $S$ be the ring of polynomials in infinitely many variables $T_{1},T_{2},\dots$ over the field $\F_{2}$. Let $\Fa$ be the ideal in $S$ generated by the products $T_{i}T_{j}$, with $1\leq i\leq j$, and consider the factor ring $R=S/\Fa$. Denoting the class of $T_{i}$ modulo $\Fa$ by $t_{i}$, we have that every element of $R$ is of the form $p=a_{0}+\sum_{i=1}^{\infty}a_{i}t_{i}$, with $a_{i}=0$ or $1$ for all $i\geq 0$, and $a_{i}\neq 0$ for finitely many $i$. Moreover, $p=0$ if and only if $a_{i}=0$ for all $i$; in particular, all the elements $t_{i}$ are pairwise distinct, so $R$ is infinite.

Since $R$ has characteristic $2$, we have $p^{2}=a_{0}^{2}+\sum_{i=1}^{\infty}a_{i}^{2}t_{i}^{2}=a_{0}^{2}=0$ or $1$, so the nonconstant polynomial $x^{2}(x^{2}-1)\in R[x]$ vanishes on all of $R$, and moreover $p^{2}-p=\sum_{i=1}^{\infty}a_{i}t_{i}$. Thus, $p^{2}=p$ implies $a_{i}=0$ for all $i\geq 1$, that is $p=a_{0}$, which shows that $R$ is indecomposable. Obviously $R$ is not reduced, as $t_{i}\neq 0$ for each $i$ but $t_{i}^{2}=0$.

\end{example}

\subsection{About first-order characterizations of some subclasses of reduced indecomposable polynomial rings}

Recall that in \Cref{definethemall} we wrote the class $\CH$ of reduced indecomposable polynomial rings as the union $\CH_{1}\cup\CH_{2}$, with $\CH_{1}$ being the subclass of rings in $\CH$ where every nonzero integer is invertible, and $\CH_{2}$ the subclass of rings in $\CH$ expressible as $R[x]$, where $R$ is a nonfield of characteristic zero.

The negation of the sentence $\Xi$ defined in the statement of \Cref{XiXi} characterizes the members of $\CH_{1}$ in the class $\CH$. In what follows we construct two sentences that characterize, respectively, those rings in $\CH$ having characteristic zero, and those of the form $R[x]$, with $R$ a field. As a consequence, we obtain a first-order characterization of the subclass $\CH_{2}$ in the class $\CH$. Moreover, since reducedness and indecomposability are finitely (first-order) axiomatizable (see \Cref{Th-red-ind}), we may easily modify these two sentences to characterize the rings mentioned (reduced indecomposable polynomial rings of characteristic zero, and polynomial rings in one variable over a field, respectively) \emph{in the whole class of polynomial rings} (in any set of variables, by \Cref{R[Xfat]}).

Let $S$ be a ring. It is easy to see that $\Char(S)>0$ if and only if $-1\in\ZZ_{S}^{+}$. From this, and recalling that the formula $\Omega(\cdot)$ appearing in \Cref{finalboss} defines $\ZZ_{S}^{+}$ for any ring $S\in\CH$, we conclude that the sentence $\neg\Omega(-1)$ characterizes the rings in $\CH$ having characteristic zero\footnote{We remind the reader that, as a standard application of the compactness theorem, we get that the theory of commutative unital rings of characteristic zero is not finitely axiomatizable.}\!.

On the other hand, the argument in the proof of \Cref{ele1k=uele2k} can be used to characterize, among members of $\CH$, the polynomial rings in one variable over a field. Let $S=R[x]\in\CH$, and consider the sentence
\begin{equation}\label{sentence}
\forall\ell_{1}\,\forall\ell_{2}\,\bigl[\,\,(\ell_{1}\in W\ \wedge\ \ell_{2}\in W\,)\rightarrow\exists u\,\exists\rho\,(\,u\in S^{*}\ \wedge\ \rho\in\{0\}\cup S^{*}\ \wedge\ \ell_{1}=u\ell_{2}+\rho\,)\,\bigr]\,,\tag{$\ddagger$}
\end{equation}
$W$ being as in \Cref{defWL}. Notice that $S^{*}=R^{*}$ by \Cref{redunits}, so in particular the elements $u$ and $\rho$ appearing in \labelcref{sentence} must belong to $R$.

If $R$ is a field, then $W$ coincides with the set $L$ of linear polynomials in $R[x]$, by \Cref{WandL-c}, and it is easy to show that \labelcref{sentence} holds in this case. Conversely, suppose that \labelcref{sentence} holds, and take any $a\in R$: we have $x,x+a\in L\subseteq W$, and therefore there are $u\in S^{*}=R^{*}$ and $\rho\in\{0\}\cup S^{*}=\{0\}\cup R^{*}$ with $x+a=ux+\rho$. In particular we have $a=\rho\in\{0\}\cup R^{*}$, hence $R\subseteq\{0\}\cup R^{*}$, which proves that $R$ is a field.

The characterization above implies the following: let $S=k[x]$, with $k$ being a field and $x$ an indeterminate over $k$. If $R$ is a subring of $S$ and $y\in S$ are such that $y$ is an indeterminate over $R$ and $S=R[y]$, then $R=k$. In fact, as $k$ is a field, we have $S\in\CH$, and the sentence \labelcref{sentence} is true in $S$. Since $S$ is reduced and indecomposable, so is $R$. Therefore, by the characterization made above (using $y$ instead of $x$) we have that $R$ is also a field. Finally, by \Cref{redunits} we have $R^{*}=S^{*}=k^{*}$, and therefore $R=\{0\}\cup R^{*}=\{0\}\cup k^{*}=k$, as claimed\footnote{Another proof runs as follows: Since $R^{*}=k^{*}$, we have $k=\{0\}\cup k^{*}=\{0\}\cup R^{*}\subseteq R$; in particular we have $S=k[x]=R[x]$ (but $x$ is not necessarily an indeterminate over $R$). For the reverse inclusion, notice that any element $r\in R\subseteq k[x]$ can be written in the form $r=a_{0}+a_{1}x+\cdots+a_{n}x^{n}$, for some $n\geq 0$ and some $a_{0},\ldots,a_{n}\in k$ with $a_{n}\neq 0$. If $n$ were nonzero, then $x\in S$ would be a root of the monic polynomial $T^{n}+a_{n-1}a_{n}^{-1}T^{n-1}+\cdots+a_{1}a_{n}^{-1}T+(a_{0}-r)a_{n}^{-1}\in\bigl(k[r]\bigr)[T]\subseteq R[T]$. Thus, $x$ is $R$-integral, and so $R[y]=S=R[x]$ is a finitely generated $R$-module (\cite{Hungerford1980}*{Theorem VIII.5.3}), which is impossible because $y$ is an indeterminate over $R$.}\!. Thus, for any polynomial ring in \emph{one} variable over a field, its coefficient field is unique (the indeterminate can vary, of course: for example, by affine maps).

\pagebreak

\subsection{Comparing powers with logical powers}

\noindent Let $S$ be a ring, and consider the definable subsets $T$ and $U$ of $S$ of \Cref{defsTandUPV}. By \Cref{TandUPV}, every element $p\in T$ satisfies $\pow(p)\subseteq\lpow(p)$. In addition, if $S=R[x]$, with $R$ reduced and indecomposable, then every element $p\in U$ satisfies $\pow(p)=\lpow(p)$.

If we replace ``$p$ is irreducible'' by ``$p\in\lpow(p)$'' in the definition of $T$, then it remains true that $\pow(p)\subseteq\lpow(p)$ for each $p\in T$. The converse is almost true: it is easy to show that if $p\in S$ satisfies $\pow(p)=\lpow(p)$, then $p\in T$. Moreover, for any unit $p$ we have, by \Cref{lpowfurther-c}, that $\lpow(p)=\{(p-1)g+1\colon g\in S\}$. Taking $g=1$ we obtain $p\in\lpow(p)$, and if $h\in\lpow(p)$, say $h=(p-1)g+1$, with $g\in S$, then $ph=(p-1)pg+p=(p-1)(pg+1)+1\in\lpow(p)$. This shows that $S^{*}\subseteq T$ under the modified definition of $T$.

If $p\in U$ is a unit, then we may take $a=p$ in the second condition of the definition of $U$, obtaining $p=1$. Consequently, if we impose the additional condition ``$p\neq 1$'' in the definition of $U$, then $U$ consists entirely of nonunits. As we want all elements $p$ in $U$ to satisfy $\lpow(p)=\pow(p)$, this restriction will be unharmful, because for a unit $p\in S$ we have, in most cases, that $\lpow(p)$ strictly contains $\pow(p)$: namely, when $p-1$ is regular, by \Cref{lpowunitlarge}. Note that this regularity condition is essential for the proofs of our main results to work.

Even with the extra requirement ``$p\neq 1$'' in the definition of $U$, and the modification in the definition of $T$ (``$p$ is irreducible'' by ``$p\in\lpow(p)$''), we are still able to prove that if $S=R[x]$, with $R$ reduced and indecomposable, then $\lpow(p)=\pow(p)$ for each $p\in U$. The proof is almost identical to that of \Cref{TandUPV}, with the following modification: If $p\in U$ and $f\in\lpow(p)$ is infinitely divisible by $p$, then the new definition of $T$ no longer implies that $p$ is irreducible, but we are still able to conclude that $p\in\{0\}\cup R^{*}$ by \Cref{cm+1dividescm-a}. Moreover, $p\in\lpow(p)$ (part of the new definition of $T$) , together with \Cref{p-reduinde-b}, implies that $p$ is either invertible or irreducible. Putting together these facts, we conclude that $p$ is necessarily a unit, which is impossible under the modified definition of $U$.

Finally, consider the modified versions of $T$ and $U$. If $S=R[x]$, with $R$ reduced (not necessarily indecomposable), and if $p\in S$ is nonconstant with regular leading coefficient, and satisfies $\lpow(p)=\pow(p)$, then $p\in T$ (as discussed above) and obviously $p\neq 1$. The proof of the remaining conditions for membership in $U$ is similar as the proof of \Cref{TandUPV-a} for the special case $p=x$, and in this way we conclude that $p\in U$. Notice that \Cref{lp=pindom} provides examples of nonlinear polynomials satisfying the requirements above (the classical example being $R=\mathbb{R}$ and $p=x^{2}+1$). This is in contrast with \Cref{automorphU}, which merely guarantees that linear polynomials with invertible leading coefficient belong to the set $U$ (in the case $R$ reduced and indecomposable).

\enlargethispage*{5mm}

\subsection{Revisiting examples}\label{expandidos}

\noindent It is possible to prove the definability of the integers in $R[x]$, for $R$ as in \Cref{ZxZ}, and as well for $R$ as in some instances of \Cref{C(X|B),B^I,stronglocal}, by constructing a definable set $A$ of $R[x]$ satisfying $\ZZ^{+}\subseteq A\subseteq R$, and applying \Cref{ZisdefinA}. Notice that the rings in \Cref{ZxZ}, as well as some instances of the rings in \nameCrefs{C(X|B)} \labelcref{C(X|B)}\footnote{In general, the ring $R=\mathcal{C}(X,B)$ contains an isomorphic copy of $B$, namely, the subring of constant functions. We claim that $R$ is a field if and only if $B$ is a field and $R=B$. In fact, if $B$ is a field and $R=B$, then obviously $R$ is a field. For the converse, suppose that $B$ is not a field or that $R$ properly contains $B$. In the first case, if $b$ is any nonzero nonunit in $B$, then the constant function with value $b$ has no inverse in $R$. In the second case, some function $f\in R$ takes two distinct values, say $a\neq b$ in $B$, and therefore the function $f-a$ is nonzero (as it takes the value $b-a$) and noninvertible (as it takes the value $0$). In either case we conclude that $R$ is not a field.}\!, \labelcref{B^I,stronglocal}, are nonfields of characteristic zero, and thus they are also covered by \Cref{nonfieldZdef}.

\begin{example}[\Cref{stronglocal}, revisited]

If $R$ is a local and reduced ring, then the set $A=\{f\in R[x]\colon f\in R[x]^{*}\textnormal{ or }f+1\in R[x]^{*}\}$ is definable. We have $A\subseteq R$ by \Cref{redunits}, and $R\subseteq A$ by the definition of local ring. Therefore $R=A$.

\end{example}

\begin{example}[\Cref{C(X|B)}, revisited]

Let $R$ as in \Cref{C(X|B)}. If every nonzero integer in $B$ is invertible, then the same happens to each nonzero integer constant function from $X$ to $B$, so we may apply \Cref{Z+invertible} in this case.

\end{example}

\begin{example}[\Cref{ZxZ}, revisited]\label{ZxZrevisited}

Consider the ring $R[x]$, where $R$ is defined as in \Cref{ZxZ}. Note that an element $(m,n)\in R$ is regular if and only if $m,n\neq 0$.

Let $p=(5,1),q=(1,5)\in R[x]$. The set $B=\{p,q\}$ can be defined\footnotemark\! by the formula
\[\beta(t)\colon\ \ \exists r\,(\,r\mid 1\ \wedge\ r\neq 1\ \wedge\ r\neq -1\ \wedge\ t=3+2r\,)\,.\]
We claim that $\lpow(p)=\pow(p)$ and $\lpow(q)=\pow(q)$. It is easy to see that $p$ is prime in $R$, so it remains prime in $R[x]$\footnotemark\!. Since $p$ is also regular, \Cref{lpowgeneral-d} implies that $\pow(p)\subseteq\lpow(p)$. Conversely, let $h\in\lpow(p)$, and let us denote $h_{m}$ by $(f_{m},g_{m})\in R$. We have that $p-1=(4,0)$ divides $h-1$, so $4$ divides $f_{0}-1$, and thus $f_{0}$ is odd. If $p^{k}$ divides $h$ in $R$, then $5^{k}$ must divide $f_{0}\neq 0$ in $\Z$, and so $h$ cannot be infinitely divisible by $p$. Therefore, by \Cref{lpowgeneral-a}, we can write $h=up^{n}$, with $n\geq 1$ and $u\in R[x]^{*}=R^{*}$ (see \Cref{redunits}) such that $p-1=(4,0)$ divides $u-1$ in $R$. We have $u-1=(0,0),(-2,-2),(0,-2)$ or $(-2,0)$; since $-2$ is not multiple of $0$ or $4$, it follows that $u-1$ must be equal to $(0,0)$, so $h\in\pow(p)$. The proof of $\lpow(q)=\pow(q)$ is analogous.

As a consequence of these two equalities right above, we obtain that the set $C=\{(5^{m}-1,5^{n}-1)\colon m,n\geq 1\}$ is definable by the formula
\begin{align*}
\gamma(t)\colon\ \ \exists r\,\exists s\,\exists v\,\exists w\,[\,\beta(r)&\wedge\ r+s=6\ \wedge\ v\in\lpow(r)\\
&\wedge\ w\in\lpow(s)\ \wedge\ t=v+w-2\,]\,.
\end{align*}
If $D\subseteq R[x]$ is the set of divisors of elements in $C$, then obviously $D$ is also definable; moreover, since $R$ is reduced and $C\subseteq R$ consist entirely of regular elements, it follows from item \subcref{basic-a} of \Cref{basic} that $D\subseteq R$. If $\phi$ denotes the Euler's totient function, then it is well-known that, for any positive integer $a$ not a multiple of $5$, we have that $5^{\phi(a)}-1$ is divisible by both $a$ and $-a$. Therefore $D$ contains all the elements $(d,d)$, with $d\in\Z$ not a multiple of $5$ as a rational integer. Consequently, the set $A$ of elements $t$ such that $t\in D$ or $t+1\in D$ is definable, and it satisfies $\ZZ\subseteq A\subseteq R$.

\end{example}

\addtocounter{footnote}{-1}

\footnotetext{It is not possible to tell apart $p$ from $q$ by using a first-order formula.}

\stepcounter{footnote}

\footnotetext{It is well-known that if $\Fa$ is an ideal in a ring $R$, then $R[x]\Fa$, the ideal in $R[x]$ generated by $\Fa$, is precisely the set $\Fa[x]$ of polynomials with coefficients in $\Fa$, and $R[x]\bigl/\Fa[x]\cong(R/\Fa)[x]$. If $p\in R$ is prime and $\mathfrak{q}=Rp$, then $R[x]p=\mathfrak{q}[x]$, so $R[x]\bigl/R[x]p\cong(R/\mathfrak{q})[x]$, which is an integral domain, and this shows that $p$ remains prime in $R[x]$.}

\begin{example}[\Cref{B^I}, revisited]

Let $R$ be as in \Cref{B^I} with $B=\Z$ and $\Fb=2\Z$, and set $S=R[x]$. We may also think of elements of $S$ as $I$-tuples of integers polynomials whose coefficients in any fixed degree have the same parity. Let $D=\{d\in S^{*}\colon 2d+3\textnormal{ is irreducible}\}, C=\{2d+3\colon d\in D\}$ and $E=\{d+1\colon d\in D\}$. It is easy to check that $C,D$ and $E$ are definable sets, that $D$ is the set of $I$-tuples with one entry equal to $1$ and all other entries equal to $-1$, that $C$ consists precisely of the $I$-tuples, all irreducible, with one entry equal to $5$ and all other entries equal to $1$, and finally, that $E$ consists of those $I$-tuples with one entry equal to $2$ and all other entries equal to zero.

We claim that, for any $c\in C$, one has $\lpow(c)=\pow(c)$. Indeed, for $c\in C$, there must be $j\in I$ such that $c_{j}=5$ and all other entries of $c$ are $1$. If $f\in\pow(c)$, then $c\mid f$ and $c-1\mid f-1$ are obviously satisfied and, if $g\mid f$, then all but one entry of $g$ are $\pm 1$ and the other, $g_{j}$, divides a power of $5$. Thus $g_{j}$ must be a constant, by \Cref{basic-a}, and consequently it is either $\pm 1$ or a multiple of $5$. Therefore $g$ is either invertible or a multiple of $c$.

Conversely, if $f\in\lpow(c)$, then $c-1\mid f-1$ forces all but the $j$-th component of $f-1$ to vanish and $4\mid f_{j}-1$; in particular $f_{j}\neq 0$. Since $c\mid f$, it follows that $f_{j}=m\cdot 5^{n}$, with $n>0$ and $m\equiv 1\pmod 4$ not a multiple of $5$. Furthermore, if $m$ were not invertible, then the element $\tilde{f}$ with $\tilde{f}_{j}=m$ and all other entries equal to $1$, not divisible by $c$, would be a noninvertible divisor of $f\in\lpow(c)$, a contradiction. Therefore $m=\pm 1$. Since $m\equiv 1\pmod 4$, we conclude that $m=1$ and, therefore, $f\in\pow(c)$.

Consider the following formula:
\begin{align*}
\alpha(t)\colon\ \ \forall e\,\bigl[\,e\in E\rightarrow\exists c\,\exists y\,\bigl(\,c\in C&\wedge\ y\in\lpow(c)\\
&\wedge\ [\,t\cdot e\mid y-1\ \vee(t+1)\cdot e\mid y-1\,]\,\bigr)\,\bigr]\,.
\end{align*}
The formula holds whenever multiplication of $t$ or $t+1$ by any element of $E$ divides $y-1$, for some logical power $y$ of a suitable $c\in C$. We claim that $r\in R$ precisely when $\alpha(r)$ holds. Indeed, let $r\in R$ be a constant element, and let $e\in E$. There exists $j\in I$ such that $e$ has all entries equal to zero but its $j$-th entry, which is equal to $2$; then $r\cdot e$ and $(r+1)\cdot e$ have one constant integer entry, namely $2r_{j}$ and $2(r_{j}+1)$, respectively, and all other entries equal to zero. By Euler's theorem, any rational integer not divisible by $5$ divides some element of the form $5^{n}-1$. In view of this, since $2r_{j}$ and $2(r_{j}+1)$ cannot both be a multiple of $5$, and using $\lpow(c)=\pow(c)$, for all $c\in C$, we conclude that one of $r\cdot e$ or $(r+1)\cdot e$ divides $y-1$, for some logical power $y$ of the element $c\in C$ with $5$ in the $j$-th entry and $1$ in all other entries, and so $\alpha(r)$ holds.

Conversely, let $s\in R[x]$ be nonconstant, say $\deg(s_{j})\geq 1$, and consider the element $e\in E$ such that $e_{j}=2$ and all other entries of $e$ are zero. Suppose either $s\cdot e$ or $(s+1)\cdot e$ divides $y-1$, for some $y\in\lpow(c)$ and some $c\in C$. As $y-1$ has only one nonzero entry, which is a constant, and $s\cdot e$ has all but the $j$-th entry equal to zero, we must have that $(y-1)_{j}\neq 0$ is a constant and $(y-1)_{i}=0$ for all $i\neq j$. But both $(s\cdot e)_{j}=2s_{j}$ and $\bigl((s+1)\cdot e\bigr)_{j}=2(s_{j}+1)$ have positive degree and therefore they cannot divide, in the reduced ring $\Z[x]$, the nonzero constant element $(y-1)_{j}$ (by \Cref{basic-a}), proving that $\alpha(s)$ is false. Therefore $R$ is definable and we can just take $A=R$.

\end{example}

\enlargethispage*{15mm}

\subsection{About definability of integers in some nonreduced/decomposable rings}\label{otherrings}

\noindent In this subsection we discuss definability/undefinability of integers in certain rings not considered in our work.

First, we deal with the direct product of two rings. If both rings have positive characteristic, definability of integers obviously holds; moreover, the proof below shows that one can define integers in some cases where exactly one of the rings has characteristic zero, but such a definition is impossible when both rings have characteristic zero. Since decomposable rings are precisely those isomorphic to direct products of two (nonzero) rings (condition \subcref{pentavalente-d} in \Cref{pentavalente}), this shows that the indecomposability condition is essential to prove definability of integers in the most interesting cases.

Part of the result below is the claim discussed right after \Cref{ZxZ}. We would like to reiterate that the proof we will present below is actually a very minor modification of the proof of \cite{AschenbrennerKNS2018}*{Lemma 4.7}.

\begin{proposition}\label{ZnotdefinAxB}

Let $A$ and $B$ be rings, and let $S=A\times B$.

\begin{enumerate}

\item If both $A$ and $B$ have characteristic zero, then the prime subring $\ZZ_{S}$ is not definable in $S$.

\item If $\Char(B)=0$ and $\Char(A)>0$, then $\ZZ_{S}$ is definable in $S$ if and only if the set $D=\{(0_{A},k\cdot 1_{B})\colon \Char(A)\textnormal{ divides }k\}$ is definable in $S$.

\end{enumerate}

\end{proposition}

\begin{proof}

By the Feferman-Vaught theorem, any definable subset of $S$ is a finite union of ``definable rectangles'', that is, subsets of the form $P\times Q$, where $P$ and $Q$ are definable subsets of $A$ and $B$, respectively (\cite{Hodges1993}*{Corollary 9.6.4}). Notice that in this case we have
\[\ZZ_{S}=\{(j\cdot 1_{A},j\cdot 1_{B})\colon j\in\Z\}\,.\]
Let $P\times Q$ be a nonempty definable rectangle contained in $\ZZ_{S}$, and let $m\in\Z$ be fixed such that $m\cdot 1_{A}\in P$. We claim that, if $\Char(B)=0$, then $P=\{m\cdot 1_{A}\}$ (that is, a singleton) and $Q\subseteq B_{m}$, where
\[B_{m}=\{k\cdot 1_{B}\colon k\equiv m\hspace{-2mm}\mod\hspace{-1mm}\Char(A)\}\footnote{Notice that $B_{m}$ only depends on the congruence class of $m$ modulo $\Char(A)$.}\!.\]
In fact, let $n,k\in\Z$ be such that $n\cdot 1_{A}\in P$ and $k\cdot 1_{B}\in Q$. Then the pairs $(m\cdot 1_{A},k\cdot 1_{B})$ and $(n\cdot 1_{A},k\cdot 1_{B})$ belong to $\ZZ_{S}$, so there are $i,j\in\Z$ such that
\[(m\cdot 1_{A},k\cdot 1_{B})=(i\cdot 1_{A},i\cdot 1_{B})\ \ \textnormal{and}\ \ (n\cdot 1_{A},k\cdot 1_{B})=(j\cdot 1_{A},j\cdot 1_{B})\,.\]
The hypothesis $\Char(B)=0$ forces $i=k=j$, hence $n\cdot 1_{A}=i\cdot 1_{A}=k\cdot 1_{A}=j\cdot 1_{A}=m\cdot 1_{A}$. Therefore $n\cdot 1_{A}=m\cdot 1_{A}$ and $k\equiv m\hspace{-1mm}\mod\hspace{-1mm}\Char(A)$, as claimed.

\begin{enumerate}

\item Applying the previous reasoning, and using that $\Char(A)=0$, we conclude that the set $B_{m}$ must be a singleton, which forces the definable rectangle $P\times Q$ to be a singleton. Therefore, any definable subset of $S$ contained in $\ZZ_{S}$ is finite, and since $\ZZ_{S}$ is infinite (because $\Char(S)=0$), we conclude that $\ZZ_{S}$ is not definable in $S$.

\item If $\Char(B)=0$ and $\ZZ_{S}$ is definable, then the previous reasoning implies that $\Char(A)>0$ and that $\ZZ_{S}$ is the union of finitely many definable ``vertical segments''. In particular, the union of some (finite) subfamily of these segments will be equal to the intersection of $\ZZ_{S}$ with the ``$B$-axis''. In other words, the set
\[\ZZ_{S}\cap(0\times B)=\{(0_{A},k\cdot 1_{B})\colon\Char(A)\mid k\}=D\]
is definable. Conversely, if $\Char(A)>0$ and $D$ is definable in $S$, then $\ZZ_{S}$ is the union of finitely many definable translates of $D$, namely
\[\ZZ_{S}=\bigcup_{i=1}^{\Char(A)}D+\{(i\cdot 1_{A},i\cdot 1_{B})\}\,,\]
and therefore $\ZZ_{S}$ is a definable subset of $S$.\qedhere

\end{enumerate}

\end{proof}

\noindent As a specific example, if $A=\Z/n\Z$ and $B=\Z$, then $D=nS$, so $D$ is definable.

Now we turn our attention to definability results in some nonreduced rings.

\begin{proposition}\label{AKNS}

Let $S$ be a ring.

\begin{enumproposition}

\item \textup{(}\cite{AschenbrennerKNS2018}*{Corollary 2.19}\textup{)} If $\Char(S)=0$ and $S$ is bi-interpretable with $\Z$, then $\ZZ_{S}$ is definable.

\item\label{AKNS-b} \textup{(}\cite{AschenbrennerKNS2018}*{Main Theorem}\textup{)} Suppose that $S$ is finitely generated as $\Z$-algebra, and denote the nilradical of $S$ by $N$. We have that $S$ is bi-interpretable with $\Z$ if and only if both $\operatorname{ann}_{\Z}(N)\neq 0$ and $\Spf(S)$, the subset of $\Sp(S)$ of nonmaximal prime ideals of $S$, is nonempty and connected \textup{(}with respect to the Zariski topology of the ambient space\textup{)}.

\end{enumproposition}

\end{proposition}

\noindent In general, if $R$ is a ring such that $\Spf(R[x])$ is connected, then $\Sp(R)$ is connected. In fact, the inclusion $i\colon R\to R[x]$ induces a continuous map $i^{*}\colon\Sp(R[x])\to\Sp(R)$, given by $i^{*}(\Fp)=\Fp\cap R$. If $\Fq\in\Sp(R)$, then $\Fq=i^{*}(\Fq[x])$, and $\Fq[x]\in\Spf(R[x])$ because $R[x]\bigl/\Fq[x]\cong(R/\Fq)[x]$, which is never a field. Therefore $\Sp(R)$ is a continuous image of the connected set $\Spf(R[x])$.

Consequently, if we assume the Boolean prime ideal theorem or that $R$ is Noetherian, then $R$ will be indecomposable (see \Cref{SpecBPI}), and this shows, once again, how crucial the indecomposability condition is to definability of integers.

Despite the main result of our work deals with nonfinitely generated rings, it is limited to reduced ($\operatorname{ann}_{\Z}(N)=\Z$) polynomial rings. In what follows we exhibit an example of a nonreduced, polynomial, finitely generated ring of characteristic zero, which is bi-interpretable with $\Z$. By \Cref{AKNS}, in such ring the prime subring will be definable, but since it is not reduced, the ring is not covered by our result; notice that, by the discussion after \Cref{AKNS}, such example will be necessarily indecomposable.

Let $R$ be a finitely generated ring of characteristic zero, and let $\Fa$ be an ideal in $R$ whose radical is prime and nonmaximal, say $\sqrt{\Fa}=\Fp\in\Spf(R)$. Given integers $m,d>1$, let $\Fb=\Fa^{m}+d\Fp$, and define $S=R/\Fb$. If $\Z\cap\Fb=0$, then $\Char(S)=0$. Moreover, using the inclusions $\Fa^{m}\subseteq\Fb\subseteq\Fp$ we get $\sqrt{\Fb}=\Fp$.

Given a subset $C$ of $R$, let $V(C)$ be the subset of $\Sp(R)$ of prime ideals in $R$ containing the set $C$. It is well-known that $\Spf(S)$ is homeomorphic to
\[V(\Fb)\cap\Spf(R)=V(\sqrt{\Fb})\cap\Spf(R)=V(\Fp)\cap\Spf(R)\,.\]
We have $V(\Fp)\cap\Spf(R)\neq\varnothing$ by nonmaximality of $\Fp$, and $\Fp$ being prime implies that $V(\Fp)\cap\Spf(R)$ is connected: indeed, if $V(\Fp)\cap\Spf(R)\subseteq V(\Fg)\cup V(\Fh)$, then $\Fp\in V(\Fp)$ belongs to either $V(\Fg)$ or $V(\Fh)$, say $V(\Fg)$. Therefore we have $\Fg\subseteq\Fp$ and so $V(\Fp)\subseteq V(\Fg)$. The reasoning above shows then that $\Spf(S)$ is nonempty and connected. Moreover, if $N(S)$ denotes the nilradical of $S$, then $N(S)=\Fp/\Fb$, hence $dN(S)=0$ because $d\Fp\subseteq\Fb$. Thus, the ring $S$ is bi-interpretable with $\Z$, by \Cref{AKNS-b}.

It remains to impose conditions on $S$ in such a manner that $S$ be nonreduced, that is $\Fb\subset\Fp$. We have
\[\Fb=\Fa^{m}+d\Fp\subseteq\Fa^{2}+d\Fp\subseteq\sqrt{\Fa}^{2}+d\Fp=\Fp^{2}+d\Fp\subseteq\Fp\,.\]
Therefore $\Fp=\Fb$ would imply $\Fp=\Fp^{2}+d\Fp=\Fc\Fp$, where $\Fc=\Fp+dR$. By the Cayley-Hamilton theorem (\cite{Eisenbud1995}*{Corollary 4.7}), there exists $c\in\Fc$ such that $(c-1)\Fp=0$. If $\Fp$ contains a regular element, then $c=1$, so $\Fp$ and $dR$ would be comaximal.

Thus, it is sufficient to assume, besides the conditions imposed above, that $\Fp$ contain a regular element, and that the ideal $\Fp+dR$ be proper. As a concrete example, we may take $R=\Z[t]$, where $t$ is an indeterminate, and $\Fa=tR$, so $\sqrt{\Fa}=tR=\Fp$. In this case we have $\Fb=t^{m}R+dtR$, for some $m,d>1$, and $1\notin tR+dR=\Fp+dR$.

\pagebreak

Notice that this example $S=R/\Fb$ of finitely generated ring bi-interpretable with $\Z$ is not a polynomial ring; fortunately, the polynomial version $S[x]$ inherits all the relevant properties of $S$, and therefore constitutes our aimed example. This follows from the identities
\begin{align*}
S[x]\cong&\,R[x]\bigl/\Fb[x]\,,\\
\sqrt{\Fa[x]}=&\,\sqrt{\Fa}[x]\,,\\
\Fb[x]=&\,(\Fa[x])^{m}+d\Fp[x]\,,\\
\Z\cap\Fb[x]=&\,\Z\cap\Fb\,,\\
N(S[x])=&\,N(S)[x]\,,\\
\Fp[x]+dR[x]=&\,(\Fp+dR)[x]\,,
\end{align*}
together with the fact that $V(\Fp[x])\cap\Spf(R[x])$ is nonempty and connected because $\Fp[x]$ is prime and nonmaximal.

\subsection{Further discussion concerning local rings and \texorpdfstring{\textsf{AC}}{\unichar{"1D5D4}\unichar{"1D5D6}}}\label{nonstronglocal}

\noindent It is customary to define a local ring in an alternative way to that given in \Cref{ourlocal}, namely, as a ring with a unique maximal ideal. This property is a straightforward consequence of the definition given in \Cref{ourlocal} (nonunits form an ideal), and it is well-known that, in the presence of the axiom of choice (\textsf{AC}), these definitions are equivalent (we provide below proofs of these facts).

Interestingly enough, the interchangeability between the two notions of locality is not just a consequence of \textsf{AC}, but is indeed equivalent to it. To the best of our knowledge, this is a new condition equivalent to the axiom of choice.

Let $S$ be a ring such that the set $\Fm=S\smallsetminus S^{*}$ of nonunits of $S$ forms an ideal. We claim that $\Fm$ is the unique maximal ideal in $S$: on the one hand, any ideal $\Fa$ strictly containing $\Fm$ must contain a unit, and therefore $\Fa=S$, which shows that $\Fm$ is maximal. On the other hand, if $\Fn$ is a maximal ideal in $S$, then $\Fn$ contains no unit, so $\Fn\subseteq\Fm$, and therefore $\Fn=\Fm$ by maximality of $\Fn$.

The argument above shows that every local ring in the sense of \Cref{ourlocal} has a unique maximal ideal (namely, its set of nonunits), which is the standard definition of ``local ring''. The converse is not true in \textsf{ZF}: in fact, we contend that the assertion ``In every ring with a unique maximal ideal the set of nonunits forms an ideal'' is equivalent to the claim that every (nonzero commutative unital) ring has a maximal ideal. This condition, in turn, is known to be equivalent to the axiom of choice (\cite{Hodges1979}).

To prove our claim, suppose that every nonzero ring has a maximal ideal. By working on quotient rings, we get that every nonunit in a ring belongs to a maximal ideal. Therefore, in a ring with a unique maximal ideal, all nonunits must belong to that maximal ideal, which in turn consists entirely of nonunits. This proves that the set of nonunits of the ring forms an ideal.

For the converse implication, if $A$ is a nonzero ring without maximal ideals, then the ring $S=\Q\times A$ has $0\times A$ as its unique maximal ideal. As we already saw, if the set of nonunits of $S$ were an ideal, then it would be equal to the unique maximal ideal, and so $S\smallsetminus S^{*}=0\times A$; but this equality is impossible, because $(1,0)\in S\smallsetminus S^{*}$ and $(1,0)\notin 0\times A$. This shows that nonunits in the ring $S$ do not form an ideal.

Notice that, incidentally, the ring $S$ above is not indecomposable (by condition \subcref{pentavalente-b} in \Cref{pentavalente}), so we cannot change ``local'' by ``the ring has a unique maximal ideal'' in \Cref{stronglocal}.

\pagebreak

\subsection{Diagram of implications}

\noindent In the diagram below we show the implications between the conditions of reducedness/indecomposability of a ring $R$, and properties of the subsets $\pow(x)$ and $\lpow(x)$ in $R[x]$. The converse of implication \texttt{(m)} will be denoted by \texttt{(m)'}.
\newlength{\coli}
\setlength{\coli}{\widthof{\,$\lpow(x)\subseteq\pow(x)$\,}}
\newlength{\colii}
\setlength{\colii}{\widthof{\,and indecomposable\,}}
\newlength{\coliii}
\setlength{\coliii}{\widthof{\,$R$ is indecomposable\,}}
\newlength{\alturai}
\settoheight{\alturai}{\heightof{\parbox{\colii}{\centering\strut
$R$ is reduced\\ and indecomposable\strut}}}
\newlength{\alturaii}
\settoheight{\alturaii}{\heightof{\parbox{\coliii}{\centering\strut $x$ is irreducible\\ in $R[x]$\strut}}}
\newlength{\alturaiii}
\settoheight{\alturaiii}{\heightof{\parbox{\colii}{\centering\vphantom{\LARGE L} $x\in\lpow(x)$ and\\ $\lpow(x)\subseteq\pow(x)$\strut}}}
\tikzcdset{row sep/normal=1.4cm,column sep/normal=1.4cm}
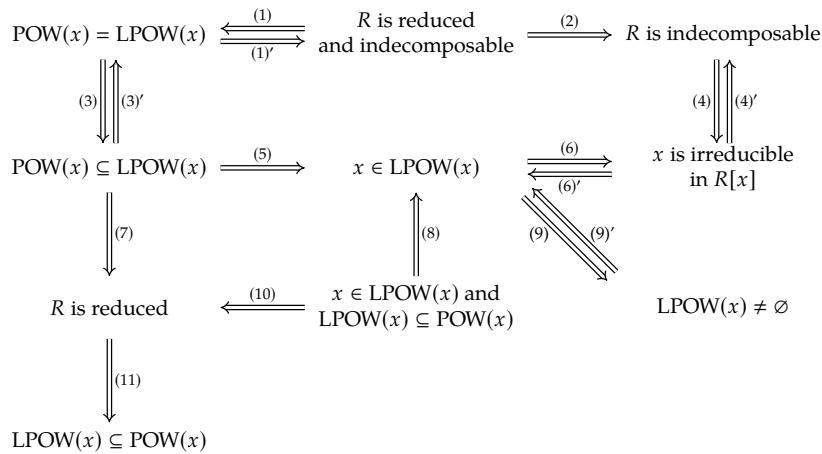
\[\scalebox{0.95}{%
\begin{tikzcd}[ampersand replacement=\&]
\parbox[c][\alturai][c]{\coli}{\centering $\pow(x)=\lpow(x)$} \arrow[r,Leftarrow,shift left=1.2,"\texttt{(1)}"] \arrow[r,Rightarrow,shift right=1.2,"\texttt{(1)'}"'] \arrow[d,Rightarrow,shift right=1.2,"\texttt{(3)}"'] \arrow[d,Leftarrow,shift left=1.2,"\texttt{(3)'}"] \& \parbox[c][\alturai][c]{\colii}{\centering $R$ is reduced\\ and indecomposable} \arrow[r,Rightarrow,"\texttt{(2)}"] \& \parbox[c][\alturai][c]{\coliii}{\centering $R$ is indecomposable} \arrow[d,Rightarrow,shift right=1.2,"\texttt{(4)}"'] \arrow[d,Leftarrow,shift left=1.2,"\texttt{(4)}'"] \\
\parbox[c][\alturaii][c]{\coli}{\centering $\pow(x)\subseteq\lpow(x)$} \arrow[r,Rightarrow,"\texttt{(5)}"] \arrow[d,Rightarrow,"\texttt{(7)}"] \& \parbox[c][\alturaii][c]{\colii}{\centering $x\in\lpow(x)$} \arrow[r,Rightarrow,shift left=1.2,"\texttt{(6)}"] \arrow[r,Leftarrow,shift right=1.2,"\texttt{(6)'}"'] \arrow[dr,phantom,Leftrightarrow,start anchor={south east},end anchor={north west},"{\mbox{\SMALL\texttt{(9)}}}"' left=3mm,"{\mbox{\SMALL\texttt{(9)'}}}" right=3mm] \arrow[dr,Rightarrow,shift right=1.2,start anchor={south east},end anchor={north west}] \arrow[dr,Leftarrow,shift left=1.2,start anchor={south east},end anchor={north west}] \& \parbox[c][\alturaii][c]{\coliii}{\centering $x$ is irreducible\\ in $R[x]$}\\
\parbox[c][\alturaiii][c]{\coli}{\centering $R$ is reduced} \arrow[r,Leftarrow,"\texttt{(10)}"] \arrow[d,Rightarrow,"\texttt{(11)}"] \& \parbox[c][\alturaiii][c]{\colii}{\centering $x\in\lpow(x)$ and\\ $\lpow(x)\subseteq\pow(x)$} \arrow[u,Rightarrow,"\texttt{(8)}"'] \& \parbox[c][\alturaiii][c]{\coliii}{\centering $\lpow(x)\neq\varnothing$} \\
\makebox[\coli]{\centering $\lpow(x)\subseteq\pow(x)$} \& \&
\end{tikzcd}
}\]
{\small\centering
\bgroup\tabcolsep=0\tabcolsep
\renewcommand{\arraystretch}{1.3}
\begin{tabular}{rlp{13.5cm}}
\texttt{(1)}&\hspace{5mm} & \Cref{lpowx=powx-b}.\\

\texttt{(1)}&\texttt{'} & \texttt{(3)+(7)}, together with \texttt{(3)+(5)+(6)+(4)'}.\\

\texttt{(2)}& & Obvious.\\

$\neg$\texttt{(2)}&\texttt{'} & Counterexample: $R=\Z/4\Z$.\\

\texttt{(3)}& & Obvious.\\

\texttt{(3)}&\texttt{'} & \texttt{(7)+(11)}.\\

\texttt{(4)},\texttt{(4)}&\texttt{'} & \Cref{bivalente}.\\

\texttt{(5)}& & Obvious.\\

$\neg$\texttt{(5)}&\texttt{'} & \texttt{(4)+(6)'+(5)'+(3)'+(1)'} imply \texttt{(2)'}, which is false. But \texttt{(4)},\texttt{(6)'},\texttt{(3)'} and \texttt{(1)'} are true.\\

\texttt{(6)}& & \Cref{lpowx=powximplies}.\\

\texttt{(6)}&\texttt{'} & \Cref{lpowgeneral-c}.\\

\texttt{(7)}& & \Cref{lpowx=powximplies}.\\
%
$\neg$\texttt{(7)}&\texttt{'} & \texttt{(7)'+(3)'+(1)'} is false (counterexample: $R=\Z\times\Z$). But \texttt{(3)'} and \texttt{(1)'} are true.\\

\texttt{(8)}&\hspace{5mm} & Obvious.\\

$\neg$\texttt{(8)}&\texttt{'} & \texttt{(4)+(6)'+(8)'+(10)} imply \texttt{(2)'}, which is false. But \texttt{(4)},\texttt{(6)'} and \texttt{(10)} are true.\\

\texttt{(9)},\texttt{(9)}&\texttt{'} & \Cref{lpowfurther-b}.\\

\texttt{(10)}& & \Cref{lpowx=powximplies}.\\

$\neg$\texttt{(10)}&\texttt{'} & \texttt{(10)'+(8)+(6)+(4)'} is false (counterexample: $R=\Z\times\Z$). But \texttt{(8)},\texttt{(6)} and \texttt{(4)'} are true.\\

\texttt{(11)}& & \Cref{lpowx=powx-a}.\\

$\neg$\texttt{(11)}&\texttt{'} & Counterexample: any decomposable nonreduced ring $R$, such as $R=\Z\times(\Z/4\Z)$: for $R$ decomposable implies $\lpow(x)=\varnothing$, by \texttt{(9)'+(6)+(4)'}, and so we trivially have $\lpow(x)\subseteq\pow(x)$.

\end{tabular}
\egroup
}


\end{document}